\documentclass[reqno,twoside,11pt]{amsart}
\usepackage{amsmath, amsfonts, amssymb, esint, amsthm, nicefrac, bbm,
  dsfont, enumitem, comment}
\usepackage{epsfig, graphicx,xcolor}
\newcommand{\dx}{~\mathrm{d}x}
\newcommand{\sym}{\mathrm{sym}}
\newcommand{\curl}{\mathrm{curl}}
\newcommand{\Id}{\mathrm{Id}}

\newcommand{\R}{\mathbb{R}}
\def\endproof{\hspace*{\fill}\mbox{\ \rule{.1in}{.1in}}\medskip }
\newcommand*{\dbar}[1]{\bar{\bar{#1}}}
\newcommand*{\tbar}[1]{\bar{\dbar{#1}}}

\numberwithin{equation}{section}
\theoremstyle{plain}

\newtheorem*{theorem*}{Theorem}
\newtheorem{theorem}{Theorem}[section]
\newtheorem{lemma}[theorem]{Lemma}
\newtheorem{corollary}[theorem]{Corollary}

\theoremstyle{definition}
\newtheorem{remark}[theorem]{Remark}

\begin{document}
\title [Full flexibility of Poznyak's theorem]
{Full flexibility of isometric immersions \\ of metrics with low H\"older
  regularity \\  in Poznyak theorem's dimension}
\author{Marta Lewicka}
\address{M.L.: University of Pittsburgh, Department of Mathematics, 
139 University Place, Pittsburgh, PA 15260}
\email{lewicka@pitt.edu} 

\date{\today}
\thanks{M.L. was partially supported by NSF grant DMS-2006439.
AMS classification: 53C42, 53A35}

\begin{abstract}
A classical result by Poznyak asserts that any smooth $2$-dimensional
Riemannian metric $g$, posed on the closure of a
simply connected domain $\omega\subset\R^2$, has
a smooth isometric immersion into $\R^4$. Using techniques of convex
integration, we prove that for any $2$-dimensional
$g\in\mathcal{C}^{r,\beta}$, an isometric immersion of regularity 
$\mathcal{C}^{1,\alpha}(\bar\omega,\R^4)$ for any
$\alpha<\min\{\frac{r+\beta}{2},1\}$, may be found arbitrarily close
to any short immersion.
The fact that this result's regularity reaches $\mathcal{C}^{1,1-}$ for $g\in
\mathcal{C}^2$, which is referred to as ``full flexibility'',
should be contrasted with: (i)  the regularity 
$\mathcal{C}^{1,1/3-}$ achieved in \cite{CHI2} for isometric
immersions into $\R^{3}$ and the lack of flexibility (rigidity)
of such isometric immersions with regularity
$\mathcal{C}^{1, 2/3+}$ proved in \cite{Bor1}-\cite{Bor5},
\cite{CDS}; (ii) the regularity $\mathcal{C}^{1,1-}$
obtained in \cite{Kallen} for isometric immersions into higher
codimensional space $\R^{8}$; and (iii) the regularity
$\mathcal{C}^{1,\frac{1}{1+d(d+1)/k}-}$ achieved in 
\cite{lew_conv} in the general case of $d$-dimensional metrics and
$(d+k)$-dimensional immersions for the closely related Monge-Amp\`ere system.
\end{abstract}

\maketitle
\tableofcontents

\section{Introduction}

As mentioned in Nash's fundamental work ``The imbedding problem for
Riemannian manifolds'' \cite{Nash1}:
``apparently rigidity disappears completely when the imbedding space
has enough dimensions''. The purpose of
this paper is to show that for $2$-dimensional metrics, 
rigidity indeed disappears completely, already in $\R^4$. More precisely, we prove:

\begin{theorem}\label{thm_main}
Let $g\in \mathcal{C}^{r,\beta}(\bar\omega, \R^{2\times 2}_{\sym,>})$  be
a Riemannian metric, defined on the closure of an open set $\omega\subset\R^2$
diffeomorphic to $B_1$, for some regularity exponents $0<r+\beta\leq 2$ 
Then, for every immersion $u\in \mathcal{C}^1(\bar\omega,\R^4)$ satisfying:
$$ (\nabla u)^T\nabla u < g \quad \mbox{ in }\bar\omega,$$
for every $\epsilon>0$, and for every regularity exponent $\alpha$ in:
\begin{equation*}
0<\alpha<\min\Big\{\frac{r+\beta}{2}, 1\Big\},
\end{equation*}
there exists an immersion $\bar u\in \mathcal{C}^{1,\alpha}(\bar\omega,\R^4)$ such that:
$$ \|\bar u - u\|_0\leq \epsilon \quad\mbox{ and } \quad 
(\nabla \bar u)^T\nabla \bar u = g \quad \mbox{ in }\bar\omega.$$
\end{theorem}

\medskip

To position our result in regards to other known works, we
make the following observations:
\begin{itemize}
\item[-] The metric regularity $\mathcal{C}^2$ is critical for proving
  existence of isometric immersions, with the same regularity \cite{Nash1, Jaco, Gro_book, Gun}; 
since the assumed shortness condition can be easily achieved by taking
an arbitrary immersion $u$ (namely, a mapping that satisfies $(\nabla
u)^T\nabla u > 0$ in $\bar\omega$) and scaling it by a small positive constant, our result
implies that $2$-dimensional $\mathcal{C}^2$ metrics
always have an isometric immersion of any regularity up to
$\mathcal{C}^{1,1}$, i.e. just below
$\mathcal{C}^2$, already in dimension $n=4$. 
\item[-] As is known in various contexts \cite{CV, Her, HarNir, Pogo2, Bor5, CDS}, rigidity
does not disappear for isometric immersions of $2$-dimensional metrics
into $\R^3$; we show that no rigidity statement is possible for isometric immersions into $\R^4$.  
\item[-] It is known \cite{CHI2} that any short immersion of a $2$-dimensional
$\mathcal{C}^2$ metric $g$ can be approximated by $\mathcal{C}^{1,\alpha}$ isometric
immersions into $\R^3$, for any $\alpha<1/3$; we assert the same statement up
to regularity $\mathcal{C}^{1,1}$ already at the next target
dimension $n=4$. 
\item[-] Dimension $n=4$ is the target dimension in
Poznyak's theorem \cite{Poz}: the disk $B_1\subset\R^2$ with arbitrary
$\mathcal{C}^\infty$ metric can be isometrically
$\mathcal{C}^\infty$-immersed into $\R^4$. 
On the other hand \cite{Kallen}, for $n$ large enough there
exists a $\mathcal{C}^{1,\alpha}$ isometric immersion into $\R^n$, of any
$g\in\mathcal{C}^{r,\beta}$ where $r+\beta\leq 2$, with any
$\alpha<\frac{r+\beta}{2}$. We reach the same statements in 
dimension $4$, in the context of flexibility, rather than the mere
existence of a single isometric immersion. 
\item[-] The same result as ours has been as proved
\cite{in_lew2} for the Monge-Amp\`ere system; now we treat the fully nonlinear
case of (\ref{II_in}).
\end{itemize}
These observations are made precise in the historical
account of the existence, flexibility and rigidity of isometric
immersions in subsections below.

\medskip

\subsection{Classical existence and flexibility results} 
The question whether a Riemannian manifold of dimension $d$ may be
isometrically immersed in a Euclidean space of 
dimension $n$, is, in local coordinates, equivalent to solving
the first order system of partial differential equations:
\begin{equation}\label{II_in}
\begin{split}
& (\nabla u)^T\nabla u = g \quad \text{ in } \; \omega \subset \R^d,\\
& \, \mbox{for }\; u:\omega\to\R^n,
\end{split}
\end{equation}
where $g:\omega\to\R^{d\times d}_{\sym,>}$ is a given
first fundamental form i.e. the Riemannian metric of the manifold
written in these coordinates. The exact answer to this question and to the
question on the minimal dimension $n$, depend on $g$ and on how smooth $u$
is required to be. In the analytic case, the Janet-Cartan-Burstin theorem
\cite{Janet, Cartan, Bur} gives the affirmative
resolution of Schlaefli's conjecture \cite{Scha}, stating that any
$d$-dimensional analytic $g$ has a local analytic
iso\-met\-ric immersion $u$ into $\R^{s_d}$. Here, $s_d$ is the  Janet
dimension, corresponding to the number of equations in (\ref{II_in})
and the number of independent entries in $d\times d$ symmetric matrices:
\begin{equation}\label{SD}
s_d \doteq \frac{d(d+1)}{2}.
\end{equation}
In this theorem $n=s_d$ cannot be replaced by
$s_d-1$. For $d=2$, we have $s_d=3$. The crucial part of the
Janet-Burstin proof consists in showing that a local analytic
isometric immersion that
is ``free'', always exists, in the minimal dimension 
$d+s_d$, the freeness of an
immersion meaning that $d$ vectors of its first derivatives and $s_d$ of second
derivatives are linearly independent at each $x\in \omega$. Invoking this 
result at $d-1$ is then supplemented by extending the hence produced
analytic immersion by one extra
dimension that is missing from $(d-1)+s_{d-1} = s_d-1$ to $s_d$, via an
application of the Cauchy-Kowalewsky theorem.

\medskip

In the non-analytic case and for general metrics of regularity at least
$\mathcal{C}^2$, an isometric immersion into any
$\R^n$ cannot be of regularity better than the regularity of the
metric. This fact was first noticed by Hartman and Wintner
\cite{HarW}: given a $2$-dimensional Riemannian metric
$g\in\mathcal{C}^{r,\beta}$ with $r\geq 2$
and $\beta\in [0,1]$, its Gaussian curvature $\kappa$ can be computed from the second
fundamental form of an immersion $u\in\mathcal{C}^{l}$, which gives
$\kappa\in\mathcal{C}^{l-2}$ (for $l$ not necessarily an integer). 
On the other hand, $\kappa$ is an intrinsic quantity and as such, can be
computed directly from $g$, whereas $\kappa\in\mathcal{C}^{r-2,\beta}$. One
can explicitly construct a metric whose $\kappa$ is no smoother than
$\mathcal{C}^{r-2, \beta}$, hence in general there must be: $l\leq r+\beta$. In
dimension $d>2$, the same argument applies to sectional
curvatures and can be found in \cite{Jaco}.
The fact that the regularity exponent $l = r+\beta$ of an isometric
immersion can be achieved, is based on the following ideas of Nash.

\medskip

For the integer regularity $r>2$, the Nash theorem \cite{Nash1} states
that every $\mathcal{C}^r$ metric has a $\mathcal{C}^r$ isometric
immersion into some $\R^n$. A partition of unity argument reduces this theorem to 
the local case in (\ref{II_in}), in which $n$ can be taken as
$3s_d+4d$, the dimension later decreased by Gromov \cite{Gro_book} to
$s_d+2d+3$, and by G\"unter (with a different proof) \cite{Gun} to
$\max\{s_d+2d, s_d+d+5\}$.
For $d=2$, both these numbers equal $10$.  It is worth noting that the discussed dimensions
do not depend on the differentiability exponent $r$. 
The H\"older regular case has been studied by Jacobowitz, who proved
\cite{Jaco} that every $\mathcal{C}^{r,\beta}$ Riemannian metric with
$r\geq 2$ and $\beta\in (0,1]$ has a $\mathcal{C}^{r,\beta}$ isometric
immersion into some $\R^n$. Combining Nash's and Jacobowitz's results
yields immersability of any $\mathcal{C}^l$ metric %with regularity $\mathcal{C}^l$,
where $l>2$ is not necessarily an integer, with the immersion of the
same regularity. The key argument of their proofs relies on an
implicit function theorem (later generalized to what
is called the Nash-Moser implicit function theorem)
showing that for a large class of analytic
metrics, any $\mathcal{C}^l$ metric in their vicinity (depending on
$l$), possesses a $\mathcal{C}^l$ isometric immersion. Such analytic
metrics are shown to be dense in $\mathcal{C}^0$, due to the
earlier construction by Nash producing $\mathcal{C}^1$ isometries,
that we now describe. We note in passing that for $r\geq 2$, a $\mathcal{C}^r$ Riemannian
manifold of dimension $d$ has a local $\mathcal{C}^r$ isometric
immersion into $\R^{d}$ if an only if its Riemann curvature tensor
vanishes at each $x$. 

\medskip

In case of low immersion regularity
$\mathcal{C}^1$, the Nash theorem in \cite{Nash2} states that 
if a $\mathcal{C}^0$ Riemannian manifold of dimension $d$ has a short immersion into
$\R^{d+2}$, then it has a nearby $\mathcal{C}^1$ isometric immersion (into $\R^{d+2}$). In
coordinates, the shortness of $u\in\mathcal{C}^1(\bar
\omega, \R^n)$ is expressed by:
$$g - (\nabla u)^T \nabla u > 0 \quad \text{ in } \; \bar \omega,$$ 
i.e. by requiring that the matrix in left hand side be strictly positive definite for
every $x\in\bar \omega$. This condition can be achieved by scaling an
arbitrary immersion by  a small positive
number. Nash's theorem is proved via iterative modifications of an initial short immersion
$u$ by a cascade of highly oscillatory perturbations, with sum of
their amplitudes small, and with the property that each of them
replaces the specific (rank-one) portion of 
the defect:
\begin{equation}\label{def_def}
\mathcal{D}(g,u)\doteq g - (\nabla u)^T \nabla u, 
\end{equation}
by a much smaller (in the $\mathcal{C}^0$ norm) contribution. After
decreasing $\mathcal D$ in this manner, the new defect is
computed and the construction is repeated in, what is nowadays
referred to as, the Nash-Kuiper iteration scheme.
The nearby (to the given short immersion $u$) immersion $\tilde u$ is
obtained in the limit of such modifications.
The perturbations, called Nash's spirals, each of them
constituting a single Step in this convex integration algorithm,
utilize $2$ codimensions, and we exhibit a version of this
construction in Lemma \ref{lem_step1}, later used in the proof of
Theorem \ref{lem_stage0}. At the end of \cite{Nash2}, Nash conjectured
that the condition $n\geq d+2$ needed to achieve the modifications, can be weakened
to $n\geq d+1$. A  proof of this assertion was subsequently provided by Kuiper \cite{Kuiper}
via another perturbation definition, which indeed utilizes only one
codimension. We exhibit a version of such Kuiper's
corrugation  in Lemma \ref{lem_step2}, constituting a convex integration Step
in the proof of Theorem \ref{thm_main}. To reassume, the
Nash-Kuiper theorem demonstrates that any short immersion
$u\in\mathcal{C}^{1}(\bar\omega, \R^{d+1})$ can be uniformly
approximated, with arbitrary precision, 
by an isometric immersion $\bar u\in \mathcal{C}^{1}(\bar\omega, \R^{d+1})$, for
any continuous, $d$-dimensional $g$, yielding the 
abundance, or flexibility, of $\mathcal{C}^1$ solutions to (\ref{II_in}) with $n=d+1$.

\medskip

\subsection{The dimension $d=2$ and rigidity results} 

For $d=2$ we have $s_d=3$, so an analytic $2$-dimensional Riemannian
metric always has a local analytic isometric immersion into $\R^3$. Further, the Gromov
and G\"unter exponents equal $10$, so any $\mathcal{C}^r$ metric possesses a
$\mathcal{C}^r$ isometric immersion into $\R^{10}$, for $r>2$. One way
of reading our Theorem \ref{thm_main} is that, at the critical
regularity $r=2$, a $\mathcal{C}^2$ metric
always has an isometric immersion of any regularity just below
$\mathcal{C}^2$ (i.e., $\mathcal{C}^{1,\alpha}$ for any $\alpha<1$),
already in dimension $n=4$. This is precisely the target dimension in
Poznyak's theorem: as shown in \cite{Poz}, the disk $B_1\subset\R^2$ with arbitrary
$\mathcal{C}^\infty$ metric can be isometrically
$\mathcal{C}^\infty$-immersed in $\R^4$. For completeness, we present
a proof of this result in the Appendix section \ref{sec_appendix}.
In fact, the dimension $n=4$ cannot be lowered for this general
statement, because of an example in \cite{Poz} of a smooth metric
on a disk, for which there is no $\mathcal{C}^2$ isometric immersion
into $\R^3$. In \cite{GroRokh} a more striking example has been put forward, of an
analytic metric of positive Gaussian curvature that is not induced by any
$\mathcal{C}^2$ immersion into $\R^3$ and keeps this property under
any $\mathcal{C}^2$-small perturbation of the metric. We also note that 
$s_2+2 =5$, so below this dimension, a free immersion of any
$2$-dimensional Riemannian metric cannot exist. 

\medskip

The target dimension $n=3$ is different. Firstly, according to the classical
rigidity theorems due to Cohn-Vossen \cite{CV} and Herglotz
\cite{Her}, in the resolution of the Weyl problem \cite{Weyl}, a
$\mathcal{C}^2$ isometric immersion of $\mathbb{S}^2$ into $\R^3$ must
be a rigid motion, i.e. a composition
of a rotation and a translation. Second, any $\mathcal{C}^2$ surface
with positive Gauss's curvature must be locally convex; for
generalizations of this statement we refer to \cite{HarNir}, and 
to \cite[Chapter II]{Pogo1}, \cite[Chapter IX]{Pogo2}
where the requirement of the $\mathcal{C}^2$ regularity is
replaced by the requirement that the immersion is $\mathcal{C}^1$ 
and that the measure on $\mathbb{S}^2$ induced by the Gauss map has
bounded variation.
%and that the image of the Gauss map has measure zero in $\mathbb{S}^2$. 
Third, using geometric arguments, the same rigidity statement has been
proved by Borisov in a series of papers \cite{Bor1, Bor2, Bor3, Bor4, Bor5}, for
$\mathcal{C}^{1,\alpha}$ isometric immersions at $\alpha>2/3$, of
$\mathcal{C}^2$ metrics, with a simpler
analytic proof of this result obtained in \cite{CDS}.
In particular, we see that $2/3$ is an upper bound
on the range of H\"older exponents that can be reached using convex
integration for isometric immersions of $2$-dimensional metrics into
$\R^3$. To the contrary, our Theorem \ref{thm_main} shows that no such
rigidity statement is possible for the isometric immersions already into $\R^4$.

\medskip

\subsection{Recent results on flexibility} 
For the codimension-one case, i.e. when $n=d+1$ in (\ref{II_in}), it
already transpires from the Nash-Kuiper construction, that flexibility
should hold not only in $\mathcal{C}^1$, but also in
$\mathcal{C}^{1,\alpha}$ provided that $\alpha>0$
is sufficiently small. A first precise study of this type is due to
 Borisov, who announced in \cite{Bor6} and provided a
proof in case of $d=2$ in \cite{Bor7}, that for any analytic $g$,
the Nash-Kuiper theorem extends to local isometric immersions with regularity
$\mathcal{C}^{1,\alpha}$ and $\alpha<\frac{1}{1+2s_d}$. For arbitrary
dimension $d$, this statement
was proved by Conti, De Lellis and Szekelyhidi
in \cite{CDS}. For the special case $d=2$, the hence derived flexibility
exponent $\frac{1}{1+2s_2} = 1/7$ was improved to $1/5$ in \cite{DIS1/5},
capitalizing on the conformal equivalence of $2$-dimensional Riemannian
metrics with the Euclidean metric, thus reducing the number of
primitive (rank-one) defects in the decomposition of
$\mathcal{D}$ in  (\ref{def_def}) from $s_2=3$ to $2$. 

\medskip

The reasoning behind all the so-far H\"older exponent
improvements in flexibility results, follows precisely this line: if one can construct a modification
of a short immersion $u$ to another immersion, as done in \cite{Nash2,
Kuiper}, to the effect that $\mathcal{D}$ decreases by a factor of $\sigma^S$, for any large
$\sigma$ and some power $S$, at the expense of having 
$\nabla^2u$ increase by a factor of $\sigma^J$, for some
power $J$, then the Nash-Kuiper iteration scheme (see our
Theorem \ref{thm_NK}) is capable of producing, in the limit, an isometric immersion 
$\tilde u\in\mathcal{C}^{1,\alpha}$ with any prescribed regularity in
the range:
\begin{equation}\label{SJbasic}
0< \alpha < \frac{1}{1+2J/S},
\end{equation}
further restricted only by the regularity of $g$ when below
$\mathcal{C}^2$ (see formula (\ref{ran_NK})). Now, a rule of thumb is
that $\nabla^2u$ increases 
by one power of $\sigma$ already at the single application of
Kuiper's corrugation Step, so in case of a single codimension
available, and since in general one requires $s_d$ of such
Steps to decrease the defect by one power of $\sigma$, corresponding
to $s_d$ rank-one components in $\mathcal{D}$, formula (\ref{SJbasic}) with
$J=s_d$ and $S=1$ implies the Borisov-Conti-De Lellis-Szekelyhidi
exponent $\frac{1}{1+2s_d}$. Similarly, for $d=2$ we get after a
conformal change of variable that $J=2$ and $S=1$,
yielding the exponent $1/5$ as in \cite{DIS1/5}. 
In a recent paper \cite{CHI2}, Cao, Hirch and Inauen utilized
an auxiliary construction which transfers exactly $d$ rank-one primitive
defects onto the remaining
$s_d-d$ ones, hence allowing for $J=s_d-d$ and $S=1$, and resulting in
flexibility up to $\frac{1}{1+2(s_d-d)} =
\frac{1}{1+d^2-d}$ H\"older regularity in the first derivatives. 
For $d=2$ this implies that any short immersion of a $2$-dimensional
$\mathcal{C}^2$ metric $g$ can be approximated by $\mathcal{C}^{1,\alpha}$ isometric
immersions into $\R^3$, for any $\alpha<1/3$. In
this context, our Theorem \ref{thm_main} asserts the same statement up
to regularity $\mathcal{C}^{1,1}$ already at the next target
dimension $n=4$.

\medskip

As it is clear from the proofs, allowing for a higher codimension
$n-d>1$ only increases the regularity of the approximating
isometric immersions. From our technical viewpoint, this is linked to 
the increase of $\nabla^2u$ by only one power of $\sigma$ despite
several applications of Kuiper's corrugations in multiple Steps, as long as these
are done in different codimensions. This implies the increase of $S$
relative to $J$ or, the decrease of the ratio $J/S$ and therefore the increase
of $\alpha$. Without paying attention to such detailed considerations,
but applying the convex integration construction with Nash's spirals
as Steps, K\"allen proved in \cite{Kallen} that for $n\geq 6(d+1)(s_d+1)+2d$, 
% large enough, 
there exists a $\mathcal{C}^{1,\alpha}$ isometric immersion of any
$g\in\mathcal{C}^{r,\beta}$ where $r+\beta<2$, for any
$\alpha<\frac{r+\beta}{2}$. Our proof of Theorem \ref{thm_main}
combines the insights from \cite{Kallen, CHI2, CDS}
with some new tricks, to reach exactly the K\"allen exponent already
in dimension $4$, in the context of flexibility rather than the mere
existence of a single isometric immersion. Some aspects of the analysis leading to our
result, was developed in the parallel context of the Monge-Amp\`ere system,
discussed below.

\medskip

\subsection{The Monge-Amp\`ere system} 
The Monge-Amp\`ere system was introduced in \cite{lew_conv} as a
higher-dimensional version of the classical Monge-Amp\`ere equation:
\begin{equation}\label{e:MA}
\begin{split}
& \det \nabla^2 v =f \quad \text{ in } \; \omega \subset \R^2,\\
& \, \mbox{for }\; v:\omega\to\R,
\end{split}
\end{equation}
in relation to its interpretation as the prescription, to the leading order terms, of the Gaussian
curvature of a shallow surface given as the graph of $v$.
In the same vein, prescription of the full Riemann curvature tensor
$[R_{ij, st}]_{i,j,s,t:1\ldots d}$ of a $d$-dimensional shallow
manifold, reads: %that is the graph of $v:(\omega\subset\R^d)\to\R^k$,  is
%reflected by: 
\begin{equation}\label{e:MAdk}
\begin{split}
& \mathfrak{Det}\, \nabla^2 v \doteq \big[\langle\partial^2_{is}v,\partial^2_{jt}v\rangle - 
\langle\partial^2_{it}v,\partial^2_{js}v\rangle\big]_{i,j,s,t:1\ldots d}
= F \quad \text{ in } \; \omega \subset \R^d,\\
& \mbox{for }\; v:\omega\to\R^k,
\end{split}
\end{equation}
where $F:\omega\to\R^{d^4}$ is a given field. 
Indeed, the Riemann curvatures of the family of immersions $u_\epsilon
:\omega\to\R^{d+k}$ parametrized by $\epsilon\to 0$ in
$u_\epsilon(x) = (x, \epsilon v(x))$, are calculated as: 
$$R_{ij,st}\big((\nabla u_\epsilon)^T\nabla u_\epsilon)\big) =
R_{ij,st}\big(\mbox{Id}_d + \epsilon^2(\nabla v)^T \nabla v\big)=
\epsilon^2(\mathfrak{Det}\,\nabla^2v)_{ij,st} + o(\epsilon^2),$$ 
similarly to the formula for Gauss's curvature $\kappa$ in case $d=2$, $k=1$:
$$\kappa((\nabla u_\epsilon)^T\nabla
u_\epsilon)=\kappa(\mbox{Id}_2+\epsilon^2 \nabla v\otimes\nabla v)=
\frac{\epsilon^2\det \nabla^2v}{(1+\epsilon^2|\nabla v|^2)^2}
=\epsilon^2\det\nabla^2 v + o(\epsilon^2).$$ 
Relying on the special structure of $\mathfrak{Det}\,\nabla^2$, we
call $v\in H^1(\omega,\R^k)$ a {weak solution} to (\ref{e:MAdk}),
if the following identity holds in the sense of distributions: 
\begin{equation}\label{e:MA_weak}
-\frac{1}{2}\mathfrak{C}^2((\nabla v)^T\nabla v) = F \quad \text{ in } \; \omega.
\end{equation}
Here, $\mathfrak{C}^2$ is a second-order differential operator, 
which in dimension $d=2$ reduces, up to symmetries,
to taking $\mbox{curl}\,\mbox{curl}$ of a $\R^{2\times
  2}_{\sym}$- valued matrix field, whereas for $k=1$, the identity
(\ref{e:MA_weak}) becomes exactly the weak formulation of (\ref{e:MA})
as studied in \cite{lewpak_MA}: 
$$-\frac12\,\curl\,\curl\left( \nabla v\otimes \nabla v\right) = f.$$
The expression in the left hand side appeared in \cite{Iwaniec},
called therein the {very weak Hessian}. It can be observed that
the system (\ref{e:MA_weak}), when posed on a
contractible $\omega\subset\R^d$, reduces to:
\begin{equation}\label{e:VK}
\begin{split}
& \frac12 (\nabla v)^T \nabla v + \sym\, \nabla w = A
\\ & \mbox{for }\; v:\omega\to\R^k, \quad w:\omega\to\R^d,
\end{split}
\end{equation} 
where $A:\omega\to \R^{d\times d}_\sym$ satisfies the
compatibility condition: $-\mathfrak{C}^2(A) = F$, viable 
under appropriate symmetry, algebraic and
differential identities assumptions on $F$. The system (\ref{e:VK}) is called the {
  von K\'arman system}, in relation to the von K\'arman stretching
content in the theory of elasticity, where for $d=2, k=1$ the fields
$v,w$ are interpreted, respectively, as the out-of-plane and in-plane
displacements of the midsurface $\omega$ of a thin elastic plate
\cite{lew_book}. 

\medskip

\subsection{Flexibility of the Monge-Amp\`ere system}

The system (\ref{e:VK}) encodes the agreement, to the leading order,
between the family of Riemannian metrics $g_\epsilon = \mathrm{Id}_d + 2\epsilon^2
A $ and the induced metrics of the augmented immersions $\bar
u_\epsilon(x) = (x + \epsilon^2 w(x) , \epsilon v(x))$, thus revealing a
close connection between (\ref{e:MAdk}) and (\ref{II_in}) at $n=d+k$. For this reason,
one expects flexibility results for (\ref{e:VK}) of the same type as
for (\ref{II_in}). Indeed, given any short displacement pair
$(v,w)\in\mathcal{C}^1(\bar\omega,\R^{d+k})$, where the shortness 
means positive definiteness of the  defect $\mathcal{D}$ in:
$$\mathcal{D}(A, v,w) \doteq A- \big(\frac{1}{2}(\nabla v)^T\nabla v +
\sym\nabla w\big) > 0 \quad \mbox{ in } \bar \omega,$$ 
and given any $\epsilon>0$, there exists a
displacement pair $(\bar v, \bar
w)\in\mathcal{C}^{1,\alpha}(\bar\omega,\R^{d+k})$ satisfying:
$$\|\bar v - v\|_0 + \|\bar w - w\|_0 \leq \epsilon\quad\mbox{
  and } \quad \mathcal{D}(A, \bar v, \bar w) = 0 \quad\mbox{ in }
\; \bar\omega.$$
The fields $(\tilde v,\tilde w)$ are obtained in the limit of a version
of the Nash-Kuiper iteration scheme with $\alpha$ in the range
specified in (\ref{SJbasic}), where $S$ is the decay rate of
$\mathcal{D}$ and $J$ is the blow-up
rate of $\nabla^2v$ resulting from a single convex integration
Stage construction. In \cite{lew_conv}, we proposed that a Stage consist of
the least common multiple $lcm(s_d, k)$ Steps, each of them applying
a version of Kuiper's corrugation in $v$ and a matching perturbation
in $w$, to the effect of replacing, as before, one rank-one primitive component of
the defect by higher order error terms. Since $\mathcal {D}$
contains $s_d$ such components, each time the counter of the
number of Steps goes over a multiple of $s_d$, the defect goes down
by one factor of $\sigma$, while each time this Step counter goes over a multiple of
$k$, the  $\nabla^2v$ goes up, also by one
factor of $\sigma$. Hence, after $lcm(s_d, k)$ of Steps, one reads the relative blow-up/decay
ratio $J/S = s_d/k$, and therefore the flexibility up to the exponent 
$\frac{1}{1+2s_d/k}$. This exponent coincides with $\frac{1}{1+2s_d}$
from \cite{CDS} for $k=1$, and for $d=2, k=1$  with the exponent
$1/7$ that has been previously obtained in \cite{lewpak_MA} for the
Monge-Amp\`ere equation. 

In \cite{lew_conv} we also exhibited an
alternative convex integration construction for (\ref{e:MAdk}), based
on superposing $s_d$ Nash's spirals in a single Step, applicable for $k\geq 2s_d$ codimensions.
We showed how the K\"allen iteration technique from \cite{Kallen} allows for the
cancellation of arbitrarily high order defects all at once, 
thus leading to $J=1$ with $S$ arbitrarily large, and ultimately yielding the flexibility exponent
$1$, always for $k\geq 2s_d$. 

\medskip

In the special case of $d=2$, the insight from
\cite{DIS1/5} was used to improve the exponent $1/7$ to $1/5$ for the equation (\ref{e:MA}),
and in \cite{lew_improved} to $\frac{1}{1+4/k}$ for the system (\ref{e:MAdk}) and
a general codimension $k$. In \cite{CHI}, the exponent $1/5$ 
was further increased to $1/3$ by introducing
another version of the basic $2$-dimensional defect decomposition with
sharper bounds that distinguish between differentiation in the slow and fast variables. That
idea preceded the construction in \cite{CHI2} in
which certain components of $\mathcal{D}$ are initially transferred on other
components, thus effectively making the spacial directions that they
correspond to in the entries of $\mathcal{D}$, fast- or slow-like. In 
\cite{lew_improved2}, the construction from \cite{CHI} and the
K\"allen-Nash spirals approach from \cite{lew_conv} led to the 
flexibility exponent $3/7$ at $k=2$, the exponent $7/15$ at $k=3$
(from the general formula $\frac{2^{k}-1}{2^{k+1}-1}$),
and the full flexibility with regularity of the
solutions to (\ref{e:VK}), up to $\mathcal{C}^{1,1}$, for $k\geq 4$. 

\medskip

In \cite{in_lew} we studied the Monge-Amp\`ere system at $d=2, k=3$
and proved its flexibility up to regularity $\mathcal{C}^{1,1-1/\sqrt{5}}$. 
That was the first result in which the obtained
H\"older exponent $1-\frac{1}{\sqrt{5}}$ was larger than $1/2$,
but not contained in a result of full flexibility up to
$\mathcal{C}^{1,1}$. The new technique in \cite{in_lew} was based on the following
observation. In all previous works, a Stage consisted of consecutive
modifications of $(v, w)$ by the oscillatory perturbations in
different codimension directions, while keeping track
of the magnitudes of the defect $\mathcal{D}$ and the partial
derivatives $\partial_{ij}v^q$. These Steps were carried out until all
of the components $\partial_{ij}v^q$ had the same order,
at which point we read the blowup rate $J$ of $\nabla^2v$ and the
corresponding decay rate $S$ of $\mathcal{D}$. Presently, we were able
to construct a perturbation cascade of Kuiper's
corrugations, relative to the progression of frequencies which never
allows for the same order in all $\partial_{ij}v^q$ at once. Since new
corrugations are built on the previous ones in a staggered manner,
then including larger and larger number of Steps in a single Stage (later
iterated by the Nash-Kuiper algorithm) keeps decreasing $J/S$. A
particular progression of frequencies, based on the Fibonacci
sequence, yielded the regularity exponent $1-\frac{1}{\sqrt{5}}$ in
\cite{in_lew}. The same construction, combined with the transfer of
the defect portions as in \cite{CHI2} and another progression of
frequencies, allowed in \cite{in_lew2} for concluding the full
flexibility up to $\mathcal{C}^{1,1}$ for the
Monge-Amp\`ere system in dimension $d=2$ and codimension $k=2$.
Our present paper achieves the same result as \cite{in_lew2}, now for the fully nonlinear
system (\ref{II_in}), at the expense of a more complex
construction that we describe below. 
%In the following subsection, we
%also gather the differences and similaritites of our and previous proofs.

\medskip

\subsection{The main technical contributions of this paper and an outline of proofs}
\label{stage_sketch}

Our main contribution is given in terms of estimates
gathered in the Stage theorem, whose iteration via
the Nash-Kuiper algorithm ultimately provides the proof of Theorem \ref{thm_main}: 

\begin{theorem}\label{thm_STA} {\textup{[STAGE]}}
Let $g\in \mathcal{C}^{r,\beta}(\bar\omega, \R^{2\times 2}_{\sym,>})$  be
defined on the closure of an open set $\omega\subset\R^2$
diffeomorphic to $B_1$, where the H\"older regularity exponents satisfy:
\begin{equation}\label{regu}
0<r+\beta\leq 2.
\end{equation}
Fix an immersion $\underline u\in\mathcal{C}^\infty(\bar\omega,\R^4)$
and a constant $\underline\gamma\ge 1$ such that:
\begin{equation}\label{imma}
\frac{1}{\underline\gamma}\Id_2\leq (\nabla \underline u)^T\nabla \underline u\leq
\underline\gamma\,\Id_2\quad\mbox{ in } \;\bar\omega.
\end{equation}
Fix two integers $N,K\geq 4$. Then, there exists
$\underline \delta\in (0,1)$ and $\underline\sigma>1$, depending 
on $\omega, g, \underline u,\underline \gamma, N, K$, such that the following holds.
Given any $u\in\mathcal{C}^2(\bar\omega,\R^4)$ and any $\delta, \mu, \sigma$ such that:
\begin{align*}
& \delta\leq \underline\delta,\qquad \mu\delta^{1/2}\geq 1,
\qquad\sigma\geq \underline\sigma,\qquad \sigma^{3N+3}\delta\leq 1,
\tag*{(\theequation)$_1$}\refstepcounter{equation} \label{Ass1} \\
& \|\nabla u-\nabla \underline u\|_0\leq \frac{\rho_{\underline\gamma}}{2},
\tag*{(\theequation)$_2$} \label{Ass2}\\
& \|\mathcal{D}(g-\delta H_0, u)\|_0\leq\frac{r_0}{4}\delta \quad\mbox{
and }\quad \|u\|_2\leq \delta^{1/2}\mu,
\tag*{(\theequation)$_3$} \label{Ass3}
\end{align*}
where $r_0, H_0$ are independent quantities in Lemma \ref{lem_dec_def}, 
and $\rho_{\underline\gamma}$ depending only on $\underline\gamma$ is
given in Lemma \ref{lem_propa}, there exists $\tilde u\in\mathcal{C}^2(\bar\omega,\R^4)$ satisfying:
\begin{align*}
& \|\tilde u-u\|_1\leq C\delta^{1/2},\qquad \|\tilde u\|_2\leq C\mu\delta^{1/2}\sigma^{2K+N},
\tag*{(\theequation)$_1$}\refstepcounter{equation} \label{Res1} \\
&\Big\|\mathcal{D}\Big(g-\frac{\delta}{\sigma^{NK}}H_0,\tilde u\Big)\Big\|_0\leq
  \frac{r_0}{5}\frac{\delta}{\sigma^{NK}} + \frac{\|g\|_{r,\beta}}{\mu^{r+\beta}},
\tag*{(\theequation)$_2$} \label{Res2}
\end{align*}
with constants $C$ depending only on $\omega, g, \underline u,\underline \gamma, N, K$, 
\end{theorem}

The above result yields the decay rate $S_{K,N} = KN$ of 
$\mathcal{D}$ and the blow-up rate $J_{K,N} = 2K+N$  of
$\nabla^2 u$, where the quotient of these rates becomes
arbitrarily small for large $K, N$:
\begin{equation*}
\lim_{K\to\infty}\lim_{N\to\infty} \frac{J_{K,N}}{S_{K,N}}=
\lim_{K\to\infty}\lim_{N\to\infty} \frac{2K+N}{KN} = 0.
\end{equation*}
The H\"older regularity of the limiting immersion deduced from iterating the
Stage, depends only on the aforementioned quotient and the regularity
of $g$, through the formula (\ref{ran_NK}), namely:

\begin{theorem}\label{thm_NK} {\textup{[NASH-KUIPER'S ITERATION]}}
Let $g\in \mathcal{C}^{r,\beta}(\bar\omega, \R^{2\times 2}_{\sym,>})$
and $\underbar u\in\mathcal{C}^\infty(\bar\omega, \R^4)$
be defined on the closure of an open set $\omega\subset\R^2$ diffeomorphic to $B_1$,
for some regularity exponents in (\ref{regu}) and the immersability
constant $\underline\gamma\geq \max\{1,\|g\|_0\}$ as in
(\ref{imma}). Assume:
\begin{equation}\label{asjad}
\mathcal{D}(g,\underline u) >0 \quad\mbox{ in }\;\bar\omega \quad
\mbox{ and } \quad \|\mathcal{D}(g,\underline u)\|_0\leq
  \Big(\frac{\rho_{2\underline\gamma}}{16N_0^{1/2}}\Big)^2,
\end{equation}
where $N_0$ is the independent constant in Lemma \ref{lem_met_deco},
and $\rho_{2\underline\gamma}$ depends only on $2\underline\gamma$ and
is given in Lemma \ref{lem_propa}.
Assume further that for some $S,J,p>0$ and for some parameters:
$$\underline \delta\in (0,1), \quad \underline\sigma>1,$$ 
depending on $\omega, g, \underline u, \underline\gamma, S, J, p$, the following holds:
\begin{equation}\label{STA}
\left[\qquad  \begin{minipage}{11cm}
Given any $u\in\mathcal{C}^2(\bar\omega,\R^4)$ and any $\delta, \mu, \sigma$ such that:
\begin{equation*}
\begin{split}
& \delta\leq \underline\delta,\qquad \mu\delta^{1/2}\geq 1,
\qquad\sigma\geq \underline\sigma,\qquad \sigma^p\delta^{1/2}\leq 1,\\ 
& \|\nabla u-\nabla \underline u\|_0 \leq \frac{\rho_{\underline\gamma}}{2},\\
& \|\mathcal{D}(g-\delta H_0, u)\|_0\leq\frac{r_0}{4}\delta \quad\mbox{
and }\quad \|u\|_2\leq \delta^{1/2}\mu,
\end{split}
\end{equation*}
with $r_0, H_0$ as in Lemma \ref{lem_dec_def} and
$\rho_{\underline\gamma}$ as in Lemma \ref{lem_propa}, 
there exists $\tilde u\in\mathcal{C}^2(\bar\omega,\R^4)$ satisfying:
\begin{equation*}
\begin{split}
& \|\tilde u-u\|_1\leq C\delta^{1/2},\qquad \|\tilde u\|_2\leq C\mu\delta^{1/2}\sigma^J,\\
&\Big\|\mathcal{D}\Big(g-\frac{\delta}{\sigma^S}H_0,\tilde u\Big)\Big\|_0\leq
  \frac{r_0}{5}\frac{\delta}{\sigma^S} + \frac{\|g\|_{r,\beta}}{\mu^{r+\beta}},
\end{split}
\end{equation*}
with constants $C$ depending only on $\omega, g, \underline u,
\underline\gamma, S, J, p$,
\end{minipage}\quad \right]
\end{equation}
Then, for every $\epsilon>0$ and every exponent $\alpha$ in the range:
\begin{equation}\label{ran_NK}
0<\alpha<\min\Big\{\frac{r+\beta}{2},\frac{1}{1+ 2J/S}\Big\},
\end{equation}
there exists an immersion $\bar u\in\mathcal{C}^{1,\alpha}(\bar\omega, \R^4)$ such that:
$$\|\bar u - \underline u\|_0\leq \epsilon\quad\mbox{ and }\quad
\mathcal{D}(g,\bar u)=0\quad\mbox{in }\; \bar\omega.$$
\end{theorem}

\bigskip

\noindent A proof of Theorem \ref{thm_NK} will be
given in section \ref{section_NK}. It relies on iterating the
assertion (\ref{STA}) (guaranteed by Theorem \ref{thm_STA}) with a particular progression 
of $\delta$, $\mu$ and $\sigma$. 
For the base of the induction, one necessitates the initial Stage to decrease the
positive definite defect $\mathcal{D}(g, \underline u)$ to a 
defect satisfying the first condition in \ref{Ass3}, tied with
specific scaling laws on the initial increase of $\nabla^2u$. A
self-contained proof of such initial Stage, requiring the second
condition in (\ref{asjad}), is presented in Theorem \ref{lem_stage0}, via a construction based on 
Nash's spirals Step in Lemma \ref{lem_step1}. The proof of Theorem \ref{thm_main}
results then in view of an additonal pre-Stage statement, given in Theorem \ref{lem_prestage},
that allows to decrease an arbitrarily large initial defect and assure (\ref{asjad}).

\begin{figure}[htbp]
\centering
\includegraphics[scale=0.6]{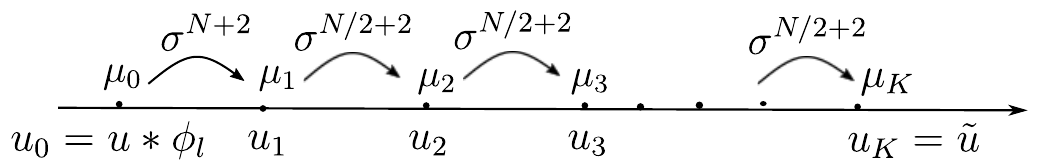}
\caption{{Progression of the principal frequencies in %(\ref{md_def}).
  the proof of Stage Theorem.}}
\label{fig_freq}
\end{figure}

\noindent We now sketch the proof of Theorem \ref{thm_STA}. 

\medskip

{\bf 1.} The new field $\bar u=u_K$ is constructed from $u$ as the
final, $K$-th immersion in the $K$-tuple $\{u_k\}_{k=1}^K$. 
The initial immersion $u_0$ is a mollified version of $u$ at the
given length scale of the order of $\mu_0=\mu$, and we set $\delta_0 = \delta$. This immersion
induces two unit normal vector fields $(E_{u_0}^1, E_{u_0}^2)$, perpendicular to
$\nabla u_0$ and to each other. Only $(E_{u_0}^1, E_{u_0}^2)$ are
defined independently, while each consequtive normal frame
$(E_{u_{k+1}}^1, E_{u_{k+1}}^2)$ is constructed from the previous
$(E_{u_{k}}^1, E_{u_{k}}^2)$, see Section \ref{sec_normals}. 
The intermediate defect bounds and frequencies
$\{\delta_k,\mu_k\}_{k=1}^K$ are set to satisfy:
\begin{equation}\label{intro0}
\mu_1 = \mu_0\sigma^{N+2}, \quad \mu_{k+1}=\mu_k\sigma^{N/2+2},\qquad
\delta_{k+1}=\frac{\delta_k}{\sigma^N}, 
\end{equation}
(see Figure \ref{fig_freq}) and the following inductive estimates are proved to hold, for
all $k=0\ldots K-1$ and all $m$ up to a sufficiently large (but finite) number:
\begin{equation}\label{intro1}
\begin{split}
& \|u_{k+1} - u_k\|_1\leq C\delta_{k}^{1/2},\\
& \|\nabla^{(m+1)} (u_{k+1} - u_k)\|_0 + \sum_{i=1}^2\|\nabla^{(m)} (E^i_{u_{k+1}} - E^i_{u_k})\|_0\leq 
C\delta_k^{1/2}\mu_{k+1}^m,\\
&\|\nabla^{(m)}\mathcal{D}(g_0-\delta_{k+1}H_0,
u_{k+1})\|_0\leq \frac{\delta_k}{\sigma^N}\mu_{k+1}^m = \delta_{k+1}\mu_{k+1}^m.
\end{split}
\end{equation}
Hence, at the counter $k=K$, the recursion (\ref{intro1}) yields: 
$$\delta_K=\frac{\delta}{\sigma^{KN}} \quad\mbox{ and }\quad 
\delta_{K-1}^{1/2}\mu_K = \delta^{1/2}\mu\sigma^{2K+N}$$
implying the claimed bounds \ref{Res1}, \ref{Res2} in view of (\ref{intro1}). The additional term in
the right hand side of \ref{Res2} is due to the fact that in
(\ref{intro1}) we calculate the defect with respect to the mollified version $g_0$ of $g$ rather than
$g$ itself, while the prefactor $r_0/5$ is secured by obtaining a
slightly larger decay $\delta_{k}/\sigma^{N+1/2}$ rather than
$\delta_k/\sigma^N$, and taking $\sigma$ large. The necessity of this
prefactor as well as the need for estimating the defect relative to a
fixed matrix $H_0$, follow from the validity of a decomposition of
a symmetric matrix
$H\in\R^{2\times 2}_\sym$ as a linear combination of three rank-one matrices
with nonnegative and uniformly bounded coefficients $\{a_i\in
(\frac{1}{2},\frac{3}{2})\}_{i=1}^3$:
\begin{equation}\label{intro2}
H=\sum_{i=1}^3 a_i^2 \eta_i\otimes \eta_i,
\end{equation}
not for all $H$, but in the vicinity of,
say, $H_0=\sum_{i=1}^3 \eta_i\otimes \eta_i$ (see Lemma \ref{lem_dec_def}). 

\medskip

{\bf 2.} The core of proof of (\ref{intro1})  relies on obtaining specific bounds on the
components of $\nabla^2u_k$: 
% for all $k=1\ldots K$ and all $m$ up to a sufficiently large (but finite) number:
\begin{equation}\label{intro3}
\begin{split}
& \|\nabla^{(m)} \langle \partial_{ij}u_k,E^1_{u_k}\rangle\|_0\leq
C\frac{\delta_{k-1}^{1/2}}{\sigma^{N/2}}\mu_{k}^{m+1}, \; \qquad 
\|\nabla^{(m)} \langle \partial_{11}u_k,E^2_{u_k}\rangle\|_0\leq
C\frac{\delta_{k-1}^{1/2}}{\sigma^{N}}\mu_{k}^{m+1},\\
& \|\nabla^{(m)} \langle \partial_{12}u_k,E^2_{u_k}\rangle\|_0\leq
C \frac{\delta_{k-1}^{1/2}}{\sigma^{N/2}}\mu_{k}^{m+1},\qquad
\|\nabla^{(m)} \langle \partial_{22}u_k,E^2_{u_k}\rangle\|_0\leq
C \delta_{k-1}^{1/2}\mu_{k}^{m+1}.
\end{split}
\end{equation}
The above bounds can be anticipated from the next definition. Namely, $u_{k+1}$ is
constructed from $u_k$ by passing through two extra intermediate immersion fields in:
\begin{equation}\label{intro4}
\begin{split}
& {U} = u_k + \frac{\Gamma(\lambda \langle x, \eta_1\rangle)}{\lambda} a_1 {E^1_{u_k}}
+ T_{u_k}\Big(\frac{\bar\Gamma(\lambda \langle x, \eta_1\rangle)}{\lambda}a_1^2\eta_1 + W\Big),\\
& {\bar U} = U + \frac{\Gamma(\kappa \langle x, \eta_2\rangle)}{\kappa} a_2{E^1_{U}}
+ T_{U}\Big(\frac{\bar\Gamma(\kappa \langle x,
  \eta_2\rangle)}{\kappa}a_2^2\eta_2 + \bar W\Big), \\
& {u_{k+1}} = \bar U + \frac{\Gamma(\mu_{k+1} \langle x,
  \eta_3\rangle)}{\mu_{k+1}} b {E^2_{\bar U}} + T_{\bar
  U}\Big(\frac{\bar\Gamma(\mu_{k+1} \langle x, \eta_3\rangle)}{\mu_{k+1}}b^2\eta_3 + \dbar W\Big),
\end{split}
\end{equation}
where the frequencies (see Figure \ref{fig_freq2}) are: 
$$\lambda = \mu_k\sigma, \quad \kappa=\mu_k\sigma^2.$$
The coefficients $\{a_i\}_{i=1}^3$, $b$ are obtained from decomposing the
defect via (\ref{gaza}), where $b$ contains $a_3$ and other components,
specified later. As a first attempt, think of applying (\ref{intro2}) to:
$$H=\mathcal{D}(g_0-\delta_{k+1}H_0, u_{k}).$$ 
Since that decomposition is linear, the induction
assumption implies $a_i, b\sim \delta_k^{1/2}$.
The definition (\ref{intro4}) is consistent with the
single Kuiper's corrugation construction (see Lemma \ref{lem_step2}),
in which the first term is always aligned with the chosen normal direction
(here, $E^1$ in $U$ and $\bar U$, and $E^2$ in $u_{k+1}$) to the current
immersion (here $u_k$, then $U$, then $\bar U$), 
while the second term carries a tangential component via
the basis of the tangent vectors to $u$ (the tangent frame) given in 
$T_u\doteq (\nabla u) ((\nabla u)^T\nabla u)^{-1}$ and a chosen
$2$-dimensional perturbation field $W$.

\medskip

To deduce (\ref{intro3}), one uses (\ref{intro4}) and the
induction assumption. Neglecting the tangential components and gathering
only the highest order powers in frequencies, we have:
\begin{equation*}
\begin{split}
& \|\nabla^{(m)} \langle \partial_{ij}u_{k+1},E^1_{u_{k+1}}\rangle\|_0\sim \|\nabla^{(m)}
\langle \partial_{ij}u_{k},E^1_{u_{k}}\rangle\|_0
+ \|a_1 \|_0\lambda^{m+1} + \|a_2 \|_0\kappa^{m+1} \\ & 
\qquad\qquad \qquad \qquad \qquad \;\, \leq C \frac{\delta_{k-1}^{1/2}}{\sigma^{N/2}}\mu_{k}^{m+1}
+ C \delta_k^{1/2}\frac{\mu_{k+1}^{m+1}}{\sigma^{N/2}}\mu_{k+1}^{m+1}
= C \frac{\delta_{k}^{1/2}}{\sigma^{N/2}}\mu_{k+1}^{m+1},\\
& \|\nabla^{(m)} \langle \partial_{22}u_{k+1},E^2_{u_{k+1}}\rangle \|_0\sim \|\nabla^{(m)}
\langle \partial_{22}u_{k},E^2_{u_{k}}\rangle \|_0
+ \|b \|_0\mu_{k+1}^{m+1}  \\ & 
\qquad\qquad \qquad \qquad \qquad \;\; \;\leq C \delta_{k-1}^{1/2}\mu_{k}^{m+1}
+ C{\delta_k^{1/2}}\mu_{k+1}^{m+1} \leq C {\delta_k^{1/2}}\mu_{k+1}^{m+1},
\end{split}
\end{equation*}
in view of the assumed increase rates: from $\delta_k$ to $\delta_{k+1}$, and from
$\mu_k$ to $\lambda$ to $\kappa$ to $\mu_{k+1}$. For the remaining
components we use that $\eta_2=e_2$, so that the bounds (\ref{intro3}) are closed in:
\begin{equation*}
\begin{split}
&\|\nabla^{(m)} \langle \partial_{11}u_{k+1},E^2_{u_{k+1}}\rangle\|_0\sim \|\nabla^{(m)}
\langle \partial_{11}u_{k},E^2_{u_{k}}\rangle\|_0 + \|b\|_0\mu_{k+1}^{m-1} \\ & 
\qquad\qquad \qquad \qquad \qquad \;\;
\leq C \frac{\delta_{k-1}^{1/2}}{\sigma^N}\mu_{k}^{m+1}
+ C {\delta_k^{1/2}}\mu_{k+1}^{m-1}
\sim C\frac{\delta_k^{1/2}}{\sigma^N}\mu_{k+1}^{m+1},\\
& \|\nabla^{(m)} \langle \partial_{12}u_{k+1},E^2_{u_{k+1}}\rangle\|_0\sim \|\nabla^{(m)}
\langle \partial_{12}u_{k},E^2_{u_{k}}\rangle\|_0 + \|b \|_0\mu_{k+1}^{m} 
\leq C\frac{\delta_k^{1/2}}{\sigma^{N/2}}\mu_{k+1}^{m+1}.
\end{split}
\end{equation*}

\medskip

{\bf 3.} We now trace the changes of the defect: from that corresponding to $u_k$,
to $U$ then to $\bar U$. The departing point in our
construction is the decomposition 
(\ref{intro2}), which needs to be administered under the K\"allen iteration technique (see
Theorem \ref{prop1}). This technique allows to decompose the defect
together with a specific portion of the future error, namely the non-oscillatory
portion which cannot be handled otherwise. A new error 
$\mathcal{F}$ is hence introduced, however it can be made arbitrarily
small. This is the key outcome of K\"allen's iteration:
\begin{equation}\label{gaza}
\begin{split}
& \mathcal{D}(g_0-\delta_{k+1}H_0, u_k) - \frac{1}{\lambda^2}\nabla a_{1}\otimes \nabla a_{1}
- \frac{1}{\kappa^2}\nabla a_{2} \otimes \nabla a_2 = \sum_{i=1}^3a_{i}^2\eta_i\otimes\eta_i 
+\mathcal{F}, \\
& \mbox{where } \quad \|\nabla^{m}\mathcal{F}\|_0\leq C\frac{\delta_k}{\sigma^{N}} \mu^m.
\end{split}
\end{equation}
Now, Kuiper's corrugation in $U$ in (\ref{intro4}), has the purpose
of cancelling both $a_1^2\eta_1\otimes \eta_1$ and $\frac{1}{\lambda^2}\nabla a_{1}\otimes \nabla a_{1}$
from $\mathcal{D}(g_0-\delta_{k+1}H_0, u_k)$ above, and it does so,
because (see Lemma \ref{lem_step2}):
\begin{equation}\label{intro5}
\begin{split}
& (\nabla U)^T\nabla U - (\nabla u_k)^T\nabla u_k = a_1^2\eta_1\otimes
\eta_1 + \frac{1}{\lambda^2}\nabla a_{1}\otimes \nabla a_{1}
\\ & +\frac{\Gamma(\lambda \langle x, \eta_1\rangle)^2-1}{\lambda^2} S_1 + 
  \frac{\Gamma\Gamma'(\lambda \langle x, \eta_1\rangle)}{\lambda}S_2 +
\frac{\Gamma(\lambda \langle x, \eta_1\rangle)}{\lambda}S_3 + \frac{\bar\Gamma(\lambda
  \langle x, \eta_1\rangle)}{\lambda}S_4 \\ & +\sym\nabla W + \mathcal{R},
\end{split}
\end{equation}
where $\mathcal{R}$ is a residual error term, arbitrarily
small in the sense given below, while the quantities $S_1, S_2, S_3, S_4$
define the leading order errors. The field $W$ is such
that it cancels all entries of these
principal four errors % in the right hand side of (\ref{intro5})
apart from their $e_2\otimes e_2$ components, at the expense of a much
smaller new error $\mathcal {G}$. This can be done precisely because 
the oscillation profiles $\Gamma^2-1$,
$\Gamma\Gamma'$, $\Gamma$ and $\bar\Gamma$ have mean zero on the
period (see Lemma \ref{lem_IBP1}). The main idea
is based on the decomposition: %put forward in \cite{CHI2}:
\begin{equation}\label{intro55}
\begin{split}
& \Gamma (\lambda x_1) H  = \Gamma(\lambda x_1) \sym ((H_{11},
2H_{12})\otimes e_1) + \Gamma(\lambda x_1) H_{22}e_2\otimes e_2
\\ & = \sym\nabla \Big(\frac{\Gamma_1(\lambda x_1)}{\lambda} (H_{11},
2H_{12})\Big) - \frac{\Gamma_1(\lambda x_1)}{\lambda}\sym\nabla
(H_{11}, 2H_{12})  + \Gamma(\lambda x_1) H_{22}e_2\otimes e_2,
\\ & \mbox{where } \Gamma_1'=\Gamma,
\end{split}
\end{equation}
that is iterated $N$ times, decreasing the second term in its right hand side to
the order $1/\lambda^N$. Of course, the primitives
$\Gamma_N$ given iteratively by the formula $(\frac{d}{dt})^{(N)}\Gamma_N=\Gamma$, remain
bounded provided that their initial oscillatory profile $\Gamma$ has mean zero on its period.
In conclusion, our first intermediate defect is composed of:
\begin{equation}\label{intro6}
\begin{split}
& \mathcal{D}(g-\delta_{k+1}H_0, U)  =\sum_{i=2}^3
  a_{i}^2\eta_i\otimes \eta_i + \frac{1}{\kappa^2}\nabla
  a_{2}\otimes\nabla a_2 + \mathcal{F} +   \mathcal{G} + \mathcal{R}- Ge_2 \otimes e_2,
\\ & \mbox{where:} \quad \|\nabla^{(m)}\mathcal{F}\|_0\sim
C \frac{\delta_k}{\sigma^N} \lambda^m \quad \mbox{ from K\"allen's iteration}\\
& \qquad\quad \;\; \; \|\nabla^{(m)}\mathcal{G}\|_0\sim
C\frac{\delta_k}{(\lambda/\mu_k)^{N}}\lambda^m 
\leq C\frac{\delta_k}{\sigma^{N}}\lambda^m \quad \mbox{ from oscillatory
  decomposition}\\
& \qquad\quad \;\; \, \|\nabla^{(m)}\mathcal{R}\|_0\leq C
\delta_k^{3/2}\sigma^{N/2}\lambda^m\leq C \frac{\delta_k}{\sigma^{N}}\lambda^m
\quad \mbox{ from the last assumption in \ref{Ass1}}.
\end{split}
\end{equation}
The bound on the derivatives of $\mathcal{G}$ is due to the fact that 
$S_1$, $S_2$, $S_3$, $S_4$ oscillate with frequency $\mu_k$ and $\lambda/\mu_k=\sigma$.
A similar reasoning applies to the second
intermediate immersion $\bar U$ in (\ref{intro4}), where the second
Kuiper corrugation is used to cancel $a_2^2\eta_2\otimes \eta_2$ 
and $\frac{1}{\kappa^2}\nabla a_{2}\otimes \nabla a_{2}$
from $\mathcal{D}(g_0-\delta_{k+1}H_0, U)$, 
while $\bar W$ is defined to cancel (see Lemma \ref{lem_IBP2}) all but the $e_2\otimes e_2$ entries of:
$$ \frac{\Gamma(\kappa \langle x,
    \eta_2\rangle)^2-1}{\kappa^2} \bar S_1 +
  \frac{\Gamma\Gamma'(\kappa \langle x, \eta_2\rangle)}{\kappa}\bar S_2 +
\frac{\Gamma(\kappa \langle x, \eta_2\rangle)}{\kappa}\bar S_3 + \frac{\bar\Gamma(\kappa
 \langle x, \eta_2\rangle)}{\kappa}\bar S_4,$$
derived as the four principal errors terms (see Lemma
\ref{lem_step2}) in $(\nabla \bar U)^T\nabla \bar U
- (\nabla U)^T\nabla U$, similarly as in (\ref{intro5}).
Hence, the second intermediate defect is composed of:
\begin{equation}\label{intro7}
\begin{split}
& \mathcal{D}(g-\delta_{k+1}H_0, \bar U)  =
 \big( a_3^2 - G-\bar G) \eta_3\otimes \eta_3 + \mathcal{E} +
 \bar{\mathcal{G}} + \bar{\mathcal{R}}
\\ & \mbox{where:} \quad \|\nabla^{(m)}\mathcal{E}\|_0\sim
C \frac{\delta_k}{\sigma^N} \lambda^m \quad \mbox{ from (\ref{intro6})}\\
& \qquad\quad \;\; \; \|\nabla^{(m)}\bar{\mathcal{G}}\|_0\sim
C\frac{\delta_k}{(\kappa/\lambda)^{N}}\kappa^m  \leq 
C\frac{\delta_k}{\sigma^{N}}\kappa^m \quad \mbox{ from oscillatory
  decomposition}\\
& \qquad\quad \;\; \, \|\nabla^{(m)}\bar{\mathcal{R}}\|_0\leq C
\delta_k^{3/2}\sigma^{N/2}\kappa^m\leq C \frac{\delta_k}{\sigma^{N}}\kappa^m 
\quad \mbox{ from the last assumption in \ref{Ass1}}.
\end{split}
\end{equation}
The bound on the derivatives of $\bar{\mathcal{G}}$ is due to the fact that 
$\bar S_1$, $\bar S_2$, $\bar S_3$, $\bar S_4$ oscillate with
frequency $\lambda$ and $\kappa/\lambda=\sigma$.
From the bounds on $G, \bar G$ one can correctly define the augmented amplitude:
$$b^2=a_{3}^2 - G - \bar  G$$ 
that obeys: $ b^2\sim \delta_k$ and $\|\nabla^{(m)} b^2\|_0\leq
C \frac{\delta_k}{\sigma}\kappa^m$.

\medskip

{\bf 4.} We now analyze the change of the defect from that
corresponding to the immersion $\bar U$ to the final $u_{k+1}$. 
The third Kuiper's corrugation in (\ref{intro4}) is used  to cancel
the term $b^2e_2\otimes e_2$ in $\mathcal{D}(g-\delta_{k+1}H_0, \bar U)$
%with frequency $\mu_{k+1}$, 
and note that it is the first time that we
use the second codimension normal field $E^2$.
The field  $\dbar W$ is defined by applying the oscillatory
decomposition in Lemma \ref{lem_IBP2}, to cancel all but the $e_1\otimes e_1$ entries of
the following principal errors in Lemma \ref{lem_step2}:
$$ \frac{\Gamma\Gamma'(\mu_{k+1}\langle x, \eta_3\rangle)}{\mu_{k+1}}\dbar S_2 +
\frac{\Gamma(\mu_{k+1} \langle x, \eta_3\rangle)}{\mu_{k+1}}\dbar S_3 + \frac{\bar\Gamma(\mu_{k+1}
\langle x, \eta_3\rangle)}{\mu_{k+1}}\dbar S_4,$$
with only two iterations of  (\ref{intro55}), rather than $N$ as
before. This is precisely why we necessitate $\mu_{k+1}/\kappa\sim
\sigma^{N/2}$  large, whereas $\lambda/\mu_k =\sigma$ and
$\kappa/\lambda=\sigma$ has sufficed. The contribution corresponding
to $\dbar S_1=\nabla b\otimes \nabla b$ is estimated directly in the
second bound below, due to $b$ and hence $\dbar S_1$
oscillating with frequency $\kappa$. It follows that:
\begin{equation}\label{intro8}
\begin{split}
& \mathcal{D}(g-\delta_{k+1}H_0, u_{k+1})  =
 \dbar G e_1\otimes e_1 + \mathcal{E} + \bar{\mathcal{E}} 
 + \frac{\Gamma(\mu_{k+1} \langle x, \eta_3\rangle)^2}{\mu_{k+1}^2}\nabla b\otimes\nabla b +
\dbar{\mathcal{R}}+ \dbar{\mathcal{G}},
\\ & \mbox{where:} \quad \|\nabla^{(m)}(\mathcal{E} + \bar{\mathcal{E}})\|_0\sim
C \frac{\delta_k}{\sigma^N} \kappa^m \quad \mbox{ from (\ref{intro6}), (\ref{intro7})},\\
& \qquad\quad \;\; \; \big\| \nabla^{(m)}\Big(
\frac{\Gamma(\mu_{k+1} \langle x, \eta_3\rangle)^2}{\mu_{k+1}^2}\nabla
b\otimes\nabla b\Big)\big\|_0\leq C
\frac{\delta_k}{(\mu_{k+1}/\kappa)^2}\mu_{k+1}^m \leq C\frac{\delta_k}{\sigma^{N}}\mu_{k+1}^m,
\\ & \qquad\quad \;\; \; \|\nabla^{(m)}\dbar{\mathcal{G}}\|_0\sim 
\frac{\delta_{k-1}}{(\mu_{k+1}/\kappa)^{3}} \mu_{k+1}^m =
\frac{\delta_k \sigma^{N/2}}{(\mu_{k+1}/\kappa)^{3}} \mu_{k+1}^m \leq 
C \frac{\delta_k}{\sigma^{N}}\mu_{k+1}^m
\\ & \qquad\qquad\qquad \qquad\qquad \qquad\qquad \qquad\quad 
\qquad\qquad \qquad
\mbox{from oscillatory decomposition},\\
& \qquad\quad \;\; \, \|\nabla^{(m)}\dbar{\mathcal{R}}\|_0\leq C
\delta_k^{3/2}\sigma^{N/2}\mu_{k+1}^m\leq C \frac{\delta_k}{\sigma^{N}}\mu_{k+1}^m 
\quad \, \mbox{by last assumption in \ref{Ass1}}.
\end{split}
\end{equation}
The third bound above follows in virtue of (\ref{intro3}), since the worst term in
$\dbar{\mathcal{G}}$ is:
$$\dbar{\mathcal{G}} \sim \frac{\Gamma(\mu_{k+1}\langle x,
  \eta_3\rangle)}{\mu_{k+1}^3}\nabla^2\dbar S
\quad \mbox{ where } \quad \dbar S \sim  b
(\nabla\bar U)^T\nabla E^2_{\bar U}. $$
In conclusion, the final assertion of the inductive bound (\ref{intro1}), namely:
$$\|\nabla^{(m)}\mathcal{D}(g_0-\delta_{k+1}H_0, u_{k+1})\|_0
\sim \frac{\delta_k}{\sigma^{N}}\mu_{k+1}^m,$$
follows in view of (\ref{intro8}), if we are able to check that: 
\begin{equation}\label{intro9}
\|\nabla^{(m)}\dbar G\|_0\sim
  \frac{\delta_k}{\sigma^{N}}\mu_{k+1}^m.
\end{equation} 
This indeed holds, and can be seen from the exact formulas on the $e_1\otimes
  e_1$ residual components in the decomposition parallel to that in
  (\ref{intro55}), given in Lemma \ref{lem_IBP3}:
\begin{equation*}
\begin{split}
& \dbar G \sim 
\sum_{i=0}^1 \frac{\Gamma_{i}(\mu_{k+1} \langle x, \eta_3\rangle
  )}{\mu_{k+1}^{i+1}} P_i(\dbar S_2 + \dbar S_3 + \dbar S_4)
\qquad \mbox{where} \quad \dbar S_2\sim b\,\sym(\nabla b\otimes \eta_3), 
\\ & \mbox{and } \quad
\dbar S_3\sim b\,\sym\big((\nabla \bar U)^T\nabla E^2_{\bar  U})\big),
\quad\dbar S_4 \sim \sym\big((\nabla\bar U)^T\nabla (b^2T_{\bar U}\eta_3)\big),\\
& \mbox{with} \quad P_0(H)=H_{11},\quad P_1(H)= \partial_1H_{12}.
\end{split}
\end{equation*}
In particular, only components $(\dbar{S}_i)_{11}$ and
$\partial_1(\dbar{S}_i)_{12}$ enter in the formula on $\dbar G$.
We prove (\ref{intro9}) using all the information in (\ref{intro3}) as
well as the relation between $\kappa$ and $\mu_{k+1}$.

\medskip

This ends the sketch of proof of Theorem \ref{thm_STA}. We close by a few extra points:
\begin{itemize} 
\item[-] The distinction between definitions of $\mu_1$ and
$\mu_{k}$ at $k>1$ in (\ref{intro0}) is needed for the proof of the inductive bounds on
$\langle \partial_{11}u_k, E^2_{u_k}\rangle$ in (\ref{intro3}), separately at the
induction base $k=1$ and then at the induction step. 
\item[-] Along the proof, we need
to keep track of the evolution of the orthonormal frame $(E^1,
E^2)$ and the tangent frame $T$ to the surfaces given by immersions $u_k, U, \bar U$, and also
to have all these immersions satisfy
the uniform condition \ref{Ass2}. %ensuring uniformity of bounds. 
\item[-] We provide bounds on any order of derivatives $\nabla^{(m)}$ of various fields, however
only a finite number of them is relevant (in the proofs we specify how
many, for completeness). Hence all constants
$C$ in our estimates are uniform, though depending on $N$ (which
refers to the decay rate of the defect passing from $u_k$ to
$u_{k+1}$), and $K$ (which stands for the number of triple Steps in a Stage). 
\item[-] Several error quantities are estimated by means
of higher powers of $\delta$ (like $\delta^{3/2}$ or $\delta^2$). All such
quantities are bounded by a quotient of $\delta$ and
the relative frequency $\sigma$ raised to a convenient power at hand, because of the last
assumption in \ref{Ass1}. 
\end{itemize}

\medskip

\begin{remark}
The same result as in Theorem \ref{thm_main}, remains valid in any target
space $\R^n$ replacing $\R^4$, for $n\geq 4$. Our
proofs only require existence of and the propagation bounds on two
normal vector fields, which are guaranteed by
having $n-2\geq 2$. The same statement as in Theorem
\ref{thm_STA} also holds and one concludes the final result by a
version of Theorem \ref{thm_NK}.
\end{remark}

%\medskip
%\subsection{Connection to other convex integration constructions}

\medskip

\subsection{Notation} By $\mathbb{R}^{k\times k}_{\sym}$ and
we denote the space of symmetric $k\times k$ matrices, and
$\mathbb{R}^{k\times k}_{\sym,>}$ is the cone of such matrices that
are additionally positive definite. The space of H\"older continuous vector fields
$\mathcal{C}^{m,\beta}(\bar\omega,\R^k)$ where $m\geq 0$, $\beta\in
[0,1]$ consists of restrictions of
all $v\in \mathcal{C}^{m,\beta}(\mathbb{R}^2,\R^k)$ to the closure
$\bar\omega$ of an open, bounded set $\omega\subset\R^2$. The
$\mathcal{C}^m(\bar\omega,\R^k)$ norm of such restriction is
denoted by $\|v\|_m$, while its H\"older norm in $\mathcal{C}^{m,
  \beta}(\bar\omega,\R^k)$ is $\|v\|_{m,\beta}$. Below we gather
other  notation that is recurring in our paper, specifying the first time
when it appears.

\smallskip

\begin{description}[leftmargin=15mm]
\item[$\quad\mathcal{D}(g,u)=g-(\nabla u)^T\nabla u$] the defect,
  i.e. the difference between the given metric and the immersion metric,
  see (\ref{def_def}),
\item[$\quad\underline u$] the referential smooth immersion with
  positive definite defect, prepared for the iteration of the Stage
  construction, see Theorem \ref{thm_NK}.
\item[$\quad\underline \gamma$] the referential immersion
  immersability constant after
 an application of  the Pre-Stage; all immersions in the inductive stages obey
  $1/{(2\underline\gamma)}\Id_2\leq (\nabla u)^T\nabla u\leq
  2\underline\gamma \Id_2$, see Lemma \ref{lem_stage0} and Theorem \ref{thm_STA}.
\item[$\quad \rho_{\underline\gamma}$] the bound on the difference of
  gradients assuring propagation of the normal frame, see Lemma \ref{lem_propa}.
\item[$\quad\eta_1=e_1, \eta_2=\frac{e_1+e_2}{\sqrt2}, \eta_3=e_2$] the
  unit vectors yielding the primitive defect decomposition, see Lemma \ref{lem_dec_def}.
\item[$\quad r_0$]  the radius assuring positivity of
  coefficients in the metric decomposition, see Lemma \ref{lem_dec_def}.
\item[$\quad H_0=\sum_{i=1}^3\eta_i\otimes \eta_i$] the referential
  (constant) metric whose all decomposition coefficients equal $1$,
  see Lemma \ref{lem_dec_def}.
\item[$\quad\{\bar a_i\}_{i=1}^3$] the linear projections in the
  decomposition of the defect, see Lemma \ref{lem_dec_def}. 
\item[$\quad N_0$] the local partition of unity constant in Lemma \ref{lem_met_deco}. 
\item[$\quad\Gamma, \bar\Gamma$] the leading oscillatory profiles in the
  given convex integration Step constructions, see Lemmas \ref{lem_step1}
  and \ref{lem_step2}.
\item[$\quad E^1_u, E^2_u$] the orthonormal pair of normal vectors
  fields (the normal frame) to
  the surface given by immersion $u$, see Lemma \ref{lem_step1} and \ref{lem_propa}.
\item[$\quad T_u=(\nabla u)((\nabla u)^T\nabla u)^{-1}$] the tangent frame
  to the surface $u(\bar\omega)$ given by the immersion $u$, see Lemma \ref{lem_Tu}. 
\item[$\quad\mathcal{R}, \mathcal{R}_1, \mathcal{R}_2$] the secondary error
  terms in the single Step constructions, see Lemmas \ref{lem_step1}, \ref{lem_step2}. 

\end{description}

\smallskip

\noindent By $C$ we denote a universal constant that may change from line to
line of the proof, where it depends on the specified parameters while being
independent from all the induction counters.

\subsection{Acknowledgement} The author was partially supported by the
NSF grant DMS-2407293. The author gratefully acknowledges support from the
Simons Center for Geometry and Physics, Stony Brook University as well
as the Simons-CRM program at McGill University.
% at which some of the research for this paper was performed. 

\section{The defect decomposition lemmas and two Step
  constructions}\label{sec_prepara1} 

This section contains some preliminary technical lemmas that serve as
the building blocks of the following proofs.
We first gather the convolution and commutator estimates from \cite{CDS}:

\begin{lemma}\label{lem_stima} 
Let $\phi\in\mathcal{C}_c^\infty(\R^2,\mathbb{R})$ be a standard
mollifier that is nonnegative, radially symmetric, supported on the
unit ball $B(0,1)\subset\R^2$ and such that $\int_{\mathbb{R}^2} \phi \dx = 1$. Denote: 
$$\phi_l (x) = \frac{1}{l^2}\phi(\frac{x}{l})\quad\mbox{ for all
}\; l\in (0,1], \;  x\in\R^2.$$
Then, for every $f,g\in\mathcal{C}^0(\mathbb{R}^2,\R)$, every
$m\geq 0$ and $\beta\in (0,1)$, there holds:
\begin{align*}
& \|\nabla^{(m)}(f\ast\phi_l)\|_{0} \leq
\frac{C}{l^m}\|f\|_0,\tag*{(\theequation)$_1$}\vspace{1mm} \refstepcounter{equation} \label{stima1}\\
& \|f - f\ast\phi_l\|_0 \leq \min\big\{l^2\|\nabla^{2}f\|_0,
l^{1+\beta}\|\nabla f\|_{0,\beta}, l\|\nabla f\|_0,
{l^\beta}\|f\|_{0,\beta}\big\}, \tag*{(\theequation)$_2$} \vspace{1mm} \label{stima2}\\
& \|\nabla^{(m)}\big((fg)\ast\phi_l - (f\ast\phi_l)
(g\ast\phi_l)\big)\|_0\leq {C}{l^{2- m}}\|\nabla f\|_{0}
\|\nabla g\|_{0}, \tag*{(\theequation)$_3$} \label{stima4}
\end{align*}
with constants $C>0$ depending only on $m$.
\end{lemma}

\medskip

The next two lemmas provide the decomposition of symmetric
matrices into linear combinations of rank-one ``primitive matrices''.
The first result is self-evident, proved for the general dimensionality $d\geq 2$ in
\cite[Lemma 5.2]{CDS}, while the second one is a combination of the
local decomposition with a partition of unity - type statement from \cite[Lemma 3.3]{Laszlo}:

\begin{lemma}\label{lem_dec_def}
There exist $r_0\in (0,1)$ and linear maps $\{\bar a_i:\R^{2\times
  2}_{\sym}\to\R\}_{i=1}^{3}$ such that, denoting:
$$ \eta_1 = e_1,\quad \eta_2 = \frac{e_1+e_2}{\sqrt{2}},\quad \eta_3=
e_2 \quad \mbox{and} \quad H_0=\sum_{i=1}^3 \eta_i \otimes \eta_i, $$
for all $H\in B(H_0,r_0)\subset \R^{2\times 2}_{\sym}$ there holds:
$$H = \sum_{i=1}^{3} \bar a_i(H) \eta_i\otimes\eta_i \quad \mbox{ and } \quad
|\bar a_i (H)-1|\leq\frac{1}{2} \; \mbox{ for all }\;  i =1\ldots 3.$$
\end{lemma}

\begin{lemma}\label{lem_met_deco}
There exists an integer $N_0$, the unit vectors $\{\eta_i\in\R^2\}_{i=1}^\infty$ 
and the nonnegative functions
$\{\varphi_i\in\mathcal{C}^\infty_c(\R^{2\times 2}_{\sym,
  >},\R)\}_{i=1}^\infty$, yielding the decomposition with properties below:
$$ H = \sum_{i=1}^{\infty} \varphi_i(H)^2 \eta_i\otimes\eta_i \quad
\mbox{ for all } \; H\in \R^{2\times 2}_{\sym, >},$$
\begin{itemize} 
\item[(i)] at most $N_0$ terms in the above sum are nonzero, for each $H$,
\item[(ii)] every compact set $\mathcal{K}\subset  \R^{2\times 2}_{\sym, >}$ 
induces a finite set of indices $J(\mathcal{K})\subset\mathbb{N}$, such that
$\varphi_i(H)=0$ for all $H\in \mathcal{K}$ and all $i\not\in J(\mathcal{K})$.
\end{itemize}
\end{lemma}

\medskip

The remaining two lemmas in this section provide two types of Step construction in the convex
integration algorithm. The given immersion is
modified by oscillatory perturbations, with the scope of
cancelling a single term of the rank-one type featured in the
Lemma \ref{lem_dec_def}. The first construction, referred to as Nash's spiral,
necessitates two normal vector fields:

\begin{lemma}\label{lem_step1}\textup{[STEP: NASH'S SPIRAL]}
Let $u\in\mathcal{C}^2(\R^2,\R^4)$ be an immersion and let $E_u^1,
E_u^2\in\mathcal{C}^{1}(\R^2, \R^4)$ be a given orthonormal pair of its normal
vector fields, so that:
$$(\nabla u)^T E_u^1 = (\nabla u)^T E_u^2 = 0,\quad |E^1_u|=|E^2_u|=1,
\quad \langle E^1_u, E^2_u\rangle = 0 \quad\mbox{ in } \R^2.$$
For a given unit vector $\eta\in\R^2$ we set $t=\langle x, \eta\rangle$ and denote:
$$\Gamma(t) = \sin t, \quad \bar\Gamma(t) = \cos t.$$
Then, for every $\lambda>0$ and $a\in\mathcal{C}^1(\R^2, \R)$, the
vector field $\tilde u\in \mathcal{C}^1(\R^2, \R^4)$ given by the formula:
$$\tilde u(x) = u(x)+ \frac{\Gamma(\lambda t )}{\lambda} a(x) E^1_u(x) + 
\frac{\bar\Gamma(\lambda t )}{\lambda}a(x)E^2_u(x),$$ 
satisfies the following identity:
\begin{equation*}
(\nabla \tilde u)^T\nabla \tilde u - (\nabla u)^T\nabla u =
a^2\eta\otimes \eta + \mathcal{R},
\end{equation*}
with the error term $\mathcal{R}$ in:
\begin{equation*}
\begin{split}
\mathcal{R} & = \mathcal{R}(\lambda, a, \eta, \nabla u,  E^1_u, E^2_u)
\\ & = \Big\{ 2\frac{\Gamma(\lambda t)}{\lambda} a \;\sym((\nabla u)^T\nabla E_u^1)
+ 2\frac{\bar\Gamma(\lambda t)}{\lambda} a \;\sym((\nabla u)^T\nabla E_u^2)
+\frac{2}{\lambda} a^2\sym\big((\nabla E^2_u)^TE_u^1\otimes\eta\big)\Big\}
\\ & \quad + \Big\{\frac{\Gamma(\lambda t)^2}{\lambda^2} a^2 (\nabla E^1_u)^T\nabla E_u^1
+ \frac{\bar\Gamma(\lambda t)^2}{\lambda^2} a^2 (\nabla E^2_u)^T\nabla E_u^2
\\ & \qquad\quad
+ 2 \frac{\Gamma(\lambda t)\bar\Gamma(\lambda t)}{\lambda^2}
a^2\; \sym((\nabla E^1_u)^T\nabla E_u^2)
+ \frac{1}{\lambda^2}\nabla a\otimes \nabla a\Big\}.
\end{split}
\end{equation*}
\end{lemma}
\begin{proof}
We first calculate the gradient of the modified immersion:
\begin{equation*}
\begin{split}
\nabla \tilde u = & \; \nabla u + \Gamma'(\lambda t) a E_u^1\otimes\eta +  
\frac{\Gamma(\lambda t)}{\lambda}\big( a \nabla E_u^1 + E^1_u\otimes\nabla a\big)
\\ & + \bar\Gamma'(\lambda t) a E_u^2\otimes\eta +  
\frac{\bar\Gamma(\lambda t)}{\lambda}\big( a \nabla E_u^2 +
E^2_u\otimes\nabla a\big).
\end{split}
\end{equation*}
This leads to the following formula, where we suppress the argument
$\lambda t$ in $\Gamma$ and $\bar\Gamma$ and order the terms according to powers of $\lambda$:
\begin{equation*}
\begin{split}
& (\nabla \tilde u)^T\nabla \tilde u - (\nabla u)^T\nabla u  \\ & =
\Big\{(\Gamma')^2a^2\eta\otimes\eta + (\bar\Gamma')^2a^2\eta\otimes\eta \Big\}
\\ & \quad + \Big\{ 2\frac{\Gamma}{\lambda} a\; \sym\big((\nabla u)^T \nabla E^1_u\big) 
+ 2\frac{\bar \Gamma}{\lambda} a \;\sym\big((\nabla u)^T \nabla E^2_u\big) 
+ 2\frac{\Gamma\Gamma'}{\lambda} a \; \sym(\nabla a \otimes\eta)  \\ & 
\qquad \; \; 
+ 2\frac{\Gamma'\bar\Gamma}{\lambda} a^2 \sym\big((\nabla E_u^2)^TE^1_u \otimes\eta\big) 
+ 2\frac{\Gamma\bar\Gamma'}{\lambda} a^2 \sym\big((\nabla E_u^1)^TE^2_u \otimes\eta\big) 
\\ & \qquad\;\; + 2\frac{\bar\Gamma\bar\Gamma'}{\lambda} a \;\sym(\nabla a \otimes\eta) \Big\}
\\ & \quad + \Big\{\frac{\Gamma^2}{\lambda^2} a^2 (\nabla E_u^1)^T\nabla E^1_u 
+ \frac{\Gamma^2}{\lambda^2} \nabla a\otimes\nabla a 
+ 2 \frac{\Gamma\bar\Gamma}{\lambda^2} a^2 \sym\big((\nabla E_u^1)^T\nabla E^2_u \big)
\\ & \qquad \; \; + \frac{\bar\Gamma^2}{\lambda^2} a^2 (\nabla E_u^2)^T\nabla E^2_u + 
\frac{\bar\Gamma^2}{\lambda^2} \nabla a\otimes\nabla a\Big\}.
\end{split}
\end{equation*}
We now observe that the sum of the two terms in the first parentheses
equals $a^2\eta\otimes\eta$ because: 
$$(\Gamma')^2+(\bar\Gamma')^2=1,$$
whereas the remaining terms are exactly
$\mathcal R$. This is because the third and the sixth term in the
second parentheses cancel out due to
$\Gamma\Gamma' + \bar\Gamma\bar\Gamma'=0$, and the fourth and the fifth
term in there combine due to
$\Gamma'\bar\Gamma - \Gamma\bar\Gamma'=1$. Also, the second and the
fifth term in the third parentheses combine, due to
$\Gamma^2+\bar\Gamma^2=1$. The proof is done.
\end{proof}

\medskip

\noindent In the second Step construction below, referred to as Kuiper's
corrugation and inspired by \cite{Kuiper}, the oscillatory modification is
added in a single normal direction, accompanied by a
matching tangential perturbation. This ansatz, see (\ref{ansi}), is in fact the
basic building block of the convex integration Step for the Monge-Amp\`ere system
in \cite{lew_conv, lewpak_MA}, where one superposes $\frac{\Gamma(\lambda t
  )}{\lambda} a(x)$ in the (constant, chosen) normal direction and
$\frac{\bar\Gamma(\lambda t )}{\lambda}a(x)^2\eta$ in the tangent
plane. The corrugation profiles $\Gamma, \bar\Gamma$ satisfy
the ``parabolic'' constraint $(\Gamma')^2+2\bar\Gamma'=1$, as opposed
to its ``elliptic'' counterpart $\Gamma^2+ \bar\Gamma^2=1$ appearing in Lemma \ref{lem_step1}.
The observation that the same ansatz also applies in the present fully
nonlinear setting of the isometric immersion system, was first made
in \cite[Lemma 3.1]{CHI2}. This led to a substantial simplification of the convex
integration analysis, improved the regularity
exponent pertaining to codimension $k=1$ from previously established
$\frac{1}{1+2s_d}$ to $\frac{1}{1+2(s_d-d)}$, and 
revealed the further close connection between the Monge-Amp\`ere
and the isometric immersion systems.

\medskip

\noindent We note that the extra tangential 
perturbation component $w$ in (\ref{ansi}) is essential for the
proofs  and it will be chosen according to the
oscillatory defect decomposition in section \ref{sec_IBP}, following
the approach of \cite{CHI2, in_lew2}. Naturally, the error terms in
(\ref{sisi}) are now more involved, and we 
split them into four groups: the non-oscillatory term
$\frac{1}{\lambda^2}\nabla a\otimes\nabla a$, then
the four principal oscillatory terms,
the term $\sym\nabla w$, and the residual terms $\mathcal{R}_1, \mathcal{R}_2$:

\begin{lemma}\label{lem_step2} \textup{[STEP: KUIPER'S CORRUGATION]}
Let $u\in\mathcal{C}^2(\R^2,\R^4)$ be an immersion and let $E_u
\in\mathcal{C}^{1}(\R^2, \R^4)$ be a given unit normal vector field to $u$, so that:
$$(\nabla u)^T E_u = 0,\quad |E_u|=1 \quad\mbox{ in } \R^2.$$
For a given unit vector $\eta\in\R^2$ we set $t=\langle x, \eta\rangle$ and denote:
$$\Gamma(t) = \sqrt{2}\sin t, \quad \bar\Gamma(t) = -\frac{1}{4}\sin(2 t).$$
Then, for every $\lambda>0$, $a\in\mathcal{C}^1(\R^2, \R)$ and
$w\in\mathcal{C}^1(\R^2, \R^2)$, the
vector field $\tilde u\in \mathcal{C}^1(\R^2, \R^4)$ in:
\begin{equation}\label{ansi}
\begin{split}
& \tilde u(x) = u(x)+ \frac{\Gamma(\lambda t )}{\lambda} a(x) E_u(x) + 
T_u(x)\Big(\frac{\bar\Gamma(\lambda t )}{\lambda}a(x)^2\eta+w(x)\Big),
\\ & \mbox{where: } T_u = (\nabla u)\big((\nabla u)^T\nabla
u\big)^{-1}\in\mathcal{C}^1(\R^2, \R^{4\times 2}),
\end{split}
\end{equation}
satisfies the following identity:
\begin{equation}\label{sisi}
\begin{split}
& (\nabla \tilde u)^T\nabla \tilde u - (\nabla u)^T\nabla u = 
a^2\eta\otimes \eta +  \frac{1}{\lambda^2}\nabla a\otimes \nabla a 
\\ & +\frac{\Gamma(\lambda t)^2-1}{\lambda^2}S_1 
+\frac{\Gamma(\lambda t)\Gamma'(\lambda t)}{\lambda}S_2
+ \frac{\Gamma(\lambda t)}{\lambda}S_3
+ \frac{\bar\Gamma(\lambda t)}{\lambda}S_4 \\ & + 2\sym\nabla w +
\mathcal{R}_1 + \mathcal{R}_2,
\end{split}
\end{equation}
with the leading order error terms given via:
\begin{equation*}
\begin{split}
& S_1= \nabla a\otimes \nabla a, 
\qquad \qquad \qquad \;\; \, S_2= 2a\,\sym(\nabla a\otimes\eta), \\ 
& S_3 = 2a\,\sym\big((\nabla u)^T\nabla E_u\big),\qquad 
S_4 = 2\,\sym\big((\nabla u)^T\nabla(a^2 T_u\eta)\big), 
\end{split}
\end{equation*}
the first residual error term given in:
\begin{equation*}
\begin{split}
& \mathcal{R}_1 = \mathcal{R}_1(\lambda, a, \eta, \nabla u, E_u) 
= (\bar\Gamma')^2a^4 |T_u\eta|^2\eta\otimes\eta 
\\ & + \Big\{ 2\frac{\Gamma'(\lambda t)\bar\Gamma(\lambda
  t)}{\lambda} a \;\sym\big((\eta\otimes E_u)\nabla(a^2T_u\eta)\big)
+ 2\frac{\Gamma(\lambda t)\bar\Gamma'(\lambda t)}{\lambda} a^3 
\;\sym\big((\eta\otimes\eta) (T_u)^T\nabla E_u)\big)
\\ & \qquad  + 2\frac{\bar\Gamma(\lambda t)\bar\Gamma'(\lambda t)}{\lambda}
a^2\sym\big((\eta\otimes \eta) (T_u)^T\nabla(a^2 T_u\eta)\big)\Big\}
\\ & +\Big\{ \frac{\Gamma(\lambda t)^2}{\lambda^2}a^2 (\nabla E_u)^T\nabla E_u
+ 2 \frac{\Gamma(\lambda t)\bar\Gamma(\lambda t)}{\lambda^2} \sym\big(
\nabla(aE_u)^T\nabla(a^2 T_u\eta)\big) \\ & \qquad +
\frac{\bar\Gamma(\lambda t)^2}{\lambda^2} \nabla(a^2
T_u\eta)^T\nabla(a^2 T_u\eta) \Big\}, 
\end{split}
\end{equation*}
and the second residual error term involving $w$:
\begin{equation*}
\begin{split}
& \mathcal{R}_2= \mathcal{R}_2(\lambda, a, \eta, \nabla u, E_u, w) \\ &
= \Big\{2\,\sym\big((\nabla u)^T\big[(\partial_1 T_u)w, (\partial_2 T_u)w\big]\big)+ 
2\Gamma'(\lambda t) a\,\sym\big((\eta\otimes E_u)\nabla(T_u w)\big)
\\ & \qquad + 2\bar\Gamma'(\lambda t)\, a^2\sym\big((\eta\otimes\eta) (T_u)^T\nabla
(T_u w)\big) + \nabla (T_u w)^T\nabla (T_uw) \Big\} \\ & \quad + \Big\{ 2\frac{\Gamma(\lambda
  t)}{\lambda} \sym\big(\nabla (a E_u)^T\nabla (T_u w)\big)
+2\frac{\bar\Gamma(\lambda t)}{\lambda}\sym\big(\nabla (a^2 T_u
\eta)^T\nabla (T_u w)\big)\Big\}.
\end{split}
\end{equation*}
\end{lemma}

\begin{proof}
We start by calculating the gradient of the modified immersion:
\begin{equation*}
\begin{split}
\nabla \tilde u = & \;\nabla u + a\Gamma'(\lambda t) E_u\otimes \eta +
a^2 \bar\Gamma'(\lambda t) T_u \eta\otimes\eta + \nabla(T_u w)
\\ & + \frac{\Gamma(\lambda t)}{\lambda}\big(a\nabla E_u + E_u\otimes \nabla
a\big) +\frac{\bar\Gamma(\lambda t)}{\lambda}\nabla(a^2 T_u\eta).
\end{split}
\end{equation*}
We obtain the following formula, where we suppress the argument
$\lambda t$ in $\Gamma$ and $\bar\Gamma$, and order the terms
according to powers of $\lambda$ and their independence / dependence on $w$:
\begin{equation*}
\begin{split}
& (\nabla \tilde u)^T\nabla \tilde u - (\nabla u)^T\nabla u \\ & = 
\Big\{a^2 (\Gamma')^2\eta\otimes\eta + 2a^2 \bar\Gamma' (\nabla u)^T
T_u \eta\otimes\eta + (\bar\Gamma')^2a^4 (\eta\otimes\eta) (T_u)^TT_u(\eta\otimes\eta)\Big\}
\\ & \quad + \Big\{ 2 \frac{\Gamma}{\lambda} a\;\sym\big((\nabla u)^T\nabla E_u\big)
+ 2 \frac{\bar\Gamma}{\lambda} \sym\big((\nabla u)^T\nabla (a^2 T_u \eta)\big)
+ 2 \frac{\Gamma\Gamma'}{\lambda} a\;\sym\big(\nabla a\otimes\eta\big)
\\ & \qquad \; \; + 2 \frac{\Gamma'\bar\Gamma}{\lambda}
a\;\sym\big((\eta\otimes E_u)\nabla (a^2 T_u \eta)\big)  
+ 2 \frac{\Gamma\bar\Gamma'}{\lambda} a^3\;\sym\big((\eta\otimes\eta)(T_u)^T\nabla E_u)\big)
\\ & \qquad \; \;  + 2 \frac{\bar\Gamma\bar\Gamma'}{\lambda}
a^2\;\sym\big((\eta\otimes\eta)(T_u)^T\nabla( a^2T_u\eta)\big) \Big\}
\\ & \quad + \Big\{ \frac{\Gamma^2}{\lambda^2}a^2 (\nabla E_u)^T\nabla E_u
+ \frac{\Gamma^2}{\lambda^2}\nabla a\otimes\nabla a 
+ 2\frac{\Gamma\bar\Gamma}{\lambda^2}\sym\big(\nabla(aE_u)^T \nabla(a^2T_u\eta) \big)
\\ & \qquad \; \; + \frac{\bar\Gamma^2}{\lambda^2}\nabla(a^2T_u\eta)^T \nabla(a^2T_u\eta) \Big\}
\\ & \quad + \Big\{ 2 \,\sym\big((\nabla u)^T\nabla (T_uw)\big) + 2\Gamma' a\,
\sym\big((\eta\otimes E_u) \nabla(T_uw)\big) 
\\ & \qquad \; \; + 2 \bar\Gamma' a^2 \sym\big((\eta\otimes\eta)(T_u)^T\nabla(T_uw)\big)
+ \nabla(T_uw)^T\nabla(T_uw)\Big\}
\\ & \quad + \Big\{2\frac{\Gamma}{\lambda} \sym\big(\nabla(a E_u)^T \nabla(T_uw)\big) 
+ 2\frac{\bar\Gamma}{\lambda} \sym\big(\nabla(a^2 T_u\eta)^T \nabla(T_uw)\big)\}. 
\end{split}
\end{equation*}
This precisely yields the claim, upon noticing that first two terms in
the first parentheses sum to $a^2\eta\otimes\eta$ because 
$(\nabla u)^T T_u=\Id_2$ and because: 
$$(\Gamma')^2+2\bar\Gamma' = 1,$$ 
while the third term in the first parentheses can be rewritten as $(\bar\Gamma')^2a^4
|T_u\eta|^2\eta\otimes\eta$. Finally, the first term in the fourth parentheses equals:
$$2\sym\nabla w + 2\sym\big((\nabla u)^T\big[(\partial_1 T_u)w, (\partial_2 T_u)w\big]\big).$$ 
The proof is done.
\end{proof}

\section{Existence and propagation of normal frame lemmas}\label{sec_normals}

In this section, we gather several properties that will be satisfied by all
the immersions inductively constructed along our convex integration algorithm. 
The first observation is standard:

\begin{lemma}\label{lem_det}
Let $u\in\mathcal{C}^1(\bar\omega,\R^4)$, defined on the closure of an
open, bounded set $\omega\subset\R^2$, satisfy:
\begin{equation}\label{immers_gamma}
\frac{1}{\gamma}\Id_2\leq (\nabla u)^T\nabla u \leq \gamma\Id_2
\quad\mbox{ in } \bar\omega,
\end{equation}
 for some $\gamma\geq 1$. Then there holds:
$$\big(\frac{2}{\gamma}\big)^{1/2} \leq |\nabla u| \leq (2\gamma)^{1/2} \quad\mbox{ and }\quad 
\frac{1}{\gamma^2}\leq\det((\nabla u)^T\nabla u)\leq \gamma^2
\quad \mbox{ in } \;\bar\omega.$$
\end{lemma}
\begin{proof}
The first assertion holds as $|\partial_iu(x)|^2=\langle \nabla
u(x)^T(\nabla u(x)) e_i, e_i\rangle \in [1/\gamma,\gamma]$ for $i=1,2$. For the
second assertion, note that both eigenvalues
of the symmetric matrix $(\nabla u)^T\nabla u$ are within the
interval $[1/\gamma, \gamma]$. This implies the stated bound on the determinant.
\begin{comment} 

Consequently, we get
the upper bound in the second assertion, as $\|\det ((\nabla u)^T\nabla
u)\|_0\leq 2\|\partial_1u\|_0^2\|\partial_2u\|_0^2$.

For the lower bound, fix $x\in\bar\omega$ and denote $A = \nabla
u(x)^T\nabla u(x) \in\R^{2\times 2}_\sym$. By assumption: $\langle A
(s,t)^T, (s,t)^T\rangle\geq \frac{1}{\gamma}(s^2+t^2) $ for every vector
$(s,t)\in\R^2$, or equivalently:
$$(\gamma A_{11}-1) s^2+2\gamma A_{12}st+
(\gamma A_{22}-1) t^2 \geq 0\quad\mbox{ for all } s,t\in\R.$$ 
Taking $t=0$ and $s=0$ implies that $\gamma A_{11}\geq 1$ and $\gamma
A_{22}\geq 1$, while setting $t=1$ the discriminant of the resulting

quadratic equation in $s$ must be nonpositive, namely:
\begin{equation*}
\begin{split}
& 0\geq 4\gamma^2A_{12}^2 -4 (\gamma A_{11}-1)(\gamma A_{22}-1) \\ & = 4
\gamma^2 (A_{12}^2-A_{11}A_{22}) +4(\gamma A_{11} +\gamma A_{22}-1)
\geq - 4\gamma^2 \det A + 4.
\end{split}
\end{equation*} 
This proves the claim.
\end{comment}
\end{proof}

\medskip

The next lemma gathers bounds on the tangent frame of an immersion:

\begin{lemma}\label{lem_Tu}
Given $u\in\mathcal{C}^{k+1}(\bar\omega,\R^4)$ posed on the closure of an
open, bounded $\omega\subset\R^2$, with  (\ref{immers_gamma}) for
some $\gamma\geq 1$, consider the tangent frame $T_u=(\nabla u)\big((\nabla u)^T\nabla
u\big)^{-1}\in\mathcal{C}^{k}(\bar\omega,\R^{4\times 2})$. Then:
\begin{itemize}
\item[(i)] $\|T_u\|_0\leq C$ and $\displaystyle{\|\nabla^{(m)}T_u\|_0\leq C
\hspace{-4mm} \sum_{p_1+2p_2+\ldots mp_m=m} \; \prod_{i=1}^m\|\nabla^{(i)}\nabla u\|_0^{p_i}}$
for all $ m=1\ldots k$, where $C$ depends only on $\gamma$ and $k$.
\item[(ii)] If $k\geq 1$ and if, with some constants $\mu>1$, $A<1$ and
  $\bar C>1$, we additionally have:
$$\|\nabla^{(m)}\nabla u\|_0\leq \bar C \mu^{m}A \quad \mbox{ for all } \; m=1\ldots k,$$
then there holds:
$$\|\nabla^{(m)}T_u\|_0\leq C\mu^m A \quad\mbox{ for all } m=1\ldots k,$$
where $C$ depends only on $\gamma$, $\bar C$ and $k$, but not on $\mu, A$.
\end{itemize}  
\end{lemma}
\begin{proof}
Write $T_u=\frac{1}{\det((\nabla u)^T\nabla u)}(\nabla
u)\,\mbox{cof}\big((\nabla u)^T\nabla u\big)$. Then Lemma
\ref{lem_det} implies the first assertion in (i), with $C$ depending on
$\gamma$. For the second assertion, we use Fa\'a di Bruno's formula
to the composition $T_u=f\circ (\nabla u)$, where the smooth field
$f(\xi) = \frac{1}{\det(\xi^T\xi)}\xi \mbox{cof}\big(\xi^T\xi\big)$
has all its derivatives bounded on the
domain containing the values of $\nabla u$, in virtue of Lemma \ref{lem_det}.

\medskip

\noindent To validate the bound in (ii), we use (i) and the assumption in:
$$\|\nabla^{(m)}T_u\|_0\leq C
\hspace{-4mm} \sum_{p_1+2p_2+\ldots mp_m=m} \; \prod_{i=1}^m\big(\bar
C \mu^i A\big)^{p_i} \leq C {\bar C}^m\mu^m A.$$
Indeed, at least one of the exponents $p_i$ in each product under the sum
above must be positive. This completes the proof.
\end{proof}

\medskip

We now provide the key construction and estimates
on the propagation of the normal frame from a given, to a nearby
immersion. The same construction, but with less precise bounds (that do not suffice for our arguments), 
appeared in \cite[Proposition 5.3]{DI2020}. Namely, we have:

\begin{lemma}\label{lem_propa}
Let $u\in\mathcal{C}^{k+1}(\bar\omega, \R^4)$ be defined on the closure of an
open, bounded set $\omega\subset\R^2$, and satisfy (\ref{immers_gamma}) for
some $\gamma\ge 1$. Let $E_u^1, E_u^2\in\mathcal{C}^{k}(\bar\omega, \R^4)$
be a normal frame of $u$, namely:
\begin{equation}\label{norm_frame}
(\nabla u)^T E_u^1 = (\nabla u)^T E_u^2 = 0,\quad |E^1_u|=|E^2_u|=1.
\quad \langle E^1_u, E^2_u\rangle = 0 \quad\mbox{ in } \bar\omega.
\end{equation}
Then, there exists a constant $\rho_\gamma\in (0,1)$ depending only on $\gamma$, such that
for every $v\in\mathcal{C}^{k+1}(\bar\omega, \R^4)$, satisfying
$\|\nabla v-\nabla u\|_0\leq \rho_\gamma$, there holds:
\begin{itemize}
\item[(i)] $v$ is an immersion and its normal frame $E_v^1,
  E_v^2\in\mathcal{C}^{k}(\bar\omega, \R^4)$, namely: 
$$(\nabla v)^T E_v^1 = (\nabla v)^T E_v^2 = 0 \quad \mbox{and}\quad
|E^1_v|=|E^2_v|=1 \quad\mbox{and} \quad \langle E^1_v, E^2_v\rangle =
0\quad \mbox{in } \bar\omega,$$ 
can be defined via the following formulas:
\begin{equation}\label{propa_def} 
\begin{split}
& E^1_v = \frac{\nu^1_v}{|\nu^1_v|}, \qquad 
E^2_v = \frac{\nu^2_v - \langle \nu_v^2, E^1_v\rangle E^1_v}{|\nu^2_v -
  \langle \nu_v^2, E^1_v\rangle E^1_v|},\\
& \mbox{where } \nu_v^i= \big(\Id_4 - T_v(\nabla v - \nabla u)^T\big)E^i_u
\quad\mbox{for } i=1,2 \\ & \mbox{and} \quad T_v= (\nabla v)((\nabla v)^T\nabla v)^{-1}.
\end{split}
\end{equation}
\item[(ii)] For all $m=0\ldots k$ there holds, with constants $C$ depending only on $\gamma$ and $k$:
\begin{equation*}
\begin{split}    
& \|\nabla^{(m)}(E^1_v-E^1_u)\|_0\\ & \leq C
\hspace{-7mm}\sum_{ p+\sum_{i=1}^mi(p_i+\bar
  p_i + \dbar p_i)=m\hspace{-10mm}} \hspace{-7mm} \|\nabla^{(p)}(\nabla v-\nabla u)\|_0
\prod_{i=1}^m \|\nabla^{(i)}(\nabla v - \nabla u)\|_0^{p_i} \|\nabla^{(i)}\nabla v\|_0^{\bar p_i}
\|\nabla^{(i)}E^1_u\|_0^{\dbar p_i},\\
& \|\nabla^{(m)}(E^2_v-E^2_u)\|_0\\ & \leq C \hspace{-7mm}\sum_{p+\sum_{i=1}^mi(p_i+\bar
  p_i + \dbar p_i+\tbar p_i)=m\hspace{-10mm}} \hspace{-11mm} \|\nabla^{(p)}(\nabla v-\nabla u)\|_0
\prod_{i=1}^m \|\nabla^{(i)}(\nabla v - \nabla u)\|_0^{p_i} \|\nabla^{(i)}\nabla v\|_0^{\bar p_i}
\|\nabla^{(i)}E^1_u\|_0^{\dbar p_i} \|\nabla^{(i)}E^2_u\|_0^{\tbar p_i}.
\end{split}
\end{equation*}
\item[(iii)]  If $k\geq 1$ and if, with some constants $\mu>1, A<1$ and $\bar C>1$, we additionally have:
\begin{equation*}
\begin{split}
& \|\nabla^{(m)}(\nabla v-\nabla u)\|_0\leq \bar C \mu^mA \quad\mbox{ for } m=0\ldots k,
\\ & \|\nabla^{(m)}\nabla u\|_0\leq \bar C\mu^{m}, ~~ \;\;
\|\nabla^{(m)}E^i_u\|_0\leq \bar C\mu^{m}  \quad\mbox{for } m=1\ldots k,\; i=1,2,
\end{split}
\end{equation*}
then there holds:
$$\|\nabla^{(m)}(E^i_v-E^i_u)\|_0\leq C \mu^mA \quad\mbox{ for all }
m=0\ldots k,\; i=1,2,$$
where $C$ depends only on $\gamma, \bar C$ and $k$, but not on $\mu, A$.
\end{itemize}
\end{lemma}
\begin{proof}
{\bf 1.} Observe that, by the first bound in Lemma \ref{lem_det}:
\begin{equation*}
\begin{split}
\|(\nabla v)^T\nabla v - (\nabla u)^T\nabla u\|_0& \leq \|\nabla
v-\nabla u\|_0(\|\nabla u\|_0+\|\nabla v\|_0) \\ & \leq \|\nabla
v-\nabla u\|_0 (2\|\nabla u\|_0+\|\nabla v-\nabla u\|_0) \leq C \rho_\gamma,
\end{split}
\end{equation*}
implying for small $\rho_\gamma$ that $v$ is an immersion and:
\begin{equation}\label{a}
\frac{1}{2\gamma}\Id_2\leq (\nabla v)^T\nabla v\leq 2\gamma \,\Id_2
\quad\mbox{in } \bar\omega.
\end{equation}
Thus, $T_v\in \mathcal{C}^{k}(\bar\omega, \R^{4\times 2})$
is well defined and $\nu_v^i$ are normal to $\nabla v$, as $(\nabla v)^TT_v=\Id_2$ and:
$$(\nabla v)^T\nu_v^i = \big((\nabla v)^T - (\nabla v-\nabla
u)^T\big)E^i_u=(\nabla u)^TE^i_u=0 \quad\mbox{ for } i=1,2.$$
Further, applying Lemma \ref{lem_Tu} to bound $\|T_v\|_0$ by a
constant that only depends on $\gamma$, we get:
\begin{equation}\label{c}
 \|\nu_v^i- E^i_u\|_0\leq  \|T_v\|_0\|\nabla v-\nabla u\|_0\leq C
\rho_\gamma\leq \frac{1}{2} \quad \mbox{ for }\; i=1,2,
\end{equation}
when $\rho_\gamma$ is sufficiently small, so in particular:
$$ |\nu_v^i(x)|\in \big[{1}/{2}, {3}/{2}\big]\quad 
\mbox{ for all } \; x\in\bar\omega, \;  i=1,2$$
implying that $E^1_v$ in (\ref{propa_def}) is well defined. Next, we define and observe:
\begin{equation}\label{jad}
\begin{split}
&\tilde\nu^2_v\doteq \nu_v^2-\langle \nu_v^2, E_v^1\rangle E_v^1\\
& \tilde\nu^2_v- E_u^2 =- T_v(\nabla v-\nabla u)^TE_u^2 -
\langle E_u^2, E_v^1\rangle E_v^1 = - T_v(\nabla v-\nabla u)^TE_u^2 -
\langle E_u^2, E_v^1-E^1_v\rangle E_v^1.
\end{split}
\end{equation}
Since $E_v^1-E_u^1=f(\nu_v^1)- f(E_u^1)$, where the function $f(\xi) = \frac{\xi}{|\xi|}$ is Lipschitz on the
set $\{\xi\in\R^4;\,|\xi|\in [1/2, 3/2]\}$, we consequently obtain, recalling
(\ref{c}) and taking $\rho_\gamma$ small:
\begin{equation*}
\begin{split}
\|\tilde\nu_v^2-E_u^2\|_0&\leq \|T_v\|_0\|\nabla u-\nabla v\|_0 +
\|E_v^1-E_u^1\|_0\leq \|T_v\|_0\|\nabla u-\nabla v\|_0 +
C\|\nu_v^1-E_u^1\|_0 \\ & \leq C \|T_v\|_0\|\nabla u-\nabla v\|_0\leq
C\rho_\gamma\leq \frac{1}{2}.
\end{split}
\end{equation*}
In particular, $E^2_v=f(\tilde \nu^2_v)$ is well defined and we note that:
$$ |\tilde \nu_v^2(x)|\in \big[{1}/{2}, {3}/{2}\big]\quad 
\mbox{ for all } \; x\in\bar\omega.$$
This justifies all the claims in (i), in view of:
$\langle E_v^1, E_v^2\rangle = \frac{1}{|\tilde \nu^2_v|}\big(\langle E_v^1, \nu_v^2\rangle -
  \langle \nu_v^2, E_v^1\rangle |E^1_v|^2\big)=0.$

\smallskip

{\bf 2.} We write:
\begin{equation*}
E^1_v-E^1_u = f(\nu^1_v) - f(E^1_u) = \Big(\int_0^1\nabla
f(t\nu^1_v+(1-t)E^1_u)~\mbox{d}t\Big) (\nu^1_v-E^1_u) 
\end{equation*}
and for each fixed $t\in (0,1)$ and $m=1\ldots k$, we estimate,
using Fa\'a di Bruno's formula:
\begin{equation*}
\|\nabla^{(m)}_x\big((\nabla f)\circ(E_u^1+
t(\nu_v^1-E^1_u))\big)\|_0\leq C \hspace{-4mm}\sum_{p_1+2p_2+\ldots mp_m=m} \;
\prod_{i=1}^m\|\nabla^{(i)}(E_u^1+ t(\nu_v^1-E^1_u))\|_0^{p_i}.
\end{equation*}
The above is trivially valid at $m=0$ as well. Consequently, for all $m=0\ldots k$:
\begin{equation}\label{goa}
\begin{split}
&\|\nabla^{(m)}(E^1_v-E^1_u)\|\leq C \hspace{-6mm}\sum_{p+\sum_{i=1}^mip_i=m}
\hspace{-4mm}\|\nabla^{(p)}(\nu_v^1-E^1_u)\|_0 \prod_{i=1}^m\Big(\|\nabla^{(i)}E_u^1\|_0^{p_i}
+ \|\nabla^{(i)}(\nu_v^1-E^1_u)\|_0^{p_i}\Big)
\\ & \leq C \hspace{-4mm}\sum_{p+\sum_{i=1}^mi(p_i+\dbar p_i)=m}\hspace{-6mm}
\|\nabla^{(p)}(\nu_v^1-E^1_u)\|_0 \prod_{i=1}^m \|\nabla^{(i)}(\nu_v^1-E^1_u)\|_0^{p_i}
\|\nabla^{(i)}E_u^1\|_0^{\dbar p_i}.
\end{split}
\end{equation}
Recall that there holds, for all $m=0\ldots k$: 
\begin{equation}\label{b}
\|\nabla^{(m)}(\nu^i_v-E^i_u)\|_0 \leq C\hspace{-2mm}\sum_{p+q+t=m}\hspace{-2mm}
\|\nabla^{(p)}(\nabla u-\nabla v)\|_0 \|\nabla^{(q)}T_v\|_0  \|\nabla^{(t)}E^i_u\|_0.
\end{equation}
Consequently, (\ref{goa}) yields, for all $m=0\ldots k$:
\begin{equation}\label{goa2}
\begin{split}
&\|\nabla^{(m)}(E^1_v-E^1_u)\|\\ & \leq C \hspace{-4mm}
\sum_{p+\sum_{i=1}^mi(p_i+\bar p_i+\dbar p_i)=m\hspace{-13mm}}\hspace{-5mm}
\|\nabla^{(p)}(\nabla v -\nabla u)\|_0 \prod_{i=1}^m \|\nabla^{(i)}(\nabla v-\nabla u)\|_0^{p_i}
\|\nabla^{(i)}T_v\|_0^{\bar p_i} \|\nabla^{(i)}E_u^1\|_0^{\dbar p_i},
\end{split}
\end{equation}
which implies the first bound in (ii) by Lemma \ref{lem_Tu} (i).
To show the second bound, we separately estimate the two terms
in the right hand side of (\ref{jad}). Firstly, similar to (\ref{b}):
\begin{equation*}
\begin{split}
&\|\nabla^{(m)}(\nu_v^2-E_u^2)\|_0 = \|\nabla^{m}\big(T_v(\nabla v-\nabla u)^TE_u^2\big)\|_0
\\ & \leq C\hspace{-2mm}\sum_{p+q+t=m}\hspace{-2mm} \|\nabla^{(p)}(\nabla v-\nabla u)\|_0
\|\nabla^{(q)}T_v\|_0\|\nabla^{(t)}E^2_u\|_0 \quad\mbox{ for all }\; m=0\ldots k.
\end{split}
\end{equation*}
Secondly, applying (\ref{goa2}) and (\ref{b}), we see that:
\begin{equation*}
\begin{split}
&\|\nabla^{(m)}(\langle E^2_u,E^1_v-E^1_u\rangle E^1_v)\|_0
\leq C\hspace{-4mm}\sum_{p+q+t=m}\hspace{-3mm} \|\nabla^{(p)}(E^1_v-E^1_u)\|_0
\big(\|\nabla^{(q)}E^1_u\|+\|\nabla^{(q)}(E^1_v-E^1_u)\|\big)\|\nabla^{(t)}E^2_u\|_0
\\ & \leq C \hspace{-4mm}
\sum_{p+t+\sum_{i=1}^mi(p_i+\bar p_i+\dbar p_i)=m\hspace{-13mm}}\hspace{-5mm}
\|\nabla^{(p)}(\nabla v -\nabla u)\|_0
\|\nabla^{(t)}E_u^2\|_0\prod_{i=1}^m \|\nabla^{(i)}(\nabla v-\nabla u)\|_0^{p_i}
\|\nabla^{(i)}T_v\|_0^{\bar p_i} \|\nabla^{(i)}E_u^1\|_0^{\dbar p_i}.
\end{split}
\end{equation*}
It follows that the same bound is valid for both terms in
$\tilde\nu_v^2-E_u^2$, so that for  $m=0\ldots k$:
\begin{equation*}
\begin{split}
&\|\nabla^{(m)}(\tilde\nu_v^2-E_u^2)\|_0
\\ & \leq C \hspace{-4mm}
\sum_{p+t+\sum_{i=1}^mi(p_i+\bar p_i+\dbar p_i)=m\hspace{-13mm}}\hspace{-5mm}
\|\nabla^{(p)}(\nabla v -\nabla u)\|_0
\|\nabla^{(t)}E_u^2\|_0\prod_{i=1}^m \|\nabla^{(i)}(\nabla v-\nabla u)\|_0^{p_i}
\|\nabla^{(i)}T_v\|_0^{\bar p_i} \|\nabla^{(i)}E_u^1\|_0^{\dbar p_i}.
\end{split}
\end{equation*}
Writing, as before:
\begin{equation*}
E^2_v-E^2_u = f(\tilde\nu^2_v) - f(E^2_u) = \Big(\int_0^1\nabla
f(t\tilde\nu^2_v+(1-t)E^2_u)~\mbox{d}t\Big) (\tilde\nu^2_v-E^2_u), 
\end{equation*}
and applying the estimate above, we arrive at:
\begin{equation*}
\begin{split}
&\|\nabla^{(m)}(E^2_v-E^2_u)\|\leq C 
\sum_{p+\sum_{i=1}^mip_i=m\hspace{-10mm}}
\|\nabla^{(p)}(\tilde\nu_v^2-E^2_u)\|_0 \prod_{i=1}^m
\big(\|\nabla^{(i)}E_u^2\|_0^{p_i} + \|\nabla^{(i)}(\tilde\nu_v^2-E^2_u)\|_0^{p_i} \big)
\\ & \leq C \hspace{-3mm}
\sum_{p+\sum_{i=1}^mi(p_i+\tbar p_i)=m\hspace{-13mm}}\hspace{-2mm}
\|\nabla^{(p)}(\tilde\nu_v^2-E^2_u)\|_0 \prod_{i=1}^m
\|\nabla^{(i)}(\tilde\nu_v^2-E^2_u)\|_0^{p_i} \|\nabla^{(i)}E_u^2\|_0^{\tbar p_i} 
\\ & \leq C \hspace{-4mm}
\sum_{p+\sum_{i=1}^mi(p_i+\bar p_i+\dbar p_i+\tbar p_i)=m\hspace{-15mm}}\hspace{-5mm}
\|\nabla^{(p)}(\nabla v-\nabla u)\|_0 \prod_{i=1}^m
\|\nabla^{(i)}(\nabla v-\nabla u)\|_0^{p_i} \|\nabla^{(i)}T_v\|_0^{\bar p_i} 
\|\nabla^{(i)}E_u^1\|_0^{\dbar p_i}\|\nabla^{(i)}E_u^2\|_0^{\tbar p_i},
\end{split}
\end{equation*}
which completes the proof of (ii) in virtue of Lemma \ref{lem_Tu} (i).

\smallskip

{\bf 3.} It remains to prove (iii). There, the assumptions and
(ii) yield for $i=1,2$ and $m=0\ldots k$:
\begin{equation*}
\begin{split}
&\|\nabla^{(m)}(E^i_v- E^i_u)\|\leq C \hspace{-4mm}
\sum_{p+\sum_{i=1}^mi(p_i+\bar p_i+\dbar p_i+\tbar p_i)=m\hspace{-5mm}}\hspace{-5mm}
\bar C\mu^p A \; \prod_{i=1}^m (\bar C\mu^i)^{p_i+\bar p_i +\dbar
  p_i}(\bar C \mu^iA)^{\tbar p_i} \leq C \bar C^{m+1}\mu^m A.
\end{split}
\end{equation*}
This ends the proof.
\end{proof}

\medskip

We close this section by deducing existence of a normal frame to any
at least $\mathcal{C}^1$-regular immersion $u$, with gradient bounds whose constants depend on the
immersability parameter $\gamma$ and on the
modulus of continuity of $\nabla u$. The mere existence of a normal frame was
essentially proved in \cite[Lemma 2.4.1]{Del_masterpieces}, by
first constructing a continuous frame through successive extensions and Gram-Schmidt
orthonormalizations, then mollifying it, and finally
projecting onto the normal bundle and applying the Gram-Schmidt procedure again. We have:

\begin{corollary}\label{lem_normals}
Let $u\in\mathcal{C}^{k+1}(\bar\omega, \R^4)$ be defined on the closure
of $\omega\subset\R^2$ diffeomorphic to $B_1$ and satisfy
(\ref{immers_gamma}) with some $\gamma\geq 1$. Then, there
exists a normal frame $E^1_u, E^2_u\in \mathcal{C}^k(\bar\omega, \R^4)$:
\begin{equation*}
(\nabla u)^T E_u^1 = (\nabla u)^T E_u^2 = 0,\quad |E^1_u|=|E^2_u|=1,
\quad \langle E^1_u, E^2_u\rangle = 0 \quad\mbox{ in } \bar\omega,
\end{equation*}
obeying the bounds:
\begin{equation}\label{stim_nor}
\|\nabla^{(m)}E_u^i\|_0\leq C \|\nabla u\|_m\quad\mbox{
  for all } \; m=1\ldots k, \; i=1,2,
\end{equation}
where $C$ depends only on $\omega, \gamma, k$, and the modulus of
continuity of $\nabla u$ on $\bar\omega$.
\end{corollary}
\begin{proof}
{\bf 1.} Observe that applying the interpolation
inequality, bounds in Lemma \ref{lem_propa} (ii) become:
\begin{equation}\label{glo3}
\begin{split}    
& \|\nabla^{(m)}(E^i_v-E^i_u)\|_0\\ & \leq C \hspace{-7mm}\sum_{p+\sum_{i=1}^mi(p_i+\bar
  p_i + \dbar p_i+\tbar p_i)=m\hspace{-10mm}} \hspace{-11mm}
\|\nabla v-\nabla u\|_0^{1-p/m}\|\nabla v-\nabla u\|_m^{p/m}
\prod_{i=1}^m \Big(\|\nabla v - \nabla u\|_0^{1-i/m}\|\nabla v -
\nabla u\|_m^{i/m}\Big)^{p_i} \times
\\ & \hspace{8cm} \times \|\nabla v\|_m^{i\bar p_i/m}
\|E^1_u\|_m^{\dbar ip_i/m} \|E^2_u\|_m^{i\tbar p_i/m}
\\ & \leq C \hspace{-2mm}\sum_{p+q+s+t=m} \hspace{-3mm}
\|\nabla v-\nabla u\|_m^{p/m}\|\nabla v\|_m^{q/m} \|E^1_u\|_m^{s/m} \|E^2_u\|_m^{t/m}
\quad \mbox{ for all }\; m=1\ldots k, \; i=1,2,
\end{split}
\end{equation}
in view of Lemma \ref{lem_det} and $\rho_\gamma\leq 1$, where $C$ depends only on $\omega, \gamma, m$.

\smallskip  

{\bf 2.} It suffices to prove the result on $\omega=B_1$. Consider the family of immersions
$\{u_j\in\mathcal{C}^{k+1}(\bar B_1 \R^4)\}_{j=1}^N$, where $u_0$ is
linear and the remaining $u_j$-s are the rescaled versions of $u$ in:
$$u_0(x) = \nabla u(0)x,\quad u_j(x) =
\frac{N}{j}u\big(\frac{jx}{N}\big) \quad\mbox{ for } j=1\ldots N,
\quad \mbox{ for all } \; x\in \bar B_1.$$
Since $\nabla u_j(x) =\nabla u(jx/N)$, we may choose $N\geq 1$ large enough to have:
$$\|\nabla u_{j+1}-\nabla u_j\|_0\leq \sup\big\{|\nabla u(x) -\nabla
u(y)|; \; |x-y|\leq \frac{1}{N}\big\}\leq \delta_\gamma\quad\mbox{ for
  all }\; j=0\ldots N-1,$$
where $N$ depends only on the modulus of continuity of $\nabla u$ and $\gamma$.
The normal frame $E^1_{u_j}. E^2_{u_j}\in\mathcal{C}^k(\bar B_1,\R^4)$
is now defined inductively via Lemma \ref{lem_propa}, starting from
the constant frame of  the linear immersion $u_0$. Since
all $u_j$ satisfy (\ref{immers_gamma}), the bound (\ref{glo3}) yields
for all $i=1,2$:
\begin{equation*}
\begin{split}    
& \|\nabla^{(m)}(E^i_{u_{j+1}}-E^i_{u_j})\|_0\\ & \leq C \hspace{-3mm}\sum_{p+q+s+t=m} \hspace{-3mm}
\|\nabla u_{j+1}-\nabla u_j\|_m^{p/m}\|\nabla u_j\|_m^{q/m} \|E^1_{u_j}\|_m^{s/m} \|E^2_{u_j}\|_m^{t/m}
\\ & \leq C \hspace{-3mm}\sum_{p+q+s+t=m} \hspace{-3mm}
\|\nabla u\|_m^{p/m}\|\nabla u\|_m^{q/m} \|E^1_{u_j}\|_m^{s/m} \|E^2_{u_j}\|_m^{t/m}
\leq C \|\nabla u\|_m \;\mbox{ for } j=0\ldots N-1, \; m=1\ldots k,
\end{split}
\end{equation*}
where $C$ depends on $\gamma, m$ and $N$. Finally, since $u_N=u$, we conclude:
$$\|\nabla^{(m)}E^i_u\|_0 \leq\sum_{j=0}^{N-1}\|\nabla^{(m)}(E^i_{u_{j+1}}-E^i_{u_j})\|_0
\leq CN \|\nabla u\|_m \quad \mbox{ for all }\; m=1\ldots k, \; i=1,2$$
which ends the proof.
\end{proof}

\noindent We also observe that if $u$ is smooth, then the
constructed normal frame is smooth as well.

\section{The pre-Stage and the initial Stage in the Nash-Kuiper scheme}\label{sec_NK0}

In this section, we show how to reduce the given positive definite
defect to a defect that is appropriately small. We provide
a self-contained proof of two assertions that we use in our future
arguments, applying the Step construction in Lemma \ref{lem_step1}.

\begin{theorem}\label{lem_prestage}\textup{[PRE-STAGE]}
Let $g\in \mathcal{C}(\bar\omega, \R^{2\times 2}_{\sym, >})$
be a Riemann metric posed on the closure of an open set $\omega\subset\R^2$ diffeomorphic to $B_1$,
and let $u\in\mathcal{C}^1(\bar\omega, \R^4)$ be an immersion, with:
$$\mathcal{D}(g, u)\doteq g-(\nabla u)^T\nabla u >0 \quad\mbox{ on } \bar\omega.$$ 
Then, there exists an immersability parameter $\underline \gamma \geq 1$
depending only on $g, u$, such that for
every $\epsilon, \delta>0$ there is $\underbar
u\in\mathcal{C}^\infty(\bar\omega,\R^4)$, satisfying:
\begin{equation*}
\begin{split}
& \frac{1}{\underline \gamma} \Id_2\leq (\nabla \underbar u)^T\nabla
\underbar u \leq \underline\gamma\Id_2 \quad\mbox{ and } \quad
\mathcal{D}(g, \underbar u)>0 \quad\mbox{ in } \,\bar\omega,\\
& \|u-\underbar u\|_0\leq \epsilon \quad\mbox{ and }\quad
\|\mathcal{D}(g, \underbar u)\|_0\leq \delta.
\end{split}
\end{equation*}
\end{theorem}

\medskip

\noindent The second result below additionally provides second-derivative
estimates obeying a specific scaling law, together with
first-derivative bounds that allow for the application of Lemma
\ref{lem_propa}, under the assumption that the original defect is
already controlled in function of the immersability parameter.
This construction was not needed in the
convex integration analysis of the Monge-Amp\`ere system in papers
\cite{lewpak_MA, lew_conv, lew_improved, lew_improved2, in_lew,
  in_lew2}, where only the decrease of $\mathcal{D}$ as in Theorem \ref{lem_stage0} mattered in
the initial Stage. Similar statement were put forward in \cite[proof of
Theorem 1.1]{CHI}, \cite[Proposition 3.2]{Cao_Sze}. 

\begin{theorem}\label{lem_stage0} \textup{[INITIAL STAGE]}
Let $\underbar u\in\mathcal{C}^\infty(\bar\omega, \R^4)$ be an
immersion, defined on the closure of an open set $\omega\subset\R^2$ diffeomorphic to $B_1$,
together with a metric $g\in \mathcal{C}^{r,\beta}(\bar\omega,
\R^{2\times 2}_{\sym, >})$ where:
$$0<r+\beta\leq 2.$$
Assume that, for some $\underline\gamma\geq \max\{1,\|g\|_0\}$, there hold:
\begin{equation*}
\begin{split}
&\frac{1}{\underline \gamma} \Id_2\leq (\nabla \underbar u)^T\nabla
\underbar u \leq \underline\gamma\Id_2 \quad\mbox{ and } \quad    
\mathcal{D}(g,\underbar u)>0 \quad\mbox{ on } \;\bar\omega.
\\ & \|\mathcal{D}(g,\underbar u)\|_0\leq \frac{\rho_{2\underline\gamma}^2}{(16)^2 N_0},
\end{split}
\end{equation*}
where $\rho_{2\underline\gamma}$ is as in Lemma \ref{lem_propa} and
$N_0$ as in Lemma \ref{lem_met_deco}.
Then, there exist  $\underline\delta \in (0,1)$ and
$\underline \tau\geq 1+\frac{1}{r+\beta}$, depending only on $\omega,
g, \underbar u, \underline\gamma $ such that the following holds. For every
$\delta\in (0,\underline\delta)$ there exists $u\in\mathcal{C}^2(\bar\omega,\R^4)$ satisfying:
\begin{align*}
& \|u-\underbar u\|_0\leq C\delta, \quad \|\nabla (u-\underbar
u)\|_0\leq \frac{\rho_{2\underline\gamma}}{3}, \quad
\|\nabla^2(u-\underbar u)\|_0\leq \frac{C}{\delta^{\underline\tau}},  
\tag*{(\theequation)$_1$}\refstepcounter{equation} \label{step0_1} \\
& \|\mathcal{D}(g-\delta H_0, u)\|_0\leq \frac{r_0}{4}\delta, 
\tag*{(\theequation)$_2$} \label{step0_2}
\end{align*}
with $r_0$ as in Lemma \ref{lem_dec_def} and with constants 
$C$ depending only on $\omega, g, \underbar u, \underline\gamma$.
\end{theorem}

\medskip

\noindent {\bf Proof of Theorem \ref{lem_prestage}}

\smallskip

{\bf 1.} We will prove the result for any fixed immersability constant $\underline\gamma$ satisfying:
$$\underline\gamma\geq
\max\{1,\|g\|_0\}\quad\mbox{ and }\quad (\nabla u)^T\nabla u\geq
\frac{2}{\underline\gamma}\Id_2\quad \mbox{ in } \bar\omega.$$
Fix $\epsilon, \delta>0$. By the density of smooth functions in
$\mathcal{C}^1$, we first find $u_0\in\mathcal{C}^\infty(\bar\omega,\R^4)$ such that:
\begin{equation}\label{0a}
\|u_0-u\|_0\leq \frac{\epsilon}{2} \quad \mbox{ and } \quad (\nabla
u_0)^T\nabla u_0\geq \frac{3}{2\underline\gamma}\Id_2,\quad
\mathcal{D}(g,u_0)\geq\frac{2}{3}\mathcal{D}(g,u)>0 \quad\mbox{in } \bar\omega. 
\end{equation}
For any sufficiently small $l\in (0,1)$, define $g_l = g\ast
\phi_l\in\mathcal{C}^\infty(\bar\omega, \R^{2\times 2}_{\sym, >})$, so that:
\begin{equation}\label{0b}
\lim_{l\to 0}\|g_l-g\|_0=0.
\end{equation}
For each $\delta\in (0, 1)$, we further consider (recalling the notion
of $H_0$ from Lemma \ref{lem_met_deco}):
$$\mathcal{D}(g_l-\delta H_0, u_0) = \mathcal{D}(g,u_0) + (g_l-g) -\delta H_0.$$
By the last assertion in (\ref{0a}), taking $l$ and $\delta_0$
sufficiently small in function of $\omega, g, u$, guarantees:
\begin{equation}\label{b3}
\mathcal{D}(g_l - \delta_0 H_0, u_0)\geq
  \frac{1}{2}\mathcal{D}(g, u)\quad\mbox{in } \bar\omega,
\end{equation}
and it also guarantees that the image of $\bar\omega$ through $\mathcal{D}(g_l - \delta_0
H_0, u_0)$ is contained in a compact region
$\mathcal{K}\subset\R^{2\times 2}_{\sym,>}$ depending only on $g, u$.
By Lemma \ref{lem_met_deco}, there exists a finite set of $N$
indices, depending on $\mathcal{K}$, for which the decomposition
of $\mathcal{D}(g_l - \delta_0H_0, u_0)$ into rank-one metrics is active:
\begin{equation*}
\begin{split}
& \mathcal{D}(g_l - \delta_0 H_0, u_0) = \sum_{i=1}^N
a_i^2\eta_i\otimes\eta_i \quad\mbox{ on } \bar\omega,\\
& \mbox{where } \{a_i\in\mathcal{C}^\infty(\bar\omega,\R)\}_{i=1}^N
\mbox{ are given by } a_i(x) = \varphi_i\Big(\mathcal{D}(g_l-\delta_0
H_0, u_0)(x)\Big) \mbox{ for all } x\in\bar\omega.
\end{split}
\end{equation*}
The coefficients $\{\varphi_i\in\mathcal{C}^\infty_c(\R^{2\times
  2}_{\sym,>}, \R)\}_{i=1}^N$ and unit  vectors
$\{\eta_i\in\R^2\}_{i=1}^N$ are as in Lemma \ref{lem_met_deco}.

\smallskip

{\bf 2.} We set $\underline u \doteq u_N$ to be the final immersion in
an inductively defined $N$-tuple of immersions $\{u_i\in\mathcal{C}^{\infty}(\bar\omega, \R^4)\}_{i=1}^N$,
where we use Lemma \ref{lem_step1}, namely:
\begin{equation}\label{def_alan}
u_{i+1}=u_i+\frac{a_{i+1}}{\lambda_{i+1}}\Gamma(\lambda_{i+1}t_{\eta_{i+1}})E^1_{u_i}
+ \frac{a_{i+1}}{\lambda_{i+1}}\bar\Gamma(\lambda_{i+1}t_{\eta_{i+1}})E^2_{u_i}
\quad\mbox{ for } i=0\ldots N-1,
\end{equation}
with $t_{\eta_{i+1}} = \langle x, \eta_{i+1}\rangle$ and with 
the orthonormal frames $\{E^1_{u_i},
E^2_{u_i}\in\mathcal{C}^{\infty}(\bar\omega, \R^4)\}_{i=0}^{N-1}$
guaranteed through Corollary \ref{lem_normals}. The $N$-tuple of frequencies
$\{\lambda_i\}_{i=1}^N$ is increasing in a sufficiently fast manner,
indicated below. In particular, we may request that:
\begin{equation}\label{0c}
\|u_{i+1}-u_i\|_0\leq 2\frac{\|a_{i+1}\|_0}{\lambda_{i+1}}\leq
2\frac{\|a_{i+1}\|_0}{\lambda_{1}}\leq\frac{\epsilon}{2N} \quad
\mbox{ for all } i=0\ldots N-1.
\end{equation}
We note that Lemma \ref{lem_step1} implies the following:
\begin{equation*}
\begin{split}
& (\nabla u_{i+1})^T \nabla u_{i+1} - (\nabla u_{i})^T \nabla u_{i} =
a_{i+1}^2\eta_{i+1}\otimes\eta_{i+1} +\mathcal{R}_{i+1},\\
& \mbox{where: }
\|\mathcal{R}_{i+1}\|_0 \leq C \Big(\frac{\|a_{i+1}\|_0}{\lambda_{i+1}} \|\nabla^{(2)}u_i\|_0
+ \frac{\|a_{i+1}\|_0}{\lambda_{i+1}} \|\nabla E_{u_i}\|_0
+ \frac{\|a_{i+1}\|^2_0}{\lambda_{i+1}^2} \|\nabla E_{u_i}\|_0^2 +
\frac{\|\nabla a_{i+1}\|^2_0}{\lambda_{i+1}^2}\Big).
\end{split}
\end{equation*}
We see that the error $\mathcal{R}_{i+1}$ is as small as needed, provided that $\lambda_{i+1}$ is large.
Consequently:
\begin{equation}\label{b5}
\begin{split}  
(\nabla u_{i+1})^T \nabla u_{i+1} & = (\nabla u_0)^T \nabla u_0 +
\sum_{j=1}^{i+1}a_j^2\eta_j\otimes\eta_j  + \sum_{j=1}^{i+1}\mathcal{R}_j \geq 
(\nabla u_0)^T \nabla u_0 + \sum_{j=1}^{i+1}\mathcal{R}_j \\ & \geq (\nabla u_0)^T \nabla u_0 -
\Big(\sum_{i=1}^{N}\|\mathcal{R}_i \|_0\Big)\Id_2\geq
\frac{1}{\underline\gamma}\Id_2 \quad\mbox{ on } \bar\omega,
\end{split}
\end{equation}
where we used (\ref{0a}). In particular, every $\{u_i\}_{i=1}^N$ is an immersion. Further, writing:
\begin{equation*}
\begin{split} 
\mathcal{D}(g,\underline u) & = \mathcal{D}(g_l-\delta_0 H_0, u_N)
+ (g-g_l) + \delta_0H_0 \\ & = \mathcal{D}(g_l-\delta_0 H_0, u_0)
- \sum_{i=0}^{N-1}\big((\nabla u_{i+1})^T \nabla u_{i+1} - (\nabla u_{i})^T \nabla u_{i}\big)
+ (g-g_l) + \delta_0H_0 \\ & = - \sum_{i=1}^{N}\mathcal{R}_i + (g-g_l) + \delta_0H_0,
\end{split}
\end{equation*}
we get that $\|\mathcal{D}(g,\underline u)\|_0\leq \delta$ when
$l$ and $\delta_0$ are sufficiently small and the progression of
$\lambda_i$-s is sufficiently fast. Similarly, 
$\mathcal{D}(g,\underline u)\geq \frac{\delta_0}{2}H_0>0$ in
$\bar\omega$ is guaranteed by taking $l$ small and the progression of
$\lambda_i$-s fast, with respect to the pre-fixed $\delta_0$. This
last property also yields:
$$(\nabla \underline u)^T\nabla \underline u\leq g\leq
\|g\|_0\Id_2\leq \gamma\Id_2 \quad\mbox{ in } \;\bar\omega.$$
Finally, we obtain:
$$\|u-\underline u\|_0\leq
\|u- u_0\|_0+\sum_{i=0}^{N-1}\|u_{i+1}-u_i\|_0\leq \epsilon,$$
by the first assertion in (\ref{0a}) and (\ref{0c}). The proof is done.
\endproof

\medskip

We now refine the proof of Theorem \ref{lem_prestage} to obtain more
detailed estimates made possible by the assumed smallness of the
defect in function of the immersability parameter.

\medskip

\noindent {\bf Proof of Theorem \ref{lem_stage0}}

\smallskip

{\bf 1.} For any $l\in (0,1)$ and $\delta\in (0, \underline \delta)$, we write:
$$\mathcal{D}(g_l-\delta H_0, \underline u) = \mathcal{D}(g,\underline
u) + (g_l-g) -\delta H_0.$$
where \ref{stima2}, \ref{stima1} yield the following estimates: 
\begin{equation}\label{a1}
\|g_l-g\|_0\leq l^{r+\beta}\|g\|_{r,\beta}, \quad
\|\nabla^{(m)}g_l\|_0\leq C_m\frac{1}{l^m}\|g\|_0,
\end{equation}
with constant $C_m$ depending only on $m$ and $\omega$.
Taking $l$ and $\underline\delta$ sufficiently small, guarantees:
\begin{equation}\label{b3}
\mathcal{D}(g_l - \delta H_0,\underline u)\geq
  \frac{1}{2}\mathcal{D}(g,\underline u)\quad\mbox{in } \bar\omega,
\end{equation}
as in (\ref{b3}). Similarly, the image of $\bar\omega$ through $\mathcal{D}(g_l - \delta
   H_0,\underline u)$ is contained in a compact region
   $\mathcal{K}\subset\R^{2\times 2}_{\sym,>}$ depending only on
   $g,\underline u$, with the decomposition: 
\begin{equation*}
\begin{split}
& \mathcal{D}(g_l - \delta H_0, \underline u) = \sum_{i=1}^N
a_i^2\eta_i\otimes\eta_i \quad\mbox{ on } \bar\omega,\\
& \mbox{where } \{a_i\in\mathcal{C}^\infty(\bar\omega,\R)\}_{i=1}^N
\mbox{ are given by } a_i(x) = \varphi_i\Big(\mathcal{D}(g_l-\delta
H_0, \underline u)(x)\Big) \mbox{ for all } x\in\bar\omega.
\end{split}
\end{equation*}
Above, the number $N$ of the active decomposition terms,
depends only on $\mathcal{K}$ and thus on $g, \underline u$.
By Fa\`a di Bruno's formula and the second bound in (\ref{a1}), we obtain:
\begin{equation}\label{441}
\begin{split}
& \|\nabla^{(m)} a_i\|_0 \leq C\hspace{-4mm}\sum_{p_1+2p_2+\ldots mp_m=m}
\hspace{-4mm} \|\nabla^{(p_1+\ldots +p_m)
}\varphi_i\|_0\prod_{j=1}^m\|\nabla^{(j)}\mathcal{D}(g_l - \delta H_0,\underline u) \|^{p_j}
\\ & \leq C \hspace{-4mm}\sum_{p_1+\ldots mp_m=m}
\prod_{j=1}^m\Big(\frac{\|g\|_{0}}{l^j} +1\Big)^{p_j}\leq
\frac{C}{l^m} \quad\mbox{ for all } m=1\ldots N-1, \;i=1\ldots N,
\end{split}
\end{equation}
with $C$ that depends only on $N,\omega, \underline u$, hence on $\omega, g, \underline u$.
Recalling Lemma \ref{lem_met_deco} (i), we allso observe the
following bound, valid for $l, \delta$ small in function of $\gamma, N_0$:
\begin{equation}\label{b1}
\begin{split}
& 2\|\sum_{i=1}^N a_i\|_0\leq 2N_0^{1/2}\big\|\sum_{i=1}^Na_i^2\big\|_0^{1/2}=
2N_0^{1/2}\big\|\mbox{Trace}\;\mathcal{D}(g_l-\delta H_0,\underbar u) \big\|_0^{1/2}
\\ & \leq 2(2N_0)^{1/2}\|\mathcal{D}(g_l-\delta H_0,\underbar u)\|_0^{1/2}
\leq 2(2N_0)^{1/2}\big(\|\mathcal{D}(g,\underline u)\|_0+\|g_l-g\|_0 +
|\delta H_0|\big)^{1/2}\\ & \leq 4 N_0^{1/2}
\frac{\rho_{2\underline\gamma}}{16N_0^{1/2}}\leq \frac{\rho_{2\underline\gamma}}{4}.
\end{split}
\end{equation}

\smallskip

{\bf 2.} We set $u_0=\underline u$ and  inductively define the $N$-tuple of
immersions $\{u_i\in\mathcal{C}^{N-i+2}(\bar\omega, \R^4)\}_{i=1}^N$
by the formula (\ref{def_alan}), where each orthonormal frame $\{E^1_{u_i},
E^2_{u_i}\in\mathcal{C}^{N-i+1}(\bar\omega, \R^4)\}_{i=0}^{N-1}$ is
given through an application of Lemma \ref{lem_propa} in view of the
assertions \ref{za0}, \ref{za5} that will be recursively verified
below. The frequencies $\{\lambda_i\}_{i=1}^N$ are defined as follows:
\begin{equation}\label{b2}
\lambda_i = \frac{1}{(\epsilon l)^i} \quad \mbox{ for all } i=0\ldots N, 
\end{equation}
where $\epsilon\in (0,1)$ is sufficiently small, in function of $\omega, g, \underline u$.
We will show that there holds:
\begin{align*}
& \frac{1}{2\underline\gamma}\Id_2\leq (\nabla u_i)^T\nabla 
u_i\leq 2\underline \gamma \,\Id_2
\; \mbox{ in } \bar\omega\quad \mbox{ for  } i=0\ldots N,
\tag*{(\theequation)$_1$}\refstepcounter{equation} \label{za0} \\
& \|\nabla^{(m)}u_i\|_0\leq C\lambda_i^{m-1}\quad 
\quad \mbox{for } i=0\ldots N-1,\; m=1\ldots N-i+2, \tag*{(\theequation)$_2$} \label{za1}\\
& \|\nabla^{(m)}E_{u_i}^j\|_0\leq C\lambda_i^{m}\quad 
\quad \mbox{ for } i=0\ldots N-1,\; m=0\ldots N-i+1,\; j=1,2, \tag*{(\theequation)$_3$} \label{za2}\\
& \|\nabla^{(m)}(u_{i+1}-u_i)\|_0\leq C\lambda_{i+1}^{m-1}\qquad 
\mbox{ for } i=0\ldots N-1,\; m=0\ldots N-i+1. \tag*{(\theequation)$_4$} \label{za3}\\
& \|\nabla(u_{i+1}-u_i)\|_0\leq \rho_{2\underline\gamma}\hspace{2.05cm}
\mbox{for } i=0\ldots N-1.\tag*{(\theequation)$_4$} \label{za5}
\end{align*}
Firstly, \ref{za0}, \ref{za1}, \ref{za2} are trivially true at $i=0$ where
$\lambda_0=1$. Secondly, \ref{za2} directly implies \ref{za3} and
\ref{za5}. Indeed, for \ref{za3}, we observe that:
\begin{equation*}
\begin{split}
\|\nabla^{(m)}(u_{i+1}-u_i)\|_0 &\leq C \hspace{-2mm}\sum_{p+q+t=m}
\hspace{-2mm} \lambda_{i+1}^{p-1}\|\nabla^{(q)}a_{i+1}\|_0
\big(\|\nabla^{(t)}E_{u_i}^1\|_0+\|\nabla^{(t)}E_{u_i}^2\|_0\big) 
\\ & \leq C \hspace{-2mm}\sum_{p+q+t=m}
\hspace{-2mm} \lambda_{i+1}^{p-1}\frac{\lambda_i^t}{l^q}
\leq  C\lambda_{i+1}^{m-1} \sum_{p+q+t=m}
\hspace{-2mm} \frac{1}{(l\lambda_{i+1})^{q}}\leq C\lambda_{i+1}^{m-1}, 
\end{split}
\end{equation*}
by (\ref{441}), (\ref{b1}) and since $l\lambda_{i+1}\geq l\lambda_1=\frac{1}{\epsilon}>1$.
For \ref{za5}, we utilize (\ref{441}),  (\ref{b1}) in:
\begin{equation*}
\begin{split}
& \|\nabla (u_{i+1}-u_i)\|_0\leq 2\|a_{i+1}\|_0+ 2\frac{\|\nabla a_{i+1}\|_0}{\lambda_{i+1}}
+ \frac{\|a_{i+1}\|_0}{\lambda_{i+1}}\big(\|\nabla E_{u_i}^1\|_0 +
\|\nabla E_{u_i}^2\|_0\big)\\ & \leq \frac{\rho_{2\underline\gamma}}{4} +
C\big(\frac{1}{l\lambda_{i+1}} +\frac{\lambda_i}{\lambda_{i+1}}\big)
\leq \frac{\rho_{2\underline\gamma}}{4} + C\epsilon\leq \rho_{2\underline\gamma}.
\end{split}
\end{equation*}
Thirdly, Lemma \ref{lem_propa} (ii) implies in virtue of \ref{za0} - \ref{za5}, that:
\begin{equation*}
\begin{split}
\|\nabla^{(m)}(E_{u_{i+1}}^j - E_{u_i}^j)\|_0 & \leq C \hspace{-4mm}
\sum_{p+\sum_{s=1}^m s(p_s+\bar p_s+\dbar p_s+\tbar p_s)=m\hspace{-4mm}}
\lambda_{i+1}^p \prod_{s=1}^m\lambda_i^{s(p_s+\bar p_s)}\lambda_{i+1}^{s(\dbar p_s+\tbar p_s)}
\\ & \leq C\lambda_{i+1}^m \quad\mbox{ for all } \; m=0\ldots N-i, \; j=1,2.
\end{split}
\end{equation*}
In conclusion, \ref{za1} - \ref{za3} together with the
above bound, imply \ref{za1}, \ref{za2} at $i+1$:
\begin{equation*}
\begin{split}
& \|\nabla^{(m)} u_{i+1}\|_0 \leq \|\nabla^{(m)}u_i\|_0 + \|\nabla^{(m)}(u_{i+1}-u_i)\|_0 \leq
C(\lambda_i^{m-1}+\lambda_{i+1}^{m-1}) \\ & \hspace{21mm} \leq C \lambda_{i+1}^{m-1}
\quad \mbox{ for } m=1\ldots N-i+1,
\\ & \|\nabla^{(m)}E_{u_{i+1}}^j\|_0 \leq
\|\nabla^{(m)}E_{u_{i}}^j\|_0 + \|\nabla^{(m)}(E_{u_{i+1}}^j -
E_{u_i}^j)\|_0 \leq C(\lambda_i^{m}+\lambda_{i+1}^{m})
\\ & \hspace{23mm} \leq C \lambda_{i+1}^{m} \quad \mbox{ for } m=0\ldots N-i,\;j=1,2.
\end{split}
\end{equation*}
We now fix $i=0\ldots N-1$ and show that the validity of \ref{za1},
\ref{za2} for all $j=0\ldots i$ implies \ref{za0} at $i+1$. To this
end, recall the form of the error $\mathcal{R}$ in Lemma \ref{lem_step1} and estimate:
\begin{equation*}
\begin{split}
& \|(\nabla u_{j+1})^T \nabla u_{j+1} - (\nabla u_{j})^T \nabla u_{j} -
a_{j+1}^2\eta_{j+1}\otimes\eta_{j+1}\|_0 = \|\mathcal{R}\|_0
\\ & \leq C \Big(\frac{\|a_{j+1}\|_0}{\lambda_{j+1}} \|\nabla^{(2)}u_j\|_0
+ \frac{\|a_{j+1}\|^2_0}{\lambda_{j+1}} \|\nabla E_{u_j}\|_0
+ \frac{\|a_{j+1}\|^2_0}{\lambda_{j+1}^2} \|\nabla E_{u_j}\|_0^2 +
\frac{\|\nabla a_{j+1}\|^2_0}{\lambda_{j+1}^2}\Big)  \\ & 
\leq C\Big(\frac{\lambda_j}{\lambda_{j+1}}+ \frac{\lambda_j^2}{\lambda_{j+1}^2}+
\frac{1}{(\lambda_{j+1}l)^2}\Big) \leq C\big(\epsilon l +(\epsilon l)^2 +\epsilon^2\big)
\leq C \epsilon \quad \mbox{ for all } j=0\ldots i.
\end{split}
\end{equation*}
Consequently:
\begin{equation}\label{b5} 
\|(\nabla u_{i+1})^T \nabla u_{i+1} - (\nabla u_0)^T \nabla u_0 -
\sum_{j=1}^{i+1}a_j^2\eta_j\otimes\eta_j \|_0\leq C\epsilon, 
\end{equation}
and further, taking $\epsilon$ sufficiently small yields in $\bar\omega$:
\begin{equation*}
\begin{split} 
& (\nabla u_{i+1})^T \nabla u_{i+1} \geq (\nabla u_0)^T \nabla u_{0} +
\sum_{j=1}^{i+1}a_j^2\eta_j\otimes\eta_j  - C\epsilon\Id_2 \geq 
\frac{1}{\underline\gamma}\Id_2 - C\epsilon\Id_2 \geq 
\frac{1}{2\underline\gamma} \Id_2,
\\ &  (\nabla u_{i+1})^T \nabla u_{i+1} \leq (\nabla u_0)^T \nabla u_{0} +
\sum_{j=1}^{N}a_j^2\eta_j\otimes\eta_j  + C\epsilon\Id_2 \\ &
\qquad\qquad\qquad \;\; \; \leq
(\nabla\underline u)^T\nabla \underline u + \mathcal{D}(g_l-\delta
H_0, \underline u) + C\epsilon\Id_2 \\ &
\qquad\qquad\qquad \;\; \; = g+ (g_l-g) -\delta H_0 + 
C\epsilon\Id_2 \leq 2\|g\|_0\Id_2 \leq 2\underline\gamma\Id_2.
\end{split}
\end{equation*}
This ends the proof of \ref{za0} and the proof of all the inductive estimates.

\smallskip

{\bf 3.} We declare $u\cdot u_N$ and from \ref{za3} deduce that:
\begin{equation}\label{b4}
\|\nabla^{(m)}(u-\underline u)\|_0\leq \sum_{i=0}^{N-1}\|\nabla^{(m)}(u_{i+1}-u_i)\|_0
\leq C \sum_{i=0}^{N-1}\lambda_{i+1}^{m-1}\quad\mbox{ for } m=0,1,2. 
\end{equation}
On the other hand, from (\ref{a1}) and (\ref{b5}), we get:
\begin{equation*}
\begin{split} 
& \|\mathcal{D}(g-\delta H_0, u)\|_0\leq \|\mathcal{D}(g_l-\delta H_0, u_N)\|_0
+ \|g-g_l\|_0 \\ & = \big\|\sum_{j=1}^Na_j^2\eta_j\otimes\eta_j
-\big((\nabla u_N)^T\nabla u_N - (\nabla u_0)^T\nabla u_0\big)\big\|_0
+ \|g-g_l\|_0 \leq \bar C(\epsilon +l^{r+\beta}),
\end{split}
\end{equation*}
where $\bar C>1$ depends only on $\omega, \underline u, g$. We hence take:
$$ \epsilon = \frac{r_0}{8\bar C}\delta , \qquad l^{r+\beta}=
\frac{r_0}{8\bar C}\delta$$
which is consistent with the requirements of smallness of $\epsilon$
and $l$ provided that $\underline\delta$ is small. This implies
\ref{step0_2}. Finally, (\ref{b4}) and (\ref{b1}) yield:
\begin{equation*}
\begin{split} 
& \|u-\underline u\|_0\leq \frac{C}{\lambda_1} \leq C\epsilon \leq C\delta \\
&\|\nabla(u-\underline u)\|_0\leq\sum_{i=1}^N\|\nabla(u_{i-1}-u_i)\|_0\leq 
2\big\|\sum_{i=1}^Na_i\big\|_0+C\epsilon \leq \frac{\rho_{2\underline\gamma}}{3}
\\ & \|\nabla^{(2)}(u-\underline u)\|_0\leq C\lambda_N
= \frac{C}{(\epsilon l)^N}\leq C \delta^{-(1+\frac{1}{r+\beta})N},
\end{split}
\end{equation*}
resulting in \ref{step0_1} with $\underline \tau =
(1+\frac{1}{r+\beta})N \geq 1+\frac{1}{r+\beta}$, as claimed. The
proof is done.
\endproof

\section{The oscillatory defect decomposition lemmas}\label{sec_IBP}

The three results in this section provide a
symmetric matrix field decomposition, complementary to that in Lemma
\ref{lem_dec_def}. It removes a further symmetric gradient from
the given oscillatory component of a defect at hand, and reduces it to
a defect of higher order in the frequency, plus another term that
agrees in the frequency yet has a lower dimensionality rank. 
The lemmas below are the explicit versions of the
argument called "integration by parts" in \cite{CHI2}. The precise
formulas on the decomposition coefficients (\ref{coef_ibp1}),
(\ref{coef_ibp3}), (\ref{coef_ibp2}) are of crucial importance
in closing the estimates in Theorem \ref{thm_STA}.
First, we quote from \cite[Lemma 2.4, Corollary 2.5]{in_lew2}:

\begin{lemma}\label{lem_IBP1}
Given  $H\in\mathcal{C}^{k+1}(\R^2,\R^{2\times 2}_\sym)$, $\lambda>0$,
and $\Gamma_0\in\mathcal{C}(\R,\R)$, we have the decomposition:
\begin{equation}\label{ibp1}
\begin{split}
\frac{\Gamma_0(\lambda x_1)}{\lambda}H = & \; (-1)^{k+1}\frac{\Gamma_{k+1}(\lambda
  x_1)}{\lambda^{k+2}} \sym\nabla L_k^{\eta_1}  \\ & + \sym\nabla
\Big(\sum_{i=0}^k(-1)^i\frac{\Gamma_{i+1}(\lambda x_1)}{\lambda^{i+2}}L_i^{\eta_1}\Big) 
+ \Big( \sum_{i=0}^{k}(-1)^i\frac{\Gamma_i(\lambda
  x_1)}{\lambda^{i+1}} P_i^{\eta_1}\Big) e_2\otimes e_2
\end{split}
\end{equation}
where the functions $\Gamma_i\in \mathcal{C}^{i}(\R,\R)$ satisfy the recursive definition:
\begin{equation}\label{recu_ibp}
\Gamma_{i+1}' = \Gamma_{i}  \quad\mbox{ for all } \; i =0\ldots k,
\end{equation}
while $L_i^{\eta_1}\in\mathcal{C}^{k+1-i}(\R^2,\R^2)$ and
$P_i^{\eta_1}\in\mathcal{C}^{k+1-i}(\R^2,\R)$ are given in:
\begin{equation}\label{coef_ibp1}
\begin{split}
& \;\; \, L_0^{\eta_1} = (H_{11}, 2H_{12}), \qquad P_0^{\eta_1} = H_{22},\\
& \left.\begin{array}{l} L_i^{\eta_1}= (\partial_1^{(i)}H_{11}, 2\partial_1^{(i)}H_{12} +
i\partial_1^{(i-1)}\partial_2H_{11}), \vspace{2mm} 
\\  P_i^{\eta_1} = 2\partial_1^{(i-1)}\partial_2H_{12} +
(i-1)\partial_1^{(i-2)}\partial_2^{(2)}H_{11} \end{array}\right\}
\quad \mbox{ for all }\; i=1\ldots k.
\end{split}
\end{equation}
\end{lemma}

\begin{lemma}\label{lem_IBP3}
Let $H$, $\lambda$, $\Gamma_0$ be as in Lemma \ref{lem_IBP1} and
$\{\Gamma_i\in \mathcal{C}^{i}(\R,\R)\}_{i=1}^{k+1}$ as in (\ref{recu_ibp}).
Then:
\begin{equation}\label{ibp3}
\begin{split}
\frac{\Gamma_0(\lambda x_2)}{\lambda}H = & \; (-1)^{k+1}\frac{\Gamma_{k+1}(\lambda
  x_2)}{\lambda^{k+2}} \sym\nabla  L_k^{\eta_3}  \\ & + \sym\nabla
\Big(\sum_{i=0}^k(-1)^i\frac{\Gamma_{i+1}(\lambda
  x_2)}{\lambda^{i+2}}L_i^{\eta_3}\Big) 
+ \Big( \sum_{i=0}^{k}(-1)^i\frac{\Gamma_i(\lambda
  x_2)}{\lambda^{i+1}} P_i^{\eta_3}\Big) e_1\otimes e_1,
\end{split}
\end{equation}
with $L_i^{\eta_3}\in\mathcal{C}^{k+1-i}(\R^2,\R^2)$,
$P_i^{\eta_3}\in\mathcal{C}^{k+1-i}(\R^2,\R)$ given in:
\begin{equation}\label{coef_ibp3}
\begin{split}
& \;\; \;  L_0^{\eta_3} = (2H_{12}, H_{22}), \qquad P_0^{\eta_3} = H_{11},\\
& \left.\begin{array}{l} L_i^{\eta_3}= (2\partial_2^{(i)}H_{12} +
i\partial_1\partial_2^{(i-1)}H_{22}, \partial_2^{(i)}H_{22}), \vspace{2mm} 
\\  P_i^{\eta_3} = 2\partial_1\partial_2^{(i-1)}H_{12} +
(i-1)\partial_1^{(2)}\partial_2^{(i-2)}H_{22} \end{array}\right\}
\quad \mbox{ for all }\; i=1\ldots k.
\end{split}
\end{equation}
\end{lemma}

\medskip

There likewise holds, with respect to the oscillations in the
$\eta_2=\frac{e_1+e_2}{\sqrt{2}}$ spatial direction and the
residue accumulating in the $e_2\otimes e_2$ component of the
matrix field, as in Lemma \ref{lem_IBP1}:

\begin{lemma}\label{lem_IBP2}
Let $H$, $\lambda$, $\Gamma_0$ be as in Lemma \ref{lem_IBP1} and
$\{\Gamma_i\in \mathcal{C}^{i}(\R,\R)\}_{i=1}^{k+1}$ as in (\ref{recu_ibp}).
Denote:
$$t=\langle x, \eta_2\rangle.$$
Then, we have:
\begin{equation}\label{ibp2}
\begin{split}
\frac{\Gamma_0(\lambda t)}{\lambda}H = & \; (-1)^{k+1}\frac{\Gamma_{k+1}(\lambda
  t)}{\lambda^{k+2}} \sym\nabla L_k^{\eta_2}  \\ & + \sym\nabla
\Big(\sum_{i=0}^k(-1)^i\frac{\Gamma_{i+1}(\lambda
  t)}{\lambda^{i+2}} L_i^{\eta_2}\Big) 
+ \Big( \sum_{i=0}^{k}(-1)^i\frac{\Gamma_i(\lambda t)}{\lambda^{i+1}} P_i^{\eta_2}\Big) e_2\otimes e_2,
\end{split}
\end{equation}
with $L_i^{\eta_2}\in\mathcal{C}^{k+1-i}(\R^2,\R^2)$,
$P_i^{\eta_2}\in\mathcal{C}^{k+1-i}(\R^2,\R)$ given in:
\begin{equation}\label{coef_ibp2}
\begin{split}
& \;\; \; L_0^{\eta_2} = \sqrt{2}(H_{11}, 2H_{12}-H_{11}), \qquad P_0^{\eta_2} = H_{22}-2H_{12}+H_{11},\\
& \left.\begin{array}{l} L_i^{\eta_2}= 2^{(i+1)/2}\big(\partial_1^{(i)}H_{11}, 
2\partial_1^{(i)}H_{12} + i \partial_1^{(i-1)}\partial_2H_{11}- (i+1)\partial_i^{(i)}H_{11}\big), \vspace{2mm} 
\\  P_i^{\eta_2} = {2}^{i/2}\big(2\partial_1^{(i-1)}\partial_2H_{12}-2\partial_1^{(i)}H_{12} +
(i-1)\partial_1^{(i-2)}\partial_2^{(2)}H_{11}\vspace{2mm} 
\\ \quad\qquad -2i\partial_1^{(i-1)}\partial_2H_{11}+
(i+1)\partial_i^{(i)}H_{11} \big) \end{array}\right\}
\mbox{ for }\; i=1\ldots k.
\end{split}
\end{equation}
\end{lemma}

\begin{proof}
{\bf 1.} Observe that $\sym(L_0^{\eta_2}\otimes \eta_2)=
H-P_0^{\eta_2} e_2\otimes e_2$ and that: 
$$ \sym\nabla \Big(\frac{\Gamma(\lambda t)}{\lambda^2} L_0^{\eta_2}\Big)
= \frac{\Gamma(\lambda t)}{\lambda^2}\sym\nabla L_0^{\eta_2} +
\frac{\Gamma'(\lambda t)}{\lambda}\sym(L_0^{\eta_2}\otimes \eta_2).$$
Applying the above with $\Gamma=\Gamma_1$ so that $\Gamma'=\Gamma_0$,
we get (\ref{ibp2}) at $k=0$:
\begin{equation}\label{naua}
\frac{\Gamma_0(\lambda t)}{\lambda}H = -\frac{\Gamma_{1}(\lambda
  t)}{\lambda^{2}} \, \sym\nabla L_0^{\eta_2}  + \sym\nabla
\Big(\frac{\Gamma_{1}(\lambda t)}{\lambda^{2}} L_0^{\eta_2}\Big) 
+ \frac{\Gamma_0(\lambda t)}{\lambda} P_0^{\eta_2} e_2\otimes e_2.
\end{equation}

\smallskip

{\bf 2.} The proof of (\ref{ibp2}) is carried out by induction on $k$. 
Assume that it holds at some $k\geq 0$ and apply (\ref{naua}) to
$H=\sym\nabla L_k^{\eta_2}$ and $\Gamma_0$ replaced by $\Gamma_{k+1}$, to get:
\begin{equation}\label{naua1}
\begin{split}
\frac{\Gamma_{k+1}(\lambda t)}{\lambda}\,\sym\nabla L_k
 = & -\frac{\Gamma_{k+2}(\lambda
  t)}{\lambda^{2}} \, \sym\nabla L_{k+1}  \\ & + \sym\nabla
\Big(\frac{\Gamma_{k+2}(\lambda t)}{\lambda^{2}} L_{k+1}\Big) 
+ \frac{\Gamma_{k+1}(\lambda t)}{\lambda} P_{k+1} e_2\otimes e_2,
\end{split}
\end{equation}
where, in virtue of the definitions (\ref{coef_ibp2}):
\begin{equation*} 
\begin{split}
L_{k+1} & =  \sqrt{2}(\partial_1(L_k^{\eta_2})_1, \partial_1(L_k^{\eta_2})_2
 + \partial_2(L_k^{\eta_2})_1 - \partial_1(L_k^{\eta_2})_1 \\ &
= 2^{(k+2)/2}\big(\partial_1^{(k+1)}H_{11}, 
2\partial_1^{(k+1)}H_{12} + (k+1) \partial_1^{(k)}\partial_2H_{11}- (k+2)\partial_i^{(k+1)}H_{11})
= L_{k+1}^{\eta_2},\\
P_{k+1} & = \partial_2(L_k^{\eta_2})_2-\partial_1(L_k^{\eta_2})_2-\partial_2(L_k^{\eta_2})_1
+ \partial_1(L_k^{\eta_2})_1 \\ &
=  {2}^{(k+1)/2}\big(2\partial_1^{(k)}\partial_2H_{12}-2\partial_1^{(k+1)}H_{12} +
k\partial_1^{(k-1)}\partial_2^{(2)}H_{11} \\ & \qquad\qquad \quad -2(k+1)\partial_1^{(k)}\partial_2H_{11}+
(k+2)\partial_i^{(k+1)}H_{11} \big) = P_{k+1}^{\eta_2}.
\end{split}
\end{equation*}
Introduce now (\ref{naua1}) into (\ref{ibp2}) to obtain:
\begin{equation*} 
\begin{split}
\frac{\Gamma_0(\lambda t)}{\lambda}H & \; = 
\frac{(-1)^{k+1}}{\lambda^{k+1}} \Big(-\frac{\Gamma_{k+2}(\lambda
  t)}{\lambda^{2}} \, \sym\nabla L_{k+1}^{\eta_2}  \\ & \qquad\qquad
\qquad + \sym\nabla \Big(\frac{\Gamma_{k+2}(\lambda t)}{\lambda^{2}} L_{k+1}^{\eta_2}\Big) 
+ \frac{\Gamma_{k+1}(\lambda t)}{\lambda} P_{k+1}^{\eta_2} e_2\otimes e_2\Big)
\\ & \quad\, + \sym\nabla
\Big(\sum_{i=0}^k(-1)^i\frac{\Gamma_{i+1}(\lambda
  t)}{\lambda^{i+2}} L_i^{\eta_2}\Big) 
+ \Big( \sum_{i=0}^{k}(-1)^i\frac{\Gamma_i(\lambda t)}{\lambda^{i+1}} P_i^{\eta_2}\Big) e_2\otimes e_2,
\\ & =(-1)^{k+2}\frac{\Gamma_{k+2}(\lambda
  t)}{\lambda^{k+3}} \sym\nabla L_{k+1}^{\eta_2}  \\ & \quad + \sym\nabla
\Big(\sum_{i=0}^{k+1}(-1)^i\frac{\Gamma_{i+1}(\lambda t)}{\lambda^{i+2}} L_i^{\eta_2}\Big) 
+ \Big( \sum_{i=0}^{k+1}(-1)^i\frac{\Gamma_i(\lambda t)}{\lambda^{i+1}} P_i^{\eta_2}\Big) e_2\otimes e_2,
\end{split}
\end{equation*}
which is exactly (\ref{ibp2}) at $k+1$, as claimed.
\end{proof}

\section{The K\"all\'en iteration technique}\label{sec_kallen}

In this section, we carry out a version of K\"all\'en's
iteration, with the purpose of canceling the non-oscillatory  portion
of the defect term $\frac{1}{\lambda^2}\nabla a\otimes\nabla a$ in
Lemma \ref{lem_step2}. The remaining portion: 
$$\frac{\Gamma(\lambda t)^2
-1}{\lambda^2} S_1= -\frac{\cos(2\lambda t)}{\lambda^2}
\nabla a\otimes\nabla a,$$ 
where we note that $\cos(2\lambda t)$ has mean zero on its period,
will be canceled in the leading order via lemmas in section \ref{sec_IBP}. The matrix field $H$
in the statement below should be thought of as the scaled defect
$\mathcal{D}$. Similar result appeared in \cite[Proposition
3.1]{in_lew2}, with only one extra term of the type
$\frac{1}{\lambda^2}\nabla a\otimes\nabla a$, while \cite[Lemma
2.2]{CHI2} featured more absorbed terms, as below. 

\begin{lemma}\label{prop1}
Let $H\in \mathcal{C}^\infty(\bar\omega,\R^{2\times 2}_\sym)$ be
defined on the closure of an open, bounded set $\omega\subset\R^2$,
and let $M, N\geq 1$ be two integers. Assume that:
\begin{equation}\label{ass_H0}
\|H-H_0\|_0\leq \frac{r_0}{2} \quad\mbox{ and }\quad 
\|\nabla^{(m)}H\|_0\leq \bar C\mu^m \quad \mbox{for all } \; m=1\ldots M+N,
\end{equation}
 for some given $\mu, \bar C>1$ and with $r_0$ as in Lemma \ref{lem_dec_def}.
Then, there exists $\underline \sigma >2$ depending only on $M, N$
such that the following holds. Given the constants 
$\kappa>\lambda>\mu$ satisfying $\lambda/\mu\geq \underline\sigma$,
there exist $\{a_i\in \mathcal{C}^\infty(\bar\omega,\R)\}_{i=1}^3$
such that, writing:
$$H= \sum_{i=1}^3 a_i^2\eta_i\otimes\eta_i + \frac{1}{\lambda^2}\nabla
a_1\otimes\nabla a_1 + \frac{1}{\kappa^2}\nabla a_2\otimes\nabla a_2 + \mathcal{F},$$
there hold the estimates:
\begin{equation}\label{Ebounds}
\begin{split}
&  \frac{1}{2}\leq a_i^2\leq \frac{3}{2} \; \; \quad \qquad \mbox{ and }\quad 
\frac{1}{2}\leq a_i\leq \frac{3}{2} \quad\mbox{ in }\bar\omega \quad \mbox{
for } i=1\ldots 3,\\
& \|\nabla^{(m)} a_i^2\|_0\leq C\mu^m\quad 
\mbox{and } \quad \|\nabla^{(m)} a_i\|_0\leq C\mu^m \;\quad\mbox{ for }\;
 m=1\ldots M+1,\;  i=1\ldots 3,\\
& \|\nabla^{(m)}\mathcal{F}\|_0\leq
C\frac{\mu^m}{(\lambda/\mu)^{2N}}\quad\mbox{ for } \; m=0\ldots M.  
\end{split}
\end{equation}
The constant $C$ above depends only on $ M, N, \bar C$.
\end{lemma}

\begin{proof}
{\bf 1. }  We will inductively define the triples of 
coefficients $\{a_{i,j}\in\mathcal{C}^\infty(\bar\omega,
\R)\}_{j=0}^N$, $i=1\ldots 3$, by setting $a_{i,0}\equiv 0$ and
further utilizing Lemma \ref{lem_dec_def} and the linear maps $\bar a_i$ in there:
\begin{equation}\label{def_aP}
\begin{split}
& a_{i,j}= \big(\bar a_i(H-\mathcal{E}_{j-1})\big)^{1/2} \quad \mbox{
  for } j=1\ldots N,\; i=1\ldots 3,\\ 
&\mbox{where } \quad \mathcal{E}_j = \frac{1}{\lambda^2}\nabla a_{1,j}\otimes\nabla a_{1,j}
+ \frac{1}{\kappa^2}\nabla a_{2,j}\otimes\nabla a_{2,j}
\quad \mbox{ for } j=1 \ldots N.
\end{split}
\end{equation}
This means that, with fixed unit vectors $\{\eta_i\}_{i=1}^3$:
$$\sum_{i=1}^3a_{i,j}^2\eta_i\otimes\eta_i = H-\mathcal{E}_{j-1} \quad \mbox{ for } j=1 \ldots N,$$
and we also denote:
$$\mathcal{F}_j = H- \sum_{i=1}^3a_{i,j}^2\eta_i\otimes\eta_i -
\frac{1}{\lambda^2}\nabla a_{1,j}\otimes\nabla a_{1,j} 
- \frac{1}{\kappa^2}\nabla a_{2,j}\otimes\nabla a_{2,j} = \mathcal{E}_{j-1}-\mathcal{E}_j.$$
We will show that the above decomposition is well posed and that it yields:
\begin{equation}\label{asm4}
a_i = a_{i, N} \quad\mbox{ for } i=1\ldots 3 \quad \mbox{ and }\quad
\mathcal{F}=\mathcal{F}_N
\end{equation}
with the desired properties (\ref{Ebounds}). To this end, we will inductively show for all $j=1\ldots N$:
\begin{align*}
& \; \, \frac{1}{2}\leq a_{i,j}^2\leq \frac{3}{2} \; \; \quad \qquad
\mbox{ and }\quad 
\frac{1}{2}\leq a_{i,j}\leq \frac{3}{2} \quad\mbox{ in }\bar\omega \quad \mbox{
for } i=1\ldots 3,
\tag*{(\theequation)$_1$}\refstepcounter{equation} \label{Ebound12}\\
& \begin{array}{ll} \displaystyle{\|\nabla^{(m)} a_{i,j}^2\|_0\leq C\mu^m\quad 
\mbox{and } \quad \|\nabla^{(m)} a_{i,j}\|_0\leq C\mu^m } \vspace{2mm }\\ 
\qquad \qquad\qquad\qquad \qquad\qquad 
\mbox{ for } m=1\ldots M+N-j+1,\;  i=1\ldots 3, \end{array}
\tag*{(\theequation)$_2$}\label{Ebound22} \\ 
& \; \, \|\nabla^{(m)}\mathcal{F}_j\|_0\leq
C\frac{\mu^m}{(\lambda/\mu)^{2j}} \qquad\mbox{for } m=0\ldots M
+ N-j, \tag*{(\theequation)$_3$} \label{Ebound32}
\end{align*}
with $C$ that depends only on $M,N$.

\smallskip

{\bf 2. } We analyze the induction base at $j=1$. By the first
condition in (\ref{ass_H0}), all $a_{i,1}^2=\bar a_i(H)$ for
$i=1\ldots 3$ are well-defined and \ref{Ebound12} holds by Lemma
\ref{lem_dec_def}. Further:
$$\|\nabla^{(m)}a_{i,1}^2\|_0\leq C\|\nabla ^{(m)}H\|_0\leq
C\mu^m\quad\mbox{ for } m=1\ldots M+N,\; i=1\ldots 3,$$
by the second condition in (\ref{ass_H0}). Observe  that the
first conditions in \ref{Ebound12}, \ref{Ebound22} always imply
the respective second conditions in there, because:
\begin{equation*}
\begin{split}
& |a_{i,j} - 1|=\frac{|a_{i,j}^2 - 1|}{a_{i,j} + 1}\leq |a_{i,j}^2 -
1|\leq \frac{1}{2} \quad\mbox{in } \bar\omega,\;j=1\ldots N, \; i=1\ldots 3,\\
& \|\nabla^{(m)} a_{i,j}\|_0\leq C\sum_{p_1+2p_2+\ldots
  mp_m=m} \Big\|
a_{i,j}^{2(1/2-p_1-\ldots p_m)}\prod_{t=1}^m\big|\nabla^{(t)}a_{i,j}^2\big|^{p_t}\Big\|_0
\leq C\mu^m, \\ & %\qquad\qquad\qquad\qquad \qquad\qquad \qquad\qquad 
\mbox{for } j=1\ldots N, \; m=1\ldots M+N-j+1,\; i=1\ldots 3,
\end{split}
\end{equation*}
by an application of Fa\'a di Bruno's formula. This proves, in
particular, \ref{Ebound12}, \ref{Ebound22} at $j=1$, whereas \ref{Ebound32}  holds by:
\begin{equation*}
\begin{split}
\|\nabla^{(m)}\mathcal{F}_1\|_0&\leq
\frac{C}{\bar\lambda^2}\sum_{p+q=m} \big(\|\nabla^{(p+1)}a_{1,1}\|_0 \|\nabla^{(q+1)}a_{1,1}\|_0
+ \|\nabla^{(p+1)}a_{2,1}\|_0 \|\nabla^{(q+1)}a_{2,1}\|_0\big)\\ & \leq
C\frac{\mu^m}{(\lambda/\mu)^2}.
\end{split}
\end{equation*}
This ends the justification of the induction base.
We additionally note that applying Fa\'a di Bruno's formula to the
inverse rather than the square root, \ref{Ebound22} yields: 
\begin{equation}\label{asm2}
\begin{split}
& \|\nabla^{(m)}\Big(\frac{1}{a_{i,j}+a_{i,j-1}}\Big)\|_0  \\ & \leq C\hspace{-3mm} \sum_{p_1+2p_2+\ldots
  mp_m=m}\hspace{-8mm}\big\|(a_{i,j}+a_{i,j-1})^{-1-p_1-\ldots p_m}
\prod_{t=1}^m|\nabla^{(t)}(a_{i,j}+a_{i,j-1})|^{p_t}\big\|_0\\ & \leq C\mu^m
\quad\mbox{for } j=1\ldots N, \; m=0\ldots M+N-j+1,\; i=1\ldots 3.
\end{split}
\end{equation}

\smallskip

{\bf 3.} We now exhibit the induction step. Assume that
\ref{Ebound12}-\ref{Ebound32} have been proved up to some counter $j=1\ldots
N-1$. To check that $\{a_{i,j+1}\}_{i=1}^3$ are well defined, we need
that $\|(H-\mathcal{E}_j)-H_0\|_0\leq r_0$, which indeed holds with
$\underline \sigma$ large, because:
\begin{equation*}
\begin{split}
\|(H-\mathcal{E}_j)-H_0\|_0 & \leq \|H-H_0\|_0 + \sum_{t=1}^j\|\mathcal{E}_t-\mathcal{E}_{t-1}\|_0
\leq\frac{r_0}{2} + \sum_{t=1}^j\|\mathcal{F}_t\|_0 \\ & \leq 
\frac{r_0}{2} + \sum_{t=1}^\infty\frac{1}{(\lambda/\mu)^{2t}}\leq 
\frac{r_0}{2} + \frac{C}{(\lambda/\mu)^{2}},
\end{split}
\end{equation*}
where we used that $1-\frac{1}{(\lambda/\mu)^2}\geq2$ valid from
$(\lambda/\mu)^2\geq {\underline \sigma}^2 \geq 2$.
This proves \ref{Ebound12} at the counter $j+1$. We similarly get \ref{Ebound22} from:
\begin{equation*}
\begin{split}
\|\nabla^{(m)}a_{i,j+1}^2\|_0 & \leq C
\big(\|\nabla^{(m)}H\|_0+\|\nabla^{(m)}\mathcal{E}_j\|_0\big)\leq
C\Big(\mu^m + \sum_{t=1}^j\|\nabla^{(m)}\mathcal{F}_t\|_0 \Big) \\ & \leq
C\Big(\mu^m + \frac{\mu^m}{(\lambda/\mu)^2)}\Big)\leq C\mu^m.
\end{split}
\end{equation*}
It remains to show \ref{Ebound32} at $j+1$. We write:
\begin{equation}\label{asm3}
\begin{split}
& \mathcal{F}_{j+1} =\mathcal{E}_j-\mathcal{E}_{j+1} \\ & = \frac{1}{\bar\lambda^2}\big(\nabla
a_{1,j}\otimes\nabla a_{1,j} - \nabla a_{1,j+1}\otimes\nabla a_{1,j+1}\big) 
+ \frac{1}{\bar\kappa^2}\big(\nabla
a_{2,j}\otimes\nabla a_{2,j} - \nabla a_{2,j+1}\otimes\nabla a_{2,j+1}\big) 
\\ & = \frac{1}{\bar\lambda^2}\big(\nabla( a_{1,j} -
a_{1,j+1})\otimes\nabla a_{1,j} + \nabla a_{1,j+1}\otimes \nabla (a_{1,j} - a_{1,j+1})\big) 
\\ & \qquad + \frac{1}{\bar\kappa^2}\big(\nabla( a_{2,j} -
a_{2,j+1})\otimes\nabla a_{2,j} + \nabla a_{2,j+1}\otimes \nabla (a_{2,j} - a_{2,j+1})\big) 
\end{split}
\end{equation}
and recall that $\mathcal{F}_j = (H-\mathcal{E}_j) - (H-\mathcal{E}_{j-1}) =
\sum_{i=1}^3\big(a_{i,j+1}^2-a_{i,j}^2\big) \eta_i\otimes \eta_i$. This yields:
\begin{equation*}
a_{i,j+1}^2-a_{i,j}^2=\bar a_i(\mathcal{F}_j) \quad \mbox{ for } i=1\ldots 3
\end{equation*}
and further, by \ref{Ebound32}:
\begin{equation*}
\begin{split}
& \|\nabla^{(m)}(a_{i,j+1}^2-a_{i,j}^2)\|_0\leq C
\|\nabla^{(m)}\mathcal{F}_j\|_0 \leq
C\frac{\mu^m}{(\lambda/\mu)^{2j}}\quad  \\ & \mbox{for all } m=0\ldots M+N-j,\; i=1\ldots 3.
\end{split}
\end{equation*}
In view of (\ref{asm2}), we thus obtain:
\begin{equation*}
\begin{split}
& \|\nabla^{(m)}(a_{i,j+1}-a_{i,j})\|_0 =
\Big\|\nabla^{(m)}\big(\frac{a_{i,j+1}^2-a_{i,j}^2}{a_{i,j+1}-a_{i,j}}\big)\Big\|_0
\\ & \leq C\sum_{p+q=m} \|\nabla^{(p)}(a_{i,j+1}^2-a_{i,j}^2)\|_0
\|\nabla^{(q)}\Big(\frac{1}{a_{i,j}+a_{i,j-1}}\Big)\|_0  
\leq C \sum_{p+q=m} \frac{\mu^p}{(\lambda/\mu)^{2j}} \mu^q \\ & \leq C
\frac{\mu^m}{(\lambda/\mu)^{2j}} \quad\mbox{ for all } m=0\ldots M+N-j,\;i=1\ldots 3.
\end{split}
\end{equation*}
In conclusion, by (\ref{asm3}):
\begin{equation*}
\begin{split}
& \|\nabla^{(m)}\mathcal{F}_{j+1}\|_0 \leq
C \sum_{p+q=m} \frac{1}{\lambda^2} \frac{\mu^{p+1}}{(\lambda/\mu)^{2j}} \mu^{q+1} \\ & \leq C
\frac{\mu^m}{(\lambda/\mu)^{2(j+1)}} \quad\mbox{ for all } m=0\ldots M+N-(j+1).
\end{split}
\end{equation*}
This ends the proof of all the inductive estimates \ref{Ebound12} -
\ref{Ebound32} and therefore of the desired bounds (\ref{Ebounds}) for (\ref{asm4}). 
The proof is done.
\end{proof}

\section{Stage construction: the proof of Theorem \ref{thm_STA}}

This section is devoted to our main construction, following its sketch
 in subsection \ref{stage_sketch}.

\bigskip

\noindent {\bf Proof of Theorem \ref{thm_STA}}

\smallskip

{\bf 1. (Mollification and initial bounds)} Let $g, \underline u, \underline\gamma,
N,K$ be in the statement of the theorem, and let
$\delta,\mu,\sigma, u$ satisfy \ref{Ass1}-\ref{Ass3}. All the
bounds in the course of the proof below will be valid under the assumption
that $\underline\delta$ is sufficiently small and $\underline\sigma$
sufficiently large, in function of $\underline\gamma, \omega, g,N,K$.
We define the following initial parameters and the mollified fields:
\begin{equation*}
\begin{split}
& \mu_0 =\mu, \quad \delta_0=\delta, \quad  l=\frac{1}{C\mu},\\
& u_0=u\ast\phi_l\in\mathcal{C}^\infty(\bar\omega,\R^4),\quad 
g_0=g\ast\phi_l\in\mathcal{C}^\infty(\bar\omega,\R^{2\times 2}_{\sym}),
\end{split}
\end{equation*}
where $C>1$ above is an independent constant, assuring that:
\begin{equation}\label{m1}
\|(\nabla u_0)^T\nabla u_0 - \big((\nabla u)^T\nabla u\big)\ast
\phi_l\|_0\leq C_0l^2\|\nabla^2u\|_0^2\leq
\frac{C_0}{C^2\mu^2}\delta\mu^2\leq \frac{r_0}{12}\delta_0,
\end{equation}
which follows from the second assumption in \ref{Ass3} and by applying \ref{stima4}
to $f=g=\nabla u$. Further application of Lemma \ref{lem_stima} yields:
\begin{align*}
& \|g_0-g\|_0\leq l^{r+\beta}\|g\|_{r,\beta}\leq \frac{\|g\|_{r,\beta}}{\mu_0^{r+\beta}},
\tag*{(\theequation)$_1$}\refstepcounter{equation} \label{m2} \\
% & \|\nabla^{(m)}g_0\|_0\leq
% \frac{C}{l^m}\|g\|_0\leq C\mu_0^m \quad\mbox{ for all }m=0\ldots 2+
% (3N+6)K, NOT ~ NEED??
% \tag*{(\theequation)$_2$} \label{mG0} \\
& \|u_0-u\|_1\leq l\|u\|_2\leq \delta_0^{1/2},
\tag*{(\theequation)$_2$} \label{m3} \\
& \|\nabla^{(m)}\nabla^{(2)}u_0\|_0\leq \frac{C}{l^m}\|\nabla^2
u\|_0\leq C\mu_0^{m+1} \delta_0^{1/2} \quad\mbox{for all } m=0\ldots (3N+6)K.
\tag*{(\theequation)$_3$} \label{m4} 
\end{align*}
In particular, $u_0$ is an immersion, because by Lemma \ref{lem_det}:
\begin{equation*}
\begin{split}
&  \|(\nabla u_0)^T\nabla u_0 - (\nabla u)^T\nabla u\|_0\leq 
\|\nabla(u_0-u)\|_0\big(\|\nabla u_0\|_0 +\|\nabla u\|_0 \big)\\ & \leq
\|\nabla(u_0-u)\|_0\big(2\|\nabla u\|_0 +\|\nabla (u_0-u)\|_0 \big)\leq
\delta_0^{1/2}\big(2 (4\gamma)^{1/2} + \delta_0^{1/2}\big),
\end{split}
\end{equation*}
so when $\delta_0\leq \underline \delta$ is sufficiently small, we obtain:
\begin{equation}\label{m5}
\|\nabla u_0-\nabla \underline u\|_0\leq \rho_{\underline\gamma}
\quad \mbox{ and } \quad \frac{1}{3\underline \gamma}\Id_2\leq (\nabla u_0)^T\nabla u_0\leq
3\underline\gamma\Id_2 \quad\mbox{ in }\bar\omega.
\end{equation}
Consequently, we may apply Lemma \ref{lem_Tu} to $u_0$ and $\mu_0$
with $A=\delta_0^{1/2}$ and obtain the following bound for the tangent
field $T_{u_0}=(\nabla u_0)\big((\nabla u_0)^T\nabla u_0\big)^{-1}$:
\begin{equation}\label{m6}
\|T_{u_0}\|_0\leq C, \qquad \|\nabla^{(m)}T_{u_0}\|_0\leq
C \delta_0^{1/2} \mu_0^m \quad\mbox{for all } m=1\ldots (3N+6)K -1.
\end{equation}
Corollary \ref{lem_normals} implies existence of a normal frame
$E_{\underline u}^1, E_{\underline u}^2\in \mathcal{C}^{\infty}(\bar\omega,\R^4)$, 
so by Lemma \ref{lem_propa} there exists a normal frame 
$E_{u_0}^1, E_{u_0}^2\in \mathcal{C}^{(3N+6)K+1}(\bar\omega,\R^4)$, namely:
$$(\nabla u_0)^TE_{u_0}^i=0,\quad |E_{u_0}^i|=1 \quad \mbox{ for }
i=1,2 \quad\mbox{ and } \quad \langle E_{u_0}^1,E_{u_0}^2\rangle =
0\quad\mbox{in }\bar\omega,$$
obeying the bounds:
\begin{equation}\label{m7}
\begin{split}
& \|\nabla^{(m)}E^i_{u_0}\|_0\leq C(\|\nabla^{(m)}E^i_{\underline u}\|_0+
\|\nabla \underline u\|_m + \|\nabla u_0\|_m) \leq 
C\delta_0^{1/2}\mu_0^m\quad\\ & \mbox{for all } m=1\ldots (3N+6)K+1, \; i=1,2,
\end{split}
\end{equation}
where we used \ref{m4}. In particular, we get:
\begin{equation}\label{m8}
\begin{split}
& \|\nabla^{(m)}((\nabla u_0)^T\nabla E^i_{u_0})\|_0 \leq
C\sum_{p+q=m}\|\nabla^{(p+1)}u_0\|_0 \|\nabla^{(q+1)}E_{u_0}^i\|_0\\ &
\leq C \sum_{p+q=m} \mu_0^{p}\mu_0^{q+1}\delta_0^{1/2}
\leq C \delta_0^{1/2} \mu_0^{m+1}\quad\mbox{for all }m=0\ldots (3N+6)K,\; i=1,2.
\end{split}
\end{equation}
Finally, writing:
$$\mathcal{D}(g_0-\delta_0H_0,u_0) = \mathcal{D}(g-\delta_0H_0,u)\ast\phi_l
+ \Big((\nabla u_0)^T\nabla u_0 - \big((\nabla u)^T\nabla u\big)\ast
\phi_l\Big),$$
we get in view of (\ref{m1}), \ref{Ass3} and Lemma \ref{lem_stima}:
\begin{equation}\label{m9}
\begin{split}
& \|\mathcal{D}(g_0-\delta_0H_0,u_0)\|_0\leq \|\mathcal{D}(g-\delta_0H_0,u)\|_0
+\frac{r_0}{12}\delta_0\leq \frac{r_0}{4}\delta_0
+\frac{r_0}{12}\delta_0= \frac{r_0}{3}\delta_0,\\
& \|\nabla^{(m)}\mathcal{D}(g_0-\delta_0H_0,u_0)\|_0 \,\leq 
\frac{C}{l^m}\|\mathcal{D}(g-\delta_0H_0,u)\|_0 + \frac{C}{l^{m-2}}
\|\nabla^2u\|_0^2\\ & \qquad\qquad\qquad \qquad\qquad \quad
\leq C\delta_0 \mu_0^m\quad\mbox{for all } m=1\ldots (3N+6)K+1.
\end{split}
\end{equation}

\smallskip

{\bf 2. (Setting up the induction)} In the course of the proof, we will define the immersions
$\{u_k\in\mathcal{C}^{2+(3N+6)(K-k)}(\bar\omega,\R^4)\}_{k=1}^K$ and their
normal vectors $E_{u_k}^1, E_{u_k}^2\in \mathcal{C}^{1+(3N+6)(K-k)}
(\bar\omega,\R^4)$:
$$(\nabla u_k)^TE_{u_k}^i=0,\quad |E_{u_k}^i|=1 \quad \mbox{ for }
i=1,2 \quad\mbox{ and } \quad \langle E_{u_k}^1,E_{u_k}^2\rangle =
0\quad\mbox{in }\bar\omega,$$
with respect to the increasing progression of frequencies
$\{\mu_k\}_{k=1}^N$ and the decreasing progression of interpolating
defect magnitudes $\{\delta_k\}_{k=1}^K$ given in:
\begin{equation}\label{md_def}
\mu_k = \mu_0\sigma^{N+2}\sigma^{(\frac{N}{2}+2)(k-1)} =
\mu_0\sigma^{2k+ \frac{N}{2}(k+1)} \quad\mbox{ and } \quad \delta_{k}
= \frac{\delta_0}{\sigma^{kN}} \quad\mbox{for } k=1\ldots K.\\ 
\end{equation}
We note in passing that $\{\delta_k^{1/2}\mu_{k+1}\}_{k=0}^{K-1}$
is an increasing $K$-tuple of numbers bigger than $1$, as:
\begin{equation}\label{bas_incr}
\frac{\delta_k^{1/2}\mu_{k+1}}{\delta_{k-1}^{1/2}\mu_{k}} \geq 
\frac{\sigma^{\frac{N}{2}+2}}{\sigma^{\frac{N}{2}}} = 
\sigma^2\geq 1 \quad\mbox{ for } k=1\ldots K-1.
\end{equation}
We will inductively prove the satisfaction of the following properties, valid for all $k=1\ldots K$:
\begin{align*}
& \frac{1}{3\underline\gamma}\Id_2\leq (\nabla u_k)^T\nabla u_k\leq 3\underline\gamma\Id_2
\quad\mbox{in } \bar\omega, \tag*{(\theequation)$_1$}\refstepcounter{equation} \label{P1} \\
& \|u_k-u_{k-1}\|_1\leq C \delta_{k-1}^{1/2},
\tag*{(\theequation)$_2$} \label{P2} \\
& \|\nabla^{(m+1)}(u_k- u_{k-1})\|_0\leq C\delta_{k-1}^{1/2}
\mu_k^m\quad\mbox{for all } m=0\ldots 1+(3N+6)(K-k),
\tag*{(\theequation)$_3$} \label{P3} \\
& \|\nabla^{(m)}(E^i_{u_k}- E^i_{u_{k-1}})\|_0\leq
C\delta_{k-1}^{1/2} \mu_k^m\quad\mbox{for all } m=0\ldots 1+(3N+6)(K-k),\; i=1,2,
\tag*{(\theequation)$_4$} \label{P4} \\
& \|\nabla^{(m)}\mathcal{D}(g_0-\delta_kH_0,u_k)\|_0\leq \frac{\delta_{k-1}}{\sigma^{N+1/2}}\mu_k^m
\quad\mbox{for all } m=0\ldots 1+(3N+6)(K-k),
\tag*{(\theequation)$_5$} \label{P5}
\end{align*}
together with the following specific bounds on the components of the second derivatives:
\begin{align*}
& \|\nabla^{(m)}\langle \partial_{ij}u_k, E^1_{u_k}\rangle\|_0 \leq
C\frac{\delta_{k-1}^{1/2}}{\sigma^{N/2}} \mu_k^{m+1}\; \quad\mbox{for }
m=0\ldots (3N+6)(K-k),\; i,j=1\ldots 2,
\tag*{(\theequation)$_1$}\refstepcounter{equation} \label{Q1} \\
& \|\nabla^{(m)}\langle \partial_{11}u_k, E^2_{u_k}\rangle\|_0 \leq
C\frac{\delta_{k-1}^{1/2}}{\sigma^{N}} \mu_k^{m+1}\; \quad\mbox{for }
m=0\ldots (3N+6)(K-k),
\tag*{(\theequation)$_2$} \label{Q2} \\
& \|\nabla^{(m)}\langle \partial_{12}u_k, E^2_{u_k}\rangle\|_0 \leq
C\frac{\delta_{k-1}^{1/2}}{\sigma^{N/2}} \mu_k^{m+1} \hspace{0.5mm}
\quad\mbox{for } m=0\ldots (3N+6)(K-k),
\tag*{(\theequation)$_3$} \label{Q3} \\
& \|\nabla^{(m)}\langle \partial_{22}u_k, E^2_{u_k}\rangle\|_0 \leq
C{\delta_{k-1}^{1/2}}\mu_k^{m+1}\quad\;\,\mbox{for } m=0\ldots (3N+6)(K-k).
\tag*{(\theequation)$_4$} \label{Q4}
\end{align*}
The above will yield the desired estimates by setting:
$$\tilde u = u_K\in\mathcal{C}^2(\bar\omega,\R^4)$$
Indeed, the bounds in \ref{Res1} follow from \ref{m3}, \ref{P2} and
then from \ref{m4}, \ref{P3}, (\ref{bas_incr}):
\begin{equation*}
\begin{split}
& \|u_K-u\|_1\leq \|u_0-u\|_1 + \sum_{k=0}^{K-1}\|u_{k+1}-u_k\|_1 \leq 
C \sum_{k=0}^{K-1}\delta_{k}^{1/2}\leq C\delta_0^{1/2},\\
& \|u_K\|_2\;\leq \|u_0\|_2 + \sum_{k=0}^{K-1}\|u_{k+1}-u_k\|_2\leq 
C\delta_0^{1/2} \mu_0 + C\sum_{k=0}^{K-1}\delta_k^{1/2}\mu_{k+1}\leq C
\delta_{K-1}^{1/2}\mu_{K} \\ & \qquad\quad 
= C\mu_0\delta_0^{1/2} \sigma^{2K+\frac{N}{2}(K+1)-\frac{N}{2}(K-1)}=
C\mu_0\delta_0^{1/2} \sigma^{2K+N},
\end{split}
\end{equation*}
in view of the definition of frequencies and defect measures in (\ref{md_def}),
whereas \ref{Res2} follows from \ref{m2} and \ref{P5}, by setting
$\sigma\geq \underline\sigma$ sufficiently large:
\begin{equation*}
\begin{split}
& \|\mathcal{D}(g-\delta_KH_0,u_K)\|_0\leq \|g_0-g\|_0 + \|\mathcal{D}(g_0-\delta_KH_0,u_K)\|_0
\\ & \leq \frac{\|g\|_{r+\beta}}{\mu_0^{r+\beta}} + \frac{\delta_{K-1}}{\sigma^{N+1/2}} = 
\frac{\|g\|_{r+\beta}}{\mu_0^{r+\beta}} + \frac{1}{\sigma^{N(K-1)}}
\frac{\delta_{0}}{\sigma^{N+1/2}}  = \frac{\|g\|_{r+\beta}}{\mu_0^{r+\beta}} 
+ \frac{r_0}{5} \frac{\delta_{0}}{\sigma^{NK}}.
\end{split}
\end{equation*}

\smallskip

{\bf 3. (Decomposition of the $k$-th defect)} We now fix the counter: 
$$k=0\ldots K-1$$ 
and carry out the construction of $u_{k+1}$. We either have $k=0$
and the estimates \ref{m3}-(\ref{m9}) in step 1, or we have
$k=1\ldots K-1$ and then we assume \ref{P1}-\ref{Q4}. We proceed by
applying the K\"allen decomposition in Theorem \ref{prop1} to the
following parameters:
$$H=\frac{\mathcal{D}(g_0-\delta_{k+1}H_0,u_k)}{\delta_k}, \quad
\mu=\mu_k, \quad \lambda = \mu_k\sigma,\quad \kappa=\mu_k\sigma^2.$$

\begin{figure}[htbp]
\centering
\includegraphics[scale=0.6]{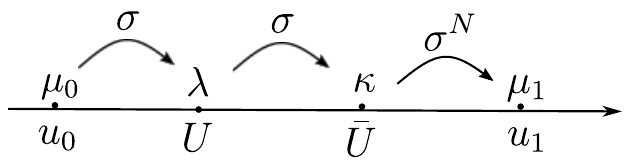} \hspace{7mm}
\includegraphics[scale=0.6]{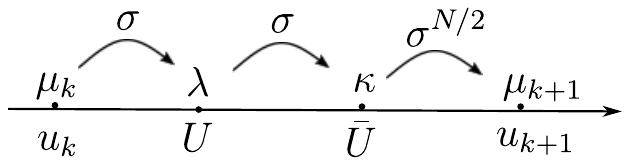}
\caption{{Progression of frequencies and the corresponding
    intermediate fields $U$, $\bar U$ in the
    three corrugation modification from $u_k$ to $u_{k+1}$. Note the
  distinction of cases $k=0$ and $k=1\ldots K-1$.}}
\label{fig_freq2}
\end{figure}

\noindent We need to verify the two assumptions in (\ref{ass_H0}). For the first
one, we write:
$$\|H-H_0\|_0 = \frac{\|\mathcal{D}(g_0-\delta_{k+1}H_0,u_k) - \delta_kH_0\|_0}{\delta_k}
\leq \frac{\|\mathcal{D}(g_0-\delta_{k}H_0,u_k) \|_0}{\delta_k} +
\frac{\delta_{k+1}}{\delta_k}|H_0|\leq \frac{r_0}{2},$$
which is valid provided that:
\begin{equation*}
\|\mathcal{D}(g_0-\delta_{k}H_0,u_k) \|_0\leq \frac{r_0}{3}\delta_k
\quad\mbox{ and }\quad \frac{\delta_{k+1}}{\delta_k}|H_0|\leq \frac{r_0}{6}.
\end{equation*}
It is clear that the second bound above holds by 
(\ref{md_def}) if we set $\sigma\geq \underline\sigma$ 
sufficiently large. The first bound holds at $k=0$ by (\ref{m9}), whereas for
$k=1\ldots K-1$ it holds by \ref{P5} for large $\underline\sigma$:
$$\|\mathcal{D}(g_0-\delta_k H_0, u_k)\|_0\leq
\frac{\delta_{k-1}}{\sigma^{N+1/2}} = \frac{\delta_k}{\sigma^{1/2}}\leq \frac{r_0}{4}\delta_k.$$
We now verify the second assumption in (\ref{ass_H0}):
$$\|\nabla^{(m)}H\|_0= \frac{\|\nabla^{(m)}\mathcal{D}(g_0-\delta_{k}H_0,u_k\|_0)}{\delta_k}
\leq C\mu_k^m \quad\mbox{ for all } m=1\ldots 1+(3N+6)(K-k),$$
following for $k=0$ from (\ref{m9}) and for $k=1\ldots K-1$ from \ref{P5}.
Assuring that $\underline\sigma$ is larger than $\underline\sigma$
required in Theorem \ref{prop1}, it yields the decomposition:
\begin{equation}\label{m10}
\mathcal{D}(g_0-\delta_{k+1}H_0,u_k) =
\sum_{i=1}^3a_i^2\eta_i\otimes\eta_i + \frac{1}{\lambda^2}\nabla a_1\otimes\nabla a_1
+ \frac{1}{\kappa^2}\nabla a_2\otimes\nabla a_2 + \mathcal{F},
\end{equation}
with the bounds resulting from (\ref{Ebounds}):
\begin{equation}\label{Ebounds2}
\begin{split}
&  \frac{1}{2}\delta_k\leq a_i^2\leq \frac{3}{2}\delta_k \; \; \quad
\quad \, \mbox{ and }\quad 
\frac{1}{2}\delta_k^{1/2}\leq a_i\leq \frac{3}{2} \delta_k^{1/2}\quad\mbox{ in }\bar\omega \quad \mbox{
for } i=1\ldots 3,\\
& \|\nabla^{(m)} a_i^2\|_0\leq C\delta_k\mu_k^m\quad 
\mbox{and } \quad \|\nabla^{(m)} a_i\|_0\leq C\delta_k^{1/2}\mu_k^m \\
&\qquad\qquad\qquad\qquad \qquad\quad \mbox{ for }\;
 m=1\ldots 2+(3N+6)(K-k)-N,\;  i=1\ldots 3,\\
& \|\nabla^{(m)}\mathcal{F}\|_0\leq
C\frac{\delta_k}{\sigma^{2N}}\mu_k^m\quad\,\mbox{ for } \; m=0\ldots 1+(3N+6)(K-k)-N.
\end{split}
\end{equation}
We end this step by gathering bounds including \ref{m4}, (\ref{m6}) and (\ref{m7}):
\begin{align*}
& \|\nabla u_k\|_0\leq (6\underline\gamma)^{1/2}
\tag*{(\theequation)$_1$}\refstepcounter{equation} \label{s1} \\
& \|\nabla^{(m)}\nabla^{(2)} u_ k\|_0\leq C\delta_{k-1}^{1/2}\mu_k^{m+1}=
C\delta_{k}^{1/2}\mu_k^{m+1}\sigma^{N/2} \quad\mbox{for all } m=0\ldots (3N+6)(K-k),
\tag*{(\theequation)$_2$} \label{s2} \\
& \|T_{u_k}\|_0\leq C \quad\mbox{ and }\quad
\|\nabla^{(m)}T_{u_k}\|_0\leq C\mu_k^m\delta_k^{1/2}\sigma^{N/2}\\
& \qquad \qquad\qquad \qquad \qquad\qquad \; \,\quad\mbox{for all } m=1\ldots 1+ (3N+6)(K-k), 
\tag*{(\theequation)$_3$} \label{s3} \\
& \|\nabla^{(m)}E^i_{u_k}\|_0\leq
C\mu_k^m\delta_k^{1/2}\sigma^{N/2}\quad\mbox{for all } m=1\ldots 1+ (3N+6)(K-k),\;i=1,2. 
\tag*{(\theequation)$_4$} \label{s4} 
\end{align*}
Indeed, \ref{s1} follows by Lemma \ref{lem_det} and (\ref{m5}) and \ref{P1}. Then \ref{s2} follows by
\ref{P3}, (\ref{bas_incr}) and (\ref{md_def})
when $k\geq 1$, whereas its final bound follows directly by
\ref{m4} for $k=0$. Similarly, \ref{s4} follows from \ref{P4}.
Applying Lemma \ref{lem_Tu} with
$A=\delta_k^{1/2}\sigma^{N/2}\leq \delta_0^{1/2}\sigma^{N/2}\leq 1$ in
view of the last assumption in \ref{Ass1}, yields \ref{s3}. 

\smallskip

{\bf 4. (The first corrugation)} We define the intermediary field
$U\in\mathcal{C}^{1+(3N+6)(K-k)-2N}(\bar\omega, \R^4)$:
\begin{equation}\label{corru1}
U=u_k+\frac{\Gamma(\lambda x_1)}{\lambda}a_1E^1_{u_k} +
T_{u_k}\Big(\frac{\bar\Gamma(\lambda x_1)}{\lambda}a_1^2 e_1+ W\Big),
\end{equation}
in accordance with Lemma \ref{lem_step2}, where $\Gamma$, $\bar\Gamma$ are the oscillatory
profiles in there, namely with:
\begin{equation}\label{m11}
\Gamma(t) = \sqrt{2}\sin t ,\qquad\bar\Gamma(t) =
-\frac{1}{4}\sin(2t) \quad\mbox{ and } \quad\dbar\Gamma(t) \doteq\Gamma(t)^2-1= -\cos(2t).
\end{equation}
The oscillation direction in (\ref{corru1}) is set to $\eta=\eta_1=e_1$,  the frequency is:
$$\lambda = \mu_k\sigma,$$ 
and the tangential correction $W$ is given recalling the definition
(\ref{coef_ibp1}) and the recursion (\ref{recu_ibp}):
\begin{equation*}
\begin{split}
-2W = & \sum_{i=0}^N (-1)^i \frac{\dbar\Gamma_{i+1}(\lambda x_1)}{\lambda^{i+3}}L_i^{\eta_1}(S_1) 
+ \sum_{i=0}^N (-1)^i \frac{(\Gamma'\Gamma)_{i+1}(\lambda x_1)}{\lambda^{i+2}}L_i^{\eta_1}(S_2) 
\\ & + \sum_{i=0}^N (-1)^i \frac{\Gamma_{i+1}(\lambda x_1)}{\lambda^{i+2}}L_i^{\eta_1}(S_3) 
+ \sum_{i=0}^N (-1)^i \frac{\bar\Gamma_{i+1}(\lambda x_1)}{\lambda^{i+2}}L_i^{\eta_1}(S_4). 
\end{split}
\end{equation*}
Above, the fields $\{S_i\}_{i=1}^4$ are as in Lemma \ref{lem_step2}, namely:
\begin{equation}\label{S_first}
\begin{split}
& S_1= \nabla a_1\otimes\nabla a_1, \qquad\qquad\qquad \qquad\quad
S_2 = 2a_1\sym(\nabla a_1\otimes e_1),\\
& S_3 = 2a_1\sym\big((\nabla u_k)^T\nabla E^1_{u_k}\big), \qquad\qquad
S_4 =2\,\sym\big((\nabla u_k)^T\nabla (a_1^2 T_{u_k}e_1)\big).
\end{split}
\end{equation}
By (\ref{ibp1}) in Lemma \ref{lem_IBP1} we see that:
\begin{equation}\label{symW}
\begin{split}
& -2\,\sym\nabla W = \frac{\dbar\Gamma(\lambda x_1)}{\lambda^2} S_1 
+ \frac{(\Gamma'\Gamma)(\lambda x_1)}{\lambda} S_2 
+ \frac{\Gamma(\lambda x_1)}{\lambda} S_3+ \frac{\bar\Gamma(\lambda
  x_1)}{\lambda} S_4  - \mathcal{G} - G e_2\otimes e_2,
\end{split}
\end{equation}
with the following formulas:
\begin{equation*}
\begin{split}
& \mathcal{G} = (-1)^{N+1} \frac{\dbar\Gamma_{N+1}(\lambda
  x_1)}{\lambda^{N+3}}\sym \nabla L_N^{\eta_1}(S_1) 
+ (-1)^{N+1} \frac{(\Gamma'\Gamma)_{N+1}(\lambda
  x_1)}{\lambda^{N+2}}\sym\nabla L_N^{\eta_1}(S_2) 
\\ & \qquad + (-1)^{N+1} \frac{\Gamma_{N+1}(\lambda
  x_1)}{\lambda^{N+2}}\sym\nabla L_N^{\eta_1}(S_3) 
+ (-1)^{N+1} \frac{\bar\Gamma_{N+1}(\lambda
  x_1)}{\lambda^{N+2}}\sym\nabla L_N^{\eta_1}(S_4), \\
& G = \sum_{i=0}^N (-1)^i \frac{\dbar\Gamma_{i}(\lambda x_1)}{\lambda^{i+2}}P_i^{\eta_1}(S_1) 
+ \sum_{i=0}^N (-1)^i \frac{(\Gamma'\Gamma)_{i}(\lambda x_1)}{\lambda^{i+1}}P_i^{\eta_1}(S_2) 
\\ & \qquad + \sum_{i=0}^N (-1)^i \frac{\Gamma_{i}(\lambda x_1)}{\lambda^{i+1}}P_i^{\eta_1}(S_3) 
+ \sum_{i=0}^N (-1)^i \frac{\bar\Gamma_{i}(\lambda x_1)}{\lambda^{i+1}}P_i^{\eta_1}(S_4). 
\end{split}
\end{equation*}
By (\ref{m10}), (\ref{sisi})  and (\ref{symW}) we now obtain: 
\begin{equation}\label{m12}
\begin{split}
& \mathcal{D}(g_0-\delta_{k+1}H_0, U) = \mathcal{D}(g_0-\delta_{k+1}H_0,
u_k) -\Big((\nabla U)^T\nabla U - (\nabla u_k)^T\nabla u_k\Big) \\ & =
\sum_{i=2}^3a_i^2\eta_i\otimes\eta_i + \frac{1}{\kappa^2}\nabla
a_2\otimes\nabla a_2 + \mathcal{F} -\mathcal{R}_1  - \mathcal{R}_2
-\mathcal{G} -Ge_2\otimes e_2,
\end{split}
\end{equation}
where the primary and the $W$-related error terms
$\mathcal{R}_1$ and $\mathcal{R}_2$ are as in Lemma \ref{lem_step2}. 

\smallskip

{\bf 5. (Bounds on the errors in first corrugation defect)}
We now estimate all terms in the decomposition (\ref{m12}), together with their derivatives.
We first observe, by (\ref{Ebounds2}), \ref{s3}:
\begin{equation}\label{s5}
\begin{split}
\|\nabla^{(m+1)}(a_1^2T_{u_k}e_1)\|_0 &\leq
C\sum_{p+q=m+1}\delta_k\mu_k^p\mu_k^q=C\delta_k\mu_k^{m+1} \\ & \mbox{for all } 
m=0\ldots 1+(3N+6)(K-k)-N.
\end{split}
\end{equation}
Consequently, recalling additionally
\ref{s4}, we get that each term in $\mathcal{R}_1$ is bounded by:
\begin{equation}\label{m13}
\|\nabla^{(m)}\mathcal{R}_1\|_0\leq C\delta_k^{3/2}\lambda^m\sigma^{N/2}\leq
\frac{\delta_k}{\sigma^{N+1}}\lambda^m\quad \mbox{ for all } m=0\ldots 1+(3N+6)(K-k)-N.
\end{equation}
where the worst term, responsible for the first bound above is
$\frac{\Gamma^2}{\lambda^2}a_1^2(\nabla E^1_{u_k})^T \nabla
E^1_{u_k}$, and where the final inequality is due to 
$C\sigma^{3N/2+1}\delta_k^{1/2}\leq \sigma^{3(N+1)/2}\delta_0^{1/2} \leq 1$, 
from the last assumption in \ref{Ass1}.
Before treating $\mathcal{R}_2$, we need to find the bounds on $W$. To
this end, we estimate:
\begin{equation}\label{SUbound}
\begin{split}
& \|\nabla^{(m)}S_1\|_0\leq C\delta_k\mu_k^{m+2}, \\
& \|\nabla^{(m)}S_2\|_0\leq C\delta_k\mu_k^{m+1}, \\
& \|\nabla^{(m)}S_3\|_0\leq C\sum_{p+q=m}\|\nabla^{(p)}a_1\|_0\|\nabla^{(q)}((\nabla u_k)^T\nabla
E_{u_k}^1)\|_0\\ & \qquad \qquad\;\; \leq
C\sum_{p+q=m}\delta_k^{1/2}\mu_k^p\delta_k^{1/2}\mu_k^{q+1}
\leq C\delta_k\mu_k^{m+1}, \\
& \|\nabla^{(m)}S_4\|_0\leq C\delta_k\mu_k^{m+1} \quad\mbox{ for all } m=0\ldots 1+(3N+6)(K-k)-N,
\end{split}
\end{equation}
where the first two bounds result from (\ref{Ebounds2}), the third bound
from (\ref{m8}) and \ref{Q1}, and the fourth from \ref{s2} and
(\ref{s5}). We now read from (\ref{coef_ibp1}) that:
\begin{equation}\label{m16}
\|\nabla^{(m)}L_i^{\eta_1}(S)\|_0\leq
C_{m+i}\|\nabla^{(m+i)}S\|_0\quad\mbox{for all } m,i\geq 0,
\end{equation}
recall (\ref{S_first}), and observe that all  the primitives of the functions $\dbar\Gamma,
(\Gamma'\Gamma), \Gamma, \bar\Gamma$ used in the definition of $W$ are
bounded, being periodic with mean zero on the period. This yields that:
\begin{equation}\label{m14}
\begin{split}
\|\nabla^{(m)}W\|_0&\leq C\sum_{i=0}^N\sum_{p+q=m}\Big(\lambda^{p-i-3}\delta_k\mu_k^{q+i+2}
+ \lambda^{p-i-2}\delta_k\mu_k^{q+i+1}\Big) \\ & \leq C
\delta_k\lambda^{m-1} \quad\mbox{ for all } m=0\ldots 1+(3N+6)(K-k)-2N.
\end{split}
\end{equation}
We are now ready to bound the terms in $\mathcal{R}_2$. Observe first
that, by (\ref{Ebounds2}), \ref{s3}, \ref{s4}:
\begin{equation*}
\begin{split}
& \|\nabla^{(m+1)}(T_{u_k}W)\|_0\leq C\delta_k\lambda^m \qquad \quad\mbox{
  for } m=0\ldots (3N+6)(K-k)-2N,
\\ &\|\nabla^{(m+1)}(a_1E^1_{u_k})\|_0\leq C\delta_k^{1/2}\mu_k^{m+1}
\, \quad \mbox{ for } m=0\ldots 1+(3N+6)(K-k)-N,
\end{split}
\end{equation*}
resulting in all the terms in $\mathcal{R}_2$ obeying:
\begin{equation}\label{m15}
\|\nabla^{(m)}\mathcal{R}_2\|_0\leq C\delta_k^{3/2}\sigma^{N/2}\lambda^m \leq
\frac{\delta_k}{\sigma^{N+1}}\lambda^m\quad \mbox{ for all } 
m=0\ldots (3N+6)(K-k)-2N,
\end{equation}
where the worst term, responsible for the first bound above is
$2\sym((\nabla u_k)^T[(\partial_1T_{u_k})W, (\partial_2T_{u_k})W])$,
and the final inequality is due to having 
$C\sigma^{3N/2+1}\delta_k^{1/2}\leq \sigma^{3(N+1)/2}\delta_0^{1/2} \leq 1$ 
according to the last assumption in \ref{Ass1} and for
$\sigma\geq\underline\sigma$ large. For estimating $\mathcal{G}$,
we again use (\ref{m16}), (\ref{SUbound}):
\begin{equation}\label{m17}
\begin{split}
\|\nabla^{(m)}\mathcal{G}\|_0& \leq C\sum_{p+q=m}\Big(\lambda^{p-N-3}\delta_k\mu_k^{q+N+3}
+ \lambda^{p-N-2}\delta_k\mu_k^{q+N+2}\Big) \\ & \leq C
\frac{\delta_k}{(\lambda/\mu_k)^{N+2}}\lambda^m 
\leq \frac{\delta_k}{\sigma^{N+1}}\lambda^m
\quad \mbox{ for all } m=0\ldots (3N+6)(K-k)-2N.
\end{split}
\end{equation}
In conclusion, the bounds (\ref{Ebounds2}), (\ref{m13}), (\ref{m15}), (\ref{m17}) in (\ref{m12}) yield:
\begin{equation}\label{m18}
\begin{split}
& \mathcal{D}(g_0-\delta_{k+1}H_0, U) = a_2^2\eta_2\otimes\eta_2 +
\frac{1}{\kappa^2}\nabla a_2\otimes\nabla a_2 + (a_3^2-G)e_2\otimes e_2
+ \mathcal{E} \\
& \mbox{where } \|\nabla^{(m)}\mathcal{E}\|_0 =
\|\nabla^{(m)}(\mathcal{F} - \mathcal{R}_1 - \mathcal{R}_2- \mathcal{G})\|_0
\leq 4\frac{\delta_k}{\sigma^{N+1}}\lambda^m
\\ &  \mbox{for all } m=0\ldots (3N+6)(K-k)-2N.
\end{split}
\end{equation}
We conclude this step by estimating the derivatives of $G$. Reading from (\ref{coef_ibp1}) that:
\begin{equation*}%\label{m19}
\|\nabla^{(m)}P_i^{\eta_1}(S)\|_0\leq
C_{m+i}\|\nabla^{(m+i)}S\|_0\quad\mbox{for all } m,i\geq 0,
\end{equation*}
we get, in view of (\ref{SUbound}) and utilizing, as usual, that $\lambda/\mu_k=\sigma$:
\begin{equation}\label{m20}
\begin{split}
\|\nabla^{(m)}G\|_0 & \leq C\sum_{i=0}^N\sum_{p+q=m}\Big(\lambda^{p-i-2}\delta_k\mu_k^{q+i+2}
+ \lambda^{p-i-1}\delta_k\mu_k^{q+i+1}\Big) \\ & \leq C\frac{\delta_k}{\sigma}\lambda^m
\quad \mbox{ for all } m=0\ldots 1+(3N+6)(K-k)-2N.
\end{split}
\end{equation}

\smallskip

{\bf 6. (Propagation of bounds in the first corrugation)} A rough
bound without specifying to components, and using only
(\ref{Ebounds2}), \ref{s3}, \ref{s4}, (\ref{m14}) yields the following: 
\begin{equation}\label{mi1}
\begin{split} 
& \|\nabla^{(m)} (U-u_k)\|_0\leq C\sum_{p+q+t=m}\lambda^{p-1}
\|\nabla^{(q)}a_1\|_0\|\nabla^{(t)}E^1_{u_k}\|_0  \\ & +  C\sum_{p+q+t=m}\lambda^{p-1}
\|\nabla^{(q)}a_1^2\|_0\|\nabla^{(t)}T_{u_k}\|_0 + C\sum_{p+q=m}
\|\nabla^{(p)}T_{u_k}\|_0 \|\nabla^{(q)}W\|_0 \leq C \delta_k^{1/2}\lambda^{m-1}.
\end{split}
\end{equation}
Now, we can use Lemma \ref{lem_propa} to $u=u_k$, $v=U$, $\mu=\lambda$,
$A=\delta_k^{1/2}$ because: 
$$\|\nabla
U-\nabla u_k\|_0\leq C\delta_k^{1/2}\leq C{\underline\delta}^{1/2}\leq \rho,$$
 in virtue of (\ref{mi1}) and upon taking $\underline\delta$ sufficiently small. Recalling
further \ref{P1}, \ref{s2}, \ref{s4}, we obtain existence of the normal frame $E^1_U,
E^2_U\in\mathcal{C}^{(3N+6)(K-k)-2N}(\bar\omega, \R^4)$ in:
$$(\nabla U)^TE^i_{U}=0,\quad |E^1_U|=1\quad\mbox{ for }\quad 
i=1,2\quad\mbox{ and } \quad \langle E^1_U, E^2_U\rangle=0\quad\mbox{in } \bar\omega,$$
satisfying the the second bound below, while the first bound is
included in (\ref{mi1}):
\begin{equation}\label{n1}
\begin{split}
& \|\nabla^{(m+1)}(U-u_k)\|_0\leq C\delta_k^{1/2}\lambda^{m}\quad 
\mbox{ and }\quad \|\nabla^{(m)}(E^i_U-E^i_{u_k})\|_0\leq C\delta_k^{1/2}\lambda^{m}
\\ & \mbox{for all } m=0\ldots (3N+6)(K-k)-2N,\; i=1,2.
\end{split}
\end{equation}
We also get, recalling \ref{m3} and the induction assumptions \ref{P2}:
\begin{equation}\label{m3UU}
\|U-u\|_1\leq \|U-u_k\|_1 + \|u_0-u\|_0 + \sum_{i=0}^{k-1}\|u_{i+1}-u_i\|_0 \leq C\delta_0^{1/2},
\end{equation}
leading, as in step 1 and the proof of (\ref{m5}), to:
\begin{equation}\label{n2}
\frac{1}{3\underline\gamma}\Id_2\leq (\nabla U)^T\nabla U\leq
3\underline\gamma\Id_2 \quad\mbox{ in } \bar\omega.
\end{equation}
We now gather estimates similar to those in \ref{s1} -
\ref{s4}. Indeed, by the above and Lemma \ref{lem_det}, we get the first 
estimate, while the second and the fourth follow from (\ref{n1}):
\begin{align*}
& \|\nabla U\|_0\leq (6\underline\gamma)^{1/2}
\tag*{(\theequation)$_1$}\refstepcounter{equation} \label{s1U} \\
& \|\nabla^{(m)}\nabla^{(2)} U\|_0\leq 
C\delta_{k}^{1/2}\lambda^{m+1}\sigma^{N/2} \quad\mbox{for all } m=0\ldots (3N+6)(K-k)-2N-1,
\tag*{(\theequation)$_2$} \label{s2U} \\
& \|T_{U}\|_0\leq C \quad\mbox{ and }\quad
\|\nabla^{(m)}T_{U}\|_0\leq C\delta_k^{1/2}\lambda^m\sigma^{N/2} \\
& \qquad\qquad\qquad \qquad\qquad\qquad \quad
\mbox{ for all } m=1\ldots (3N+6)(K-k)-2N,
\tag*{(\theequation)$_3$} \label{s3U} \\
& \|\nabla^{(m)}E^i_{U}\|_0\leq
C\delta_k^{1/2}\lambda^m\sigma^{N/2}\quad\, \mbox{for all } m=1\ldots (3N+6)(K-k)-2N,\;i=1,2,
\tag*{(\theequation)$_4$} \label{s4U} 
\end{align*}
and the estimates \ref{s3U} follow from Lemma \ref{lem_Tu}.
In the remaining part of this step, we refine the first bound in (\ref{n1}) to the components of
$\nabla^2U$. For any $i,j,s=1\ldots 2$ we write:
\begin{equation}\label{eqeq}
\langle \partial_{ij}U, E^s_{U}\rangle = \langle \partial_{ij}u_k, E^s_{u_k}\rangle
+ \langle\partial_{ij}(U-u_k), E^s_{u_k}\rangle + \langle\partial_{ij}U, (E^s_U-E^s_{u_k})\rangle. 
\end{equation}
The first term in the right hand side above obeys the bounds in
\ref{Q1}-\ref{Q2} when $k\geq 1$ and (\ref{m8}) when $k=0$. For the third term, we use
(\ref{n1}) and \ref{s2U}, \ref{s4U} and get:
\begin{equation*}
\begin{split}
&\|\nabla^{(m)}\langle\partial_{ij}U, (E^s_U-E^s_{u_k})\rangle\|_0 \leq
C\delta_k\lambda^{m+1}\sigma^{N/2 }
\\ & \mbox{for all } m=0\ldots (3N+6)(K-k)-2N-1,\; i,j,s=1\ldots 2.
\end{split}
\end{equation*}
We now estimate the second term in the right hand
side of (\ref{eqeq}), where by (\ref{n1}):
\begin{equation*}
\begin{split}
& \|\nabla^{(m)}\langle\partial_{ij}(U-u_k), E^s_{u_k}\rangle\|_0 
\leq  C\delta_k^{1/2}\lambda^{m+1} \\ & 
\mbox{for all } m=0\ldots (3N+6)(K-k)-2N-1,\; i,j,s=1\ldots 2.
\end{split}
\end{equation*}
When $s=2$, we use that $\langle E^1_{u_k}, E^2_{u_k}\rangle=0$ and write:
\begin{equation*}
\begin{split}
\langle\partial_{ij}(U-u_k), E^2_{u_k} \rangle = & ~\partial_i
\Big(\frac{\Gamma(\lambda x_1)}{\lambda} a_1\Big)\langle \partial_jE^1_{u_k},
E^2_{u_k}\rangle + \partial_j\Big(\frac{\Gamma(\lambda x_1)}{\lambda} a_1\Big)\langle \partial_iE^1_{u_k},
E^2_{u_k}\rangle \\ & + \frac{\Gamma(\lambda x_1)}{\lambda} a_1\langle \partial_{ij}E^1_{u_k},
E^2_{u_k}\rangle + \Big\langle \partial_{ij}\Big(T_{u_k}\Big(\frac{\bar\Gamma(\lambda
  x_1)}{\lambda}a_1^2e_1+W\Big)\Big), E^2_{u_k}\Big\rangle,
\end{split}
\end{equation*}
which implies, by (\ref{Ebounds2}), \ref{s3}, \ref{s4}, (\ref{m14}):
\begin{equation*}
\begin{split}
& \|\nabla^{(m)}\langle\partial_{ij}(U-u_k), E^2_{u_k}\rangle\|_0 
\leq  C\sum_{p+q+t=m}\big\|\nabla^{(p+1)}\Big (\frac{\Gamma(\lambda
  x_1)}{\lambda}a_1\Big)\big\|_0 \|\nabla^{(q+1)}E^1_{u_k}\|_0\|\nabla^{(t)}E^2_{u_k}\|_0
\\ & +\sum_{p+q+t} \big\|\nabla^{(p)}\Big (\frac{\Gamma(\lambda
  x_1)}{\lambda}a_1\Big)\big\|_0 \|\nabla^{(q+2)}E^1_{u_k}\|_0\|\nabla^{(t)}E^2_{u_k}\|_0
\\ & + C\sum_{p+q=m}\sum_{~t+s=p+2}\|\nabla^{(t)}T_{u_k}\|_0
\big\|\nabla^{(s)}\Big (\frac{\bar\Gamma(\lambda x_1)}{\lambda}a_1^2e_1 +
W\Big)\big\|_0\|\nabla^{(q)}E^2_{u_k}\|_0
\\ & \leq  C\delta_k\lambda^{m+1} \sigma^{N/2}\quad\mbox{for all }
m=0\ldots (3N+6)(K-k)-2N-1,\; i,j=1\ldots 2.
\end{split}
\end{equation*}
In conclusion, using $\delta_k^{1/2}\sigma^{N/2}\leq 1$ by
\ref{Ass1}, we obtain for all $m=0\ldots (3N+6)(K-k)-2N-1$:
\begin{align*}
&  \|\nabla^{(m)}\langle \partial_{ij}U, E^1_{U}\rangle\|_0  \leq
\|\nabla^{(m)}\langle \partial_{ij}u_k, E^1_{u_k}\rangle\|_0 +
C\delta_k^{1/2}\lambda^{m+1},
\tag*{(\theequation)$_1$}\refstepcounter{equation} \label{Q1U} \\
& \|\nabla^{(m)}\langle \partial_{ij}U, E^2_{U}\rangle\|_0 \leq
\|\nabla^{(m)}\langle \partial_{ij}u_k, E^2_{u_k}\rangle\|_0 + C\delta_k\lambda^{m+1}\sigma^{N/2}
\quad\mbox{ for } i,j=1\ldots 2.
\tag*{(\theequation)$_2$} \label{Q2U} 
\end{align*}
In particular, since \ref{Q1} and (\ref{m8}) yield
$\|\nabla^{(m)}\langle \partial_{ij}u_k, E^1_{u_k}\rangle\|_0 \leq
C\delta_k^{1/2}\lambda^{m+1}$,  we also get:
\begin{equation}\label{Q1Ugr}
\|\nabla^{(m)}((\nabla U)^T\nabla
E_{U}^1)\|_0\leq C\delta_k^{1/2}\lambda^{m+1} \quad\mbox{ for all } m=0\ldots (3N+6)(K-k)-2N-1.
\end{equation}

\smallskip

{\bf 7. (The second corrugation)} We define 
$\bar U\in\mathcal{C}^{(3N+6)(K-k)-3N-1}(\bar\omega, \R^4)$:
\begin{equation}\label{corru2}
\bar U=U+\frac{\Gamma(\kappa t)}{\kappa} a_2E^1_{U} +
T_{U}\Big(\frac{\bar\Gamma(\kappa t)}{\kappa}a_2^2 \eta_2+ \bar W\Big),
\end{equation}
that is the second intermediary field, in accordance with Lemma
\ref{lem_step2}. We use the oscillatory 
profiles $\Gamma$, $\bar\Gamma$, $\dbar\Gamma$ in (\ref{m11}), the oscillation direction 
$\eta=\eta_2$, where we denote:
$$t=\langle x,\eta_2\rangle, $$
and the oscillation frequency in:
$$\kappa = \lambda\sigma = \mu_k\sigma^2.$$ 
The tangential correction $\bar W$ is given recalling the definition
(\ref{coef_ibp2}) and the recursion (\ref{recu_ibp}):
\begin{equation*}
\begin{split}
-2\bar W = & \sum_{i=0}^N (-1)^i \frac{\dbar\Gamma_{i+1}(\kappa
  t)}{\kappa^{i+3}}L_i^{\eta_2}(\bar S_1) 
+ \sum_{i=0}^N (-1)^i \frac{(\Gamma'\Gamma)_{i+1}(\kappa t)}{\kappa^{i+2}}L_i^{\eta_2}(\bar S_2) 
\\ & + \sum_{i=0}^N (-1)^i \frac{\Gamma_{i+1}(\kappa t)}{\kappa^{i+2}}L_i^{\eta_2}(\bar S_3) 
+ \sum_{i=0}^N (-1)^i \frac{\bar\Gamma_{i+1}(\kappa t)}{\kappa^{i+2}}L_i^{\eta_2}(\bar S_4),
\end{split}
\end{equation*}
with the fields $\{\bar S_i\}_{i=1}^4$ as in Lemma \ref{lem_step2}, namely:
\begin{equation}\label{S_second}
\begin{split}
& \bar S_1= \nabla a_2\otimes\nabla a_2, \qquad\qquad\qquad \qquad\;\;
\bar S_2 = 2a_2\sym(\nabla a_2\otimes \eta_2),\\
& \bar S_3 = 2a_2\sym\big((\nabla U)^T\nabla E^1_{U}\big), \qquad\qquad
\bar S_4 =2\,\sym\big((\nabla U)^T\nabla (a_2^2 T_{U}\eta_2)\big).
\end{split}
\end{equation}
By (\ref{ibp2}) in Lemma \ref{lem_IBP2} we see that:
\begin{equation}\label{symWbar}
\begin{split}
& -2\,\sym\nabla \bar W = \frac{\dbar\Gamma(\kappa t)}{\kappa^2} \bar S_1 
+ \frac{(\Gamma'\Gamma)(\kappa t)}{\kappa} \bar S_2 
+ \frac{\Gamma(\kappa t)}{\kappa} \bar S_3+ \frac{\bar\Gamma(\kappa
  t)}{\kappa} \bar S_4  - \bar{\mathcal{G}} - \bar G e_2\otimes e_2, 
\end{split}
\end{equation}
with the following formulas:
\begin{equation*}
\begin{split}
& \bar{\mathcal{G}} = (-1)^{N+1} \frac{\dbar\Gamma_{N+1}(\kappa
  t)}{\kappa^{N+3}}\sym \nabla L_N^{\eta_2}(\bar S_1)  
+ (-1)^{N+1} \frac{(\Gamma'\Gamma)_{N+1}(\kappa
  t)}{\lambda^{N+2}}\sym\nabla L_N^{\eta_2}(\bar S_2) 
\\ & \qquad + (-1)^{N+1} \frac{\Gamma_{N+1}(\kappa
  t)}{\lambda^{N+2}}\sym\nabla L_N^{\eta_2}(\bar S_3) 
+ (-1)^{N+1} \frac{\bar\Gamma_{N+1}(\kappa
  t)}{\kappa^{N+2}}\sym\nabla L_N^{\eta_2}(\bar S_4), \\
& \bar G = \sum_{i=0}^N (-1)^i \frac{\dbar\Gamma_{i}(\kappa
  t)}{\kappa^{i+2}}P_i^{\eta_2}(\bar S_1) 
+ \sum_{i=0}^N (-1)^i \frac{(\Gamma'\Gamma)_{i}(\kappa
  t)}{\kappa^{i+1}}P_i^{\eta_2}(\bar S_2) 
\\ & \qquad + \sum_{i=0}^N (-1)^i \frac{\Gamma_{i}(\kappa
  t)}{\kappa^{i+1}}P_i^{\eta_2}(\bar S_3) 
+ \sum_{i=0}^N (-1)^i \frac{\bar\Gamma_{i}(\kappa
  t)}{\kappa^{i+1}}P_i^{\eta_2}(\bar S_4). 
\end{split}
\end{equation*}
By (\ref{m18}), (\ref{sisi})  and (\ref{symWbar}) we now obtain: 
\begin{equation}\label{m12bar}
\begin{split}
& \mathcal{D}(g_0-\delta_{k+1}H_0, \bar U) = \mathcal{D}(g_0-\delta_{k+1}H_0,
U) -\big((\nabla \bar U)^T\nabla \bar U - (\nabla U)^T\nabla U\big) \\ & =
\big(a_3^2 - G-\bar G\big) e_2\otimes e_2 +\mathcal{E}
- \bar{\mathcal{R}}_1  - \bar{\mathcal{R}}_2 -\bar{\mathcal{G}},
\end{split}
\end{equation}
where the primary and the $\bar W$-related error terms
$\bar{\mathcal{R}}_1$ and $\bar{\mathcal{R}}_2$ are as in Lemma \ref{lem_step2}. 

\smallskip

{\bf 8. (Bounds on the errors in second corrugation defect)}
We now estimate, as in step 5, all terms in the decomposition (\ref{m12bar}).
We first observe, by (\ref{Ebounds2}), \ref{s3U}:
\begin{equation}\label{s5bar}
\begin{split}
& \|\nabla^{(m+1)}(a_2^2T_{U}\eta_2)\|_0\leq
C\sum_{p+q=m+1}\delta_k\mu_k^p\lambda^q=C\delta_k\lambda^{m+1} 
\\ & \mbox{for all } m=0\ldots (3N+6)(K-k)-2N-1.
\end{split}
\end{equation}
Consequently, recalling additionally
\ref{s4U}, we get that each term in $\bar{\mathcal{R}}_1$ is bounded by:
\begin{equation}\label{m13bar}
\begin{split}
& \|\nabla^{(m)}\bar{\mathcal{R}}_1\|_0\leq C\delta_k^{3/2}\kappa^m\sigma^{N/2}\leq
\frac{\delta_k}{\sigma^{N+1}}\kappa^m
\\ & \mbox{for all } m=0\ldots (3N+6)(K-k)-2N-2,
\end{split}
\end{equation}
where the worst term, responsible for the first bound is
$\frac{\Gamma^2}{\kappa^2}a_2^2(\nabla E^1_{U})^T \nabla
E^1_{U}$, and where the final inequality follows by
$C\sigma^{3N/2+1}\delta_k^{1/2}\leq \sigma^{3(N+1)/2}\delta_0^{1/2} \leq 1$ 
according to the last assumption in \ref{Ass1} with
$\underline\sigma$ sufficiently large.
Before bounding $\bar{\mathcal{R}}_2$, we find the bounds on
$\bar W$. Firstly:
\begin{equation}\label{SUboundbar}
\begin{split}
& \|\nabla^{(m)}\bar S_1\|_0\leq C\delta_k\mu_k^{m+2} \leq C\delta_k\lambda^{m+2}, \\
& \|\nabla^{(m)}\bar S_2\|_0\leq C\delta_k\mu_k^{m+1} \leq C\delta_k\lambda^{m+1}, \\
& \|\nabla^{(m)}\bar S_3\|_0\leq
C\sum_{p+q=m}\|\nabla^{(p)}a_2\|_0\|\nabla^{(q)}((\nabla U)^T\nabla
E_{U}^1)\|_0\\ & \qquad \qquad\;\; \leq
C\sum_{p+q=m}\delta_k^{1/2}\mu_k^p\delta_k^{1/2}\lambda^{q+1}
\leq C\delta_k\lambda^{m+1}, \\
& \|\nabla^{(m)}\bar S_4\|_0\leq C\delta_k\lambda^{m+1} \quad\mbox{ for all } 
m=0\ldots (3N+6)(K-k)-2N-1,
\end{split}
\end{equation}
where the first two bounds result from (\ref{Ebounds2}), the third bound
from (\ref{Q1Ugr}), and the fourth from \ref{s1U}, \ref{s2U} and (\ref{s5bar}). 
We now read from (\ref{coef_ibp2}) that:
\begin{equation}\label{m16bar}
\|\nabla^{(m)}L_i^{\eta_2}(S)\|_0\leq
C_{m+i}\|\nabla^{(m+i)}S\|_0\quad\mbox{for all } m,i\geq 0.
\end{equation}
By (\ref{SUboundbar}), and recalling that all the primitives of $\dbar\Gamma,
(\Gamma'\Gamma), \Gamma, \bar\Gamma$ in the definition of $\bar W$ are
bounded, being periodic with mean zero on the period, we get:
\begin{equation}\label{m14bar}
\begin{split}
\|\nabla^{(m)}\bar W\|_0&\leq C\sum_{i=0}^N\sum_{p+q=m}\Big(\kappa^{p-i-3}\delta_k\lambda^{q+i+2}
+ \kappa^{p-i-2}\delta_k\lambda^{q+i+1}\Big) \\ & \leq C
\delta_k\kappa^{m-1} \quad\mbox{ for all } m=0\ldots (3N+6)(K-k)-3N-1.
\end{split}
\end{equation}
We are now ready to bound the terms in $\bar{\mathcal{R}}_2$. Observe first
that, by (\ref{Ebounds2}), \ref{s3U}, \ref{s4U}:
\begin{equation*}
\begin{split}
& \|\nabla^{(m+1)}(T_{U}\bar W)\|_0\leq C\delta_k\kappa^m
\quad\qquad \mbox{ for } m=0\ldots (3N+6)(K-k)-3N-2,\\
& \|\nabla^{(m+1)}(a_2E^1_{U})\|_0\leq C\delta_k^{1/2}\lambda^{m+1}
\,\quad \mbox{ for } m=0\ldots (3N+6)(K-k)-2N,
\end{split}
\end{equation*}
resulting in all the terms in $\bar{\mathcal{R}}_2$ being bounded by:
\begin{equation}\label{m15bar}
\|\nabla^{(m)}\bar{\mathcal{R}}_2\|_0\leq C\delta_k^{3/2}\kappa^m \sigma^{N/2}\leq
\frac{\delta_k}{\sigma^{N+1}}\kappa^m\quad \mbox{ for } m=0\ldots (3N+6)(K-k)-3N-2,
\end{equation}
where the worst term, responsible for the first bound above is
$2\sym((\nabla U)^T[(\partial_1T_{U})\bar W, (\partial_2T_{U})\bar W])$,
and the final inequality is due to having 
$C\sigma^{3N/2+1}\delta_k^{1/2}\leq \sigma^{3(N+1)/2}\delta_0^{1/2} \leq 1$ 
from the last assumption in \ref{Ass1} and for
$\sigma\geq\underline\sigma$ large. For estimating $\bar{\mathcal{G}}$,
we again use (\ref{m16bar}), (\ref{SUboundbar}):
\begin{equation}\label{m17bar}
\begin{split}
\|\nabla^{(m)}\bar{\mathcal{G}}\|_0& \leq C\sum_{p+q=m}\Big(\kappa^{p-N-3}\delta_k\lambda^{q+N+3}
+ \kappa^{p-N-2}\delta_k\lambda^{q+N+2}\Big) \leq C
\frac{\delta_k}{(\kappa/\lambda)^{N+2}}\kappa^m 
\\ & \leq \frac{\delta_k}{\sigma^{N+1}}\kappa^m
\quad \mbox{ for all } m=0\ldots (3N+6)(K-k)-3N-2.
\end{split}
\end{equation}
Consequently, the bounds (\ref{m18}), (\ref{m13bar}), 
(\ref{m15bar}), (\ref{m17bar}) in (\ref{m12bar}) yield:
\begin{equation}\label{m18bar}
\begin{split}
& \mathcal{D}(g_0-\delta_{k+1}H_0, \bar U) = (a_3^2-G - \bar G) e_2\otimes e_2
+ \bar{\mathcal{E}}\\
& \mbox{where } \|\nabla^{(m)}\bar{\mathcal{E}}\|_0 =
\|\nabla^{(m)}(\mathcal{E} - \bar{\mathcal{R}}_1 - \bar{\mathcal{R}}_2- \bar{\mathcal{G}})\|_0
\leq 7\frac{\delta_k}{\sigma^{N+1}}\lambda^m \\ &
\mbox{for all } m=0\ldots (3N+6)(K-k)-3N-2.
\end{split}
\end{equation}
We conclude this step by estimating the current leading order defect
$(a_3^2-G - \bar G) e_2\otimes e_2$. Reading from (\ref{coef_ibp2}) that:
\begin{equation*}
\|\nabla^{(m)}P_i^{\eta_2}(S)\|_0\leq
C_{m+i}\|\nabla^{(m+i)}S\|_0\quad\mbox{for all } m,i\geq 0,
\end{equation*}
we get, in view of (\ref{SUboundbar}) and utilizing, as before, that $\kappa/\lambda=\sigma$:
\begin{equation}\label{m20bar}
\begin{split}
& \|\nabla^{(m)}\bar G\|_0  \leq C\sum_{i=0}^N\sum_{p+q=m}\Big(\kappa^{p-i-2}\delta_k\lambda^{q+i+2}
+ \kappa^{p-i-1}\delta_k\lambda^{q+i+1}\Big) \leq
C\frac{\delta_k}{\sigma}\kappa^m\\ & \mbox{hence}\quad \|\nabla^{(m)}(G+\bar G)\|_0  \leq
C\frac{\delta_k}{\sigma}\kappa^m 
\quad \mbox{ for all } m=0\ldots (3N+6)(K-k)-3N-1.
\end{split}
\end{equation}
By (\ref{Ebounds2}) it results that taking $\sigma\geq \underline\sigma $ sufficiently large, 
the new positive profile function
$b\in\mathcal{C}^\infty(\bar\omega,\R)$ is well defined, through:
$$b^2 = a_3^2-G - \bar G.$$
We then obtain:
\begin{equation}\label{Ebounds2b}
\begin{split}
& \frac{1}{4}\delta_k\leq b^2\leq \frac{7}{4}\delta_k
\qquad\quad \, \mbox{ and } \quad \frac{1}{4}\delta_k^{1/2}\leq b\leq \frac{7}{4}\delta_k^{1/2}
\quad \mbox{in }\bar\omega, \\
& \|\nabla^{(m)} b^2\|_0\leq C\frac{\delta_k}{\sigma}\kappa^{m} \quad\mbox{ and } \quad 
\|\nabla^{(m)} b\|_0\leq C\frac{\delta_k^{1/2}}{\sigma}\kappa^{m}
\\ &   \qquad\qquad\qquad\qquad \qquad\mbox{ for all } m=1\ldots (3N+6)(K-k)-3N-1,
\end{split}
\end{equation}
where the bound on the non-zero derivatives of $b^2$ is a direct consequence of
(\ref{Ebounds2}) and (\ref{m20bar}), while the following bound on the
derivatives of $b$ is justified by an application of Fa\`a di Bruno's
formula, as in step 2 of the proof of Theorem \ref{prop1}.

\smallskip

{\bf 9. (Propagation of bounds in the second corrugation)} A rough
bound without specifying to components, and using only
(\ref{Ebounds2}), \ref{s3U}, \ref{s4U}, (\ref{m14bar}) yields the following: 
\begin{equation}\label{mi2}
\begin{split} 
& \|\nabla^{(m)} (\bar U-U)\|_0\leq C\sum_{p+q+t=m}\kappa^{p-1}
\|\nabla^{(q)}a_2\|_0\|\nabla^{(t)}E^1_{U}\|_0  \\ & +  C\sum_{p+q+t=m}\kappa^{p-1}
\|\nabla^{(q)}a_2^2\|_0\|\nabla^{(t)}T_{U}\|_0 + C\sum_{p+q=m}
\|\nabla^{(p)}T_{U}\|_0 \|\nabla^{(q)}\bar W\|_0 \leq C \delta_k^{1/2}\kappa^{m-1}.
\end{split}
\end{equation}
We may now apply Lemma \ref{lem_propa} with $u=U$, $v=\bar U$, $\mu=\kappa$,
$A=\delta_k^{1/2}$ because $\|\nabla
\bar U-\nabla U\|_0\leq C\delta_k^{1/2}\leq C{\underline\delta}^{1/2}\leq
\rho$ for $\underline\delta$ sufficiently small, in view of (\ref{mi2}). Recalling
further (\ref{n2}), \ref{s2U}, \ref{s4U}, we obtain existence of the
unit normal frame $E^1_{\bar U},
E^2_{\bar U}\in\mathcal{C}^{(3N+6)(K-k)-3N-2}(\bar\omega, \R^4)$ in:
$$(\nabla \bar U)^TE^i_{\bar U}=0,\quad |E^1_{\bar U}|=1\quad\mbox{ for }\quad 
i=1,2\quad\mbox{ and } \quad \langle E^1_{\bar U}, E^2_{\bar U}=0\quad\mbox{in } \bar\omega,$$
satisfying the the second bound below, with the first bound included in (\ref{mi2}):
\begin{equation}\label{n1dbar}
\begin{split}
& \|\nabla^{(m+1)}(\bar U-U)\|_0\leq C\delta_k^{1/2}\kappa^{m}
\quad \mbox{ and }\quad \|\nabla^{(m)}(E^i_{\bar U}-E^i_{U})\|_0\leq C\delta_k^{1/2}\kappa^{m}
\\ & \mbox{for all } m=0\ldots (3N+6)(K-k)-3N-2,\; i=1,2.
\end{split}
\end{equation}
We also note, recalling (\ref{m3UU}):
\begin{equation}\label{m3UUd}
\|\bar U-u\|_1\leq \|\bar U-U \|_1 + \|U-u\|_0 \leq C\delta_0^{1/2},
\end{equation}
which leads, as in step 1 and the proof of (\ref{m5}), to:
\begin{equation}\label{n2d}
\frac{1}{3\underline\gamma}\Id_2\leq (\nabla \bar U)^T\nabla \bar U\leq
3\underline\gamma\Id_2 \quad\mbox{ in } \bar\omega.
\end{equation}
We proceed to gather estimates similar to those in \ref{s1U}-\ref{s4U}. 
The first bound below follows from (\ref{n2d}) and Lemma
\ref{lem_det}, while the second and the fourth bounds follow from (\ref{n1dbar}):
\begin{align*}
& \|\nabla \bar U\|_0\leq (6\underline\gamma)^{1/2}
\tag*{(\theequation)$_1$}\refstepcounter{equation} \label{s1Ud} \\
& \|\nabla^{(m)}\nabla^{(2)} \bar U\|_0\leq 
C\delta_{k}^{1/2}\kappa^{m+1}\sigma^{N/2} \quad\mbox{for all } m=0\ldots (3N+6)(K-k)-3N-3,
\tag*{(\theequation)$_2$} \label{s2Ud} \\
& \|T_{\bar U}\|_0\leq C \quad\mbox{ and }\quad
\|\nabla^{(m)}T_{\bar U}\|_0\leq C\delta_k^{1/2}\kappa^m\sigma^{N/2}
\\ & \qquad\qquad \qquad\qquad\qquad \qquad \hspace{1mm}
\quad\mbox{for all } m=1\ldots (3N+6)(K-k)-3N-2,
\tag*{(\theequation)$_3$} \label{s3Ud} \\
& \|\nabla^{(m)}E^i_{\bar  U}\|_0\leq
C\delta_k^{1/2}\kappa^m\sigma^{N/2}\quad\mbox{for all } m=1\ldots (3N+6)(K-k)-3N-2,\;i=1,2. 
\tag*{(\theequation)$_4$} \label{s4Ud} 
\end{align*}
The estimate \ref{s3Ud} results from Lemma \ref{lem_Tu}.
Below,  we refine the first bound in \ref{s2Ud} to the components of
$\nabla^2\bar U$. For any $i,j,s=1\ldots 2$ we write:
\begin{equation}\label{eqeqdbar}
\langle \partial_{ij}\bar U, E^s_{\bar U}\rangle = \langle \partial_{ij}U, E^s_{U}\rangle
+ \langle\partial_{ij}(\bar U-U), E^s_{U}\rangle +
\langle\partial_{ij}\bar U, (E^s_{\bar U}-E^s_{U})\rangle. 
\end{equation}
The first term in the right hand side above obeys the bounds in
\ref{Q1U}-\ref{Q2U}, while for the third term, we use
(\ref{n1dbar}) and \ref{s2Ud}, \ref{s4Ud} and get:
\begin{equation*}
\begin{split}
& \|\nabla^{(m)}\langle\partial_{ij}\bar U, (E^s_{\bar U}-E^s_{U})\rangle\|_0 \leq
C\delta_k\kappa^{m+1}\sigma^{N/2}
\\ & \mbox{for }m=0\ldots (3N+6)(K-k)-3N-3,\; i,j,s=1\ldots 2.
\end{split}
\end{equation*}
We now estimate the second term in the right hand
side of (\ref{eqeqdbar}), where by (\ref{n1dbar}), \ref{s3U}:
\begin{equation*}
\begin{split}
& \|\nabla^{(m)}\langle\partial_{ij}(\bar U-U), E^s_{U}\rangle\|_0 
\leq  C\delta_k^{1/2}\kappa^{m+1} \\ & \mbox{for } m=0\ldots (3N+6)(K-k)-3N-3,\; i,j,s=1\ldots 2.
\end{split}
\end{equation*}
In the specific case of $s=2$, we use that $\langle E^1_{U}, E^2_{U}\rangle=0$ and write:
\begin{equation*}
\begin{split}
\langle\partial_{ij}(\bar U-U), E^2_{U} \rangle = & ~\partial_i
\Big(\frac{\Gamma(\kappa t)}{\kappa } a_2\Big)\langle \partial_jE^1_{U},
E^2_{U}\rangle + \partial_j\Big(\frac{\Gamma(\kappa t)}{\kappa} a_2\Big)\langle \partial_iE^1_{U},
E^2_{U}\rangle \\ & + \frac{\Gamma(\kappa t)}{\kappa} a_2\langle \partial_{ij}E^1_{U},
E^2_{U}\rangle + \Big\langle \partial_{ij}\Big(T_{U}\Big(\frac{\bar\Gamma(\kappa 
 t)}{\kappa}a_2^2\eta_2+\bar W\Big)\Big), E^2_{U}\Big\rangle,
\end{split}
\end{equation*}
which implies, by (\ref{Ebounds2}), \ref{s3U}, \ref{s4U}, (\ref{m14bar}):
\begin{equation*}
\begin{split}
& \|\nabla^{(m)}\langle\partial_{ij}(\bar U-U), E^2_{U}\rangle\|_0 
\leq  C\sum_{p+q+t=m}\big\|\nabla^{(p+1)}\Big (\frac{\Gamma(\kappa
  t)}{\kappa} a_2\Big)\big\|_0 \|\nabla^{(q+1)}E^1_{U}\|_0\|\nabla^{(t)}E^2_{U}\|_0
\\ & +\sum_{p+q+t} \big\|\nabla^{(p)}\Big (\frac{\Gamma(\kappa
  t)}{\kappa} a_2\Big)\big\|_0 \|\nabla^{(q+2)}E^1_{U}\|_0\|\nabla^{(t)}E^2_{U}\|_0
\\ & + C\sum_{p+q=m}\sum_{~t+s=p+2}\|\nabla^{(t)}T_{U}\|_0
\big\|\nabla^{(s)}\Big (\frac{\bar\Gamma(\kappa t)}{\kappa} a_2^2\eta_2 +
\bar W\Big)\big\|_0\|\nabla^{(q)}E^2_{U}\|_0
\\ & \leq  C\delta_k\kappa^{m+1} \sigma^{N/2}\quad\mbox{for all }
m=0\ldots (3N+6)(K-k)-3N-3, \; i,j=1\ldots 2.
\end{split}
\end{equation*}
Using \ref{Q1U}-\ref{Q2U} and recalling that
$\delta_k^{1/2}\sigma^{N/2}\leq \delta^{1/2}\sigma^{N/2}\leq 1$ by
\ref{Ass1}, we obtain for $m=0\ldots (3N+6)(K-k)-3N-3$:
\begin{align*}
&  \|\nabla^{(m)}\langle \partial_{ij}\bar U, E^1_{\bar U}\rangle\|_0  \leq
\|\nabla^{(m)}\langle \partial_{ij}u_k, E^1_{u_k}\rangle\|_0 +
C\delta_k^{1/2}\kappa^{m+1} \quad\mbox{ for } i,j=1\ldots 2,
\tag*{(\theequation)$_1$}\refstepcounter{equation} \label{Q1Ud} \\
& \|\nabla^{(m)}\langle \partial_{ij}\bar U, E^2_{\bar U}\rangle\|_0 \leq
\|\nabla^{(m)}\langle \partial_{ij}u_k, E^2_{u_k}\rangle\|_0 + C\delta_k \kappa^{m+1}\sigma^{N/2}.
\tag*{(\theequation)$_2$} \label{Q2Ud} 
\end{align*}
In particular, \ref{Q2}-\ref{Q4}  and (\ref{m8}) yield:
$\|\nabla^{(m)}\langle \partial_{ij}u_k, E^2_{u_k}\rangle\|_0 \leq
C\delta_k^{1/2}\mu_k^{m+1}\sigma^{N/2}$, hence:
\begin{equation}\label{Q1Udgr}
\begin{split}
& \|\nabla^{(m)}((\nabla \bar U)^T\nabla E_{\bar U}^2)\|_0
 \leq \|\nabla^{(m)}\big((\nabla u_k)^T\nabla E^2_{u_k}\big)\|_0 + C\delta_k\kappa^{m+1}\sigma^{N/2}
\\ & \leq C\delta_k^{1/2}\mu_k^{m+1}\sigma^{N/2} \leq C \delta_k^{1/2}\kappa^{m+1}\sigma^{N/2-2}
\\ & \mbox{for all } m=0\ldots (3N+6)(K-k)-3N-3,
\end{split}
\end{equation}
in view of \ref{Ass1} and provided that $N\geq 4$.

\smallskip

{\bf 10. (The third and final corrugation)}
We define $u_{k+1}\in\mathcal{C}^{(3N+6)(K-k)-3N-4}(\bar\omega,
\R^4)$, whose regularity is as stipulated in the induction set-up, as
$(3N+6)(K-k)-3N-4 = 2+ (3N+6)(K-(k+1))$. This is the final field in
our triple-corrugation Stage:
\begin{equation}\label{corru3}
u_{k+1}=\bar U+\frac{\Gamma(\mu_{k+1} x_2)}{\mu_{k+1}} bE^2_{\bar U} +
T_{\bar U}\Big(\frac{\bar\Gamma(\mu_{k+1} x_2)}{\mu_{k+1}}b^2 e_2+ \dbar W\Big),
\end{equation}
in accordance with Lemma \ref{lem_step2}, where $\Gamma$, $\bar\Gamma$ are the oscillatory
profiles in (\ref{m11}), the oscillation direction is set to
$\eta=\eta_3=e_2$, the oscillation frequency is $\mu_{k+1}$ given in
(\ref{md_def}) and the amplitude function  $b$ obeys the bounds in
(\ref{Ebounds2b}).
The tangential correction $\dbar W$ is:
\begin{equation*}
\begin{split}
-2\dbar W =   & \sum_{i=0}^1 (-1)^i \frac{(\Gamma'\Gamma)_{i+1}(\mu_{k+1} x_2
 )}{\mu_{k+1} ^{i+2}}L_i^{\eta_3}(\dbar S_2)  
\\ & + \sum_{i=0}^1 (-1)^i \frac{\Gamma_{i+1}(\mu_{k+1} x_2)}{\mu_{k+1}^{i+2}}L_i^{\eta_3}(\dbar S_3) 
+ \sum_{i=0}^1 (-1)^i \frac{\bar\Gamma_{i+1}(\mu_{k+1} x_2)}{\mu_{k+1}^{i+2}}L_i^{\eta_3}(\dbar S_4),
\end{split}
\end{equation*}
where we recall (\ref{coef_ibp3}), (\ref{recu_ibp})
and  where the fields $\{\dbar S_i\}_{i=2}^4$ are as in Lemma \ref{lem_step2}, namely:
\begin{equation}\label{S_third}
\begin{split}
& \dbar S_2 = 2b\,\sym(\nabla b \otimes e_2),\\
& \dbar S_3 = 2b\,\sym\big((\nabla \bar U)^T\nabla E^2_{\bar U}\big), \qquad\qquad
\dbar S_4 =2\,\sym\big((\nabla \bar U)^T\nabla (b^2 T_{\bar U}e_2)\big).
\end{split}
\end{equation}
By (\ref{ibp3}) in Lemma \ref{lem_IBP3} we see that:
\begin{equation}\label{symWdbar}
\begin{split}
& -2\,\sym\nabla \dbar W = \frac{(\Gamma'\Gamma)(\mu_{k+1}x_2)}{\mu_{k+1}} \dbar S_2 
+ \frac{\Gamma(\mu_{k+1} x_2)}{\mu_{k+1}} \dbar S_3+ \frac{\bar\Gamma(\mu_{k+1}
  x_2)}{\mu_{k+1}} \dbar S_4  - \dbar{\mathcal{G}} - \dbar G e_1\otimes e_1, 
\end{split}
\end{equation}
with the following formulas:
\begin{equation*}
\begin{split}
& \dbar{\mathcal{G}} = \frac{(\Gamma'\Gamma)_{2}(\mu_{k+1}x_2
  )}{\mu_{k+1}^{3}}\sym\nabla L_1^{\eta_3}(\dbar S_2) 
+ \frac{\Gamma_{2}(\mu_{k+1} x_2
  )}{\mu_{k+1}^{3}}\sym\nabla L_1^{\eta_3}(\dbar S_3) 
+ \frac{\bar\Gamma_{2}(\mu_{k+1} x_2
  )}{\mu_{k+1}^{3}}\sym\nabla L_1^{\eta_3}(\dbar S_4), \\
& \dbar G = \sum_{i=0}^1 (-1)^i \frac{(\Gamma'\Gamma)_{i}(\mu_{k+1}x_2
  )}{\mu_{k+1}^{i+1}}P_i^{\eta_3}(\dbar S_2) 
+ \sum_{i=0}^1 (-1)^i \frac{\Gamma_{i}(\mu_{k+1}x_2
  )}{\mu_{k+1}^{i+1}}P_i^{\eta_3}(\dbar S_3) 
\\ & \qquad + \sum_{i=0}^1 (-1)^i \frac{\bar\Gamma_{i}(\mu_{k+1} x_2
  )}{\mu_{k+1}^{i+1}}P_i^{\eta_3}(\dbar S_4). 
\end{split}
\end{equation*}
By (\ref{m18bar}), (\ref{sisi})  and (\ref{symWdbar}) we now obtain: 
\begin{equation}\label{m12dbar}
\begin{split}
& \mathcal{D}(g_0-\delta_{k+1}H_0, u_{k+1}) = \mathcal{D}(g_0-\delta_{k+1}H_0,
\bar U) -\big((\nabla u_{k+1})^T\nabla u_{k+1} - (\nabla \bar U)^T\nabla \bar U\big) \\ & =
\bar{\mathcal{E}} - 
\frac{1+\dbar\Gamma(\mu_{k+1}x_2)}{\mu_{k+1}^2}\nabla b\otimes\nabla
b - \dbar{\mathcal{R}}_1  - \dbar{\mathcal{R}}_2 -\dbar{\mathcal{G}}
-\dbar G e_1\otimes e_1
\end{split}
\end{equation}
where the primary and the $\dbar W$-related error terms
$\dbar{\mathcal{R}}_1$ and $\dbar{\mathcal{R}}_2$ are as in Lemma \ref{lem_step2}. 

\smallskip

{\bf 11. (Bounds on the errors in third corrugation defect)}
We now estimate, as in steps 5 and 8, all terms in the final
decomposition (\ref{m12dbar}). Firstly, by (\ref{Ebounds2b}):
\begin{equation}\label{newterm}
\begin{split}
& \big\|\nabla^{(m)}\Big(\frac{1+\dbar\Gamma(\mu_{k+1}x_2)}{\mu_{k+1}^2}\nabla
b\otimes\nabla b\Big)\big\|_0
\leq C\sum_{p+q=m}\mu_{k+1}^{p-2}\frac{\delta_k}{\sigma^2}\kappa^{q+2}
\\ & \leq C\frac{\delta_k}{(\mu_{k+1}/\kappa)^2\sigma^2}\mu_{k+1}^m \leq 
\frac{\delta_k}{\sigma^{N+1}}\mu_{k+1}^m \quad \mbox{ for all } m=0\ldots (3N+6)(K-k)-3N-2,
\end{split}
\end{equation}
if only $\sigma\geq\underline\sigma$ is sufficiently large. 
Secondly, from (\ref{Ebounds2b}), \ref{s3Ud}:
\begin{equation}\label{s5dbar}
\begin{split}
& \|\nabla^{(m+1)}(b^2T_{\bar U}e_2)\|_0\leq 
C\sum_{p+q=m+1}\delta_k\kappa^p\kappa^q=C\delta_k\kappa^{m+1} 
\\ & \mbox{for all } m=0\ldots (3N+6)(K-k)-3N-3.
\end{split}
\end{equation}
Consequently, recalling additionally
\ref{s4Ud}, we get that each term in $\dbar{\mathcal{R}}_1$ is bounded by:
\begin{equation}\label{m13dbar}
\begin{split}
& \|\nabla^{(m)}\dbar{\mathcal{R}}_1\|_0\leq C\delta_k^{3/2}\mu_{k+1}^m\sigma^{N/2}\leq
\frac{\delta_k}{\sigma^{N+1}}\mu_{k+1}^m\\ & \mbox{for all } m=0\ldots (3N+6)(K-k)-3N-2,
\end{split}
\end{equation}
where the worst term, responsible for the first bound is
$\frac{\Gamma(\mu_{k+1}x_2)^2}{\mu_{k+1}^2}b^2(\nabla E^2_{\bar U})^T \nabla
E^2_{\bar U}$, and where the final inequality follows by
$C\sigma^{3N/2+1}\delta_k^{1/2}\leq \sigma^{3(N+1)/2}\delta_0^{1/2} \leq 1$ 
from the last assumption in \ref{Ass1}, for $\underline\sigma$ sufficiently large.
Before bounding $\bar{\mathcal{R}}_2$, we find the bounds on
$\dbar W$. We get:
\begin{equation}\label{SUbounddbar}
\begin{split}
& \|\nabla^{(m)}\dbar S_2\|_0\leq C\delta_k\kappa^{m+1},\\
& \|\nabla^{(m)}\dbar S_3\|_0\leq
C\sum_{p+q=m}\|\nabla^{(p)}b\|_0\|\nabla^{(q)}((\nabla \bar U)^T\nabla
E_{\bar U}^2)\|_0\\ & \qquad \qquad\;\; \leq
C\sum_{p+q=m}\delta_k^{1/2}\kappa^p \delta_k^{1/2}\kappa^{q+1}\sigma^{N/2-2}
\leq C\delta_k\kappa^{m+1}\sigma^{N/2-2}, \\
& \|\nabla^{(m)}\dbar S_4\|_0\leq C\delta_k\kappa^{m+1} \quad\mbox{ for all } 
m=0\ldots (3N+6)(K-k)-3N-3,
\end{split}
\end{equation}
where the first bound results from (\ref{Ebounds2b}), the second bound
from (\ref{Q1Udgr}), and the third from \ref{s1Ud}, \ref{s2Ud} and (\ref{s5dbar}). 
We now read from (\ref{coef_ibp3}) that:
\begin{equation}\label{m16dbar}
\|\nabla^{(m)}L_i^{\eta_3}(S)\|_0\leq
C_{m+i}\|\nabla^{(m+i)}S\|_0\quad\mbox{for all } m,i\geq 0,
\end{equation}
which implies by recalling from (\ref{SUbounddbar}) that
$\|\nabla^{(m)}\dbar S_i\|_0 \leq C\delta_k\mu_{k+1}^{m+1}$:
\begin{equation}\label{m14dbar}
\begin{split}
& \|\nabla^{(m)}\dbar W\|_0  \leq C\sum_{i=0}^1
\sum_{p+q=m} \mu_{k+1}^{p-i-2}\delta_{k}\mu_{k+1}^{q+i+1} \leq C
\delta_k\mu_{k+1}^{m-1} \\ & \mbox{for all } m=0\ldots (3N+6)(K-k)-3N-4.
\end{split}
\end{equation}
We are now ready to bound the terms in $\dbar{\mathcal{R}}_2$. Observe first
that, by (\ref{Ebounds2b}), \ref{s3Ud}, \ref{s4Ud}:
\begin{equation*}
\begin{split}
&\|\nabla^{(m+1)}(T_{\bar U}\dbar W)\|_0\leq C\delta_k\mu_{k+1}^m
\quad \;\;\mbox{ for } m=0\ldots (3N+6)(K-k)-3N-5,
\\ & \|\nabla^{(m+1)}(bE^2_{\bar U})\|_0\leq C\delta_k^{1/2}\kappa^{m+1}
\quad \mbox{ for } m=0\ldots (3N+6)(K-k)-3N-3,
\end{split}
\end{equation*}
resulting in all the terms in $\dbar{\mathcal{R}}_2$ being bounded by:
\begin{equation}\label{m15dbar}
\begin{split}
& \|\nabla^{(m)}\dbar{\mathcal{R}}_2\|_0\leq
C\delta_k^{3/2}\mu_{k+1}^{m} \sigma^{N/2}
\leq \frac{\delta_k}{\sigma^{N+1}}\mu_{k+1}^m\\ & \mbox{for all } m=0\ldots (3N+6)(K-k)-3N-5.
\end{split}
\end{equation}
The worst term, responsible for the first bound above is
$2\sym((\nabla \bar U)^T[(\partial_1T_{\bar U})\dbar W, (\partial_2T_{\bar U})\dbar W])$,
and the final inequality is due to having 
$C\sigma^{3N/2+1}\delta_k^{1/2}\leq 1$ 
according to the last assumption in \ref{Ass1} for
$\sigma\geq\underline\sigma$ large. To estimate $\dbar{\mathcal{G}}$,
we use (\ref{m16dbar}) with $i=1$, and (\ref{SUbounddbar}):
\begin{equation}\label{m17dbar}
\begin{split}
& \|\nabla^{(m)}\dbar{\mathcal{G}}\|_0 \leq
C \sum_{i=2}^4\sum_{p+q=m} \mu_{k+1}^{p-3} \|\nabla^{(q+2)} \dbar S_i\|_0
\leq C \sum_{p+q=m} \mu_{k+1}^{p-3}\delta_k\kappa^{q+3}\sigma^{N/2-2}
\\ & = C\frac{\delta_k}{(\mu_{k+1}/\kappa)^3}\mu_{k+1}^m\sigma^{N/2-2}
\leq C \frac{\delta_k}{\sigma^{3N/2}}\mu_{k+1}^m\sigma^{N/2-2}\leq 
\frac{\delta_k}{\sigma^{N+1}}\mu_{k+1}^m
\\ & \mbox{for all } m=0\ldots (3N+6)(K-k)-3N-5,
\end{split}
\end{equation}
when $\sigma\geq \underline\sigma$ large.
In conclusion, the bounds (\ref{m18bar}), (\ref{m13dbar}), 
(\ref{m15dbar}), (\ref{m17dbar})  in (\ref{m12dbar}) yield:
\begin{equation}\label{m18dbar}
\begin{split}
& \mathcal{D}(g_0-\delta_{k+1}H_0, u_{k+1}) = - \dbar G e_1\otimes e_1 + \dbar{\mathcal{E}} \\
&\mbox{where } \|\nabla^{(m)}\dbar{\mathcal{E}}\|_0  =
\big\|\nabla^{(m)}\Big(\bar{\mathcal{E}} -
\frac{1-\dbar\Gamma(\mu_{k+1}x_2)}{\mu_{k+1}^2}\nabla b\otimes\nabla b
- \dbar{\mathcal{R}}_1 - \dbar{\mathcal{R}}_2-\dbar{\mathcal{G}}\Big)\big\|_0
\\ & \qquad\quad \leq 11\frac{\delta_k}{\sigma^{N+1}}\mu_{k+1}^m 
\quad \mbox{ for all } m=0\ldots (3N+6)(K-k)-3N-5.
\end{split}
\end{equation}

\smallskip

{\bf 12. (Estimating the error term $\dbar G$)}
In order to estimate the derivatives of $\dbar G$, we write:
$$\dbar G = \dbar G_2 + \dbar G_3 + \dbar G_4,$$
and analyze the following three terms: 
\begin{equation*}
\begin{split}
& \dbar G_2 = \sum_{i=0}^1 (-1)^i \frac{(\Gamma'\Gamma)_{i}(\mu_{k+1}x_2
  )}{\mu_{k+1}^{i+1}}P_i^{\eta_3}(\dbar S_2), \qquad
\dbar G_3 =\sum_{i=0}^1 (-1)^i \frac{\Gamma_{i}(\mu_{k+1}x_2
  )}{\mu_{k+1}^{i+1}}P_i^{\eta_3}(\dbar S_3), 
\\ & \dbar G_4 =\sum_{i=0}^1 (-1)^i \frac{\bar\Gamma_{i}(\mu_{k+1} x_2
  )}{\mu_{k+1}^{i+1}}P_i^{\eta_3}(\dbar S_4). 
\end{split}
\end{equation*}
We will use the formulas in (\ref{coef_ibp3}) at $i=0,1$, namely:
\begin{equation}\label{fora}
P_0^{\eta_3}(S) = S_{11}, \qquad P_1^{\eta_3}(S)=2\partial_1S_{12},
\end{equation}
separately for $S$ equal to $\dbar S_2, \dbar S_3$ and $\dbar S_4$.
In case of $\dbar S_2$, there holds:
$$P_0^{\eta_3}(\dbar S_2) = 0, \qquad P_1^{\eta_3}(\dbar S_2) = 4\partial_1(b\partial_1b),$$
so, we estimate directly from (\ref{Ebounds2b}):
\begin{equation}\label{dbG2}
\begin{split}
& \|\nabla^{(m)}\dbar G_2\|_0  \leq C \sum_{p+q=m}
\mu_{k+1}^{p-2}\|\nabla^{(q+1)}(b\partial_1b)\|_0 \leq 
C \sum_{p+q=m} \mu_{k+1}^{p-2}\frac{\delta_k}{\sigma}\kappa^{q+2}
\\ & \leq C \frac{\delta_k}{\sigma^{N+1}}\mu_{k+1}^m \quad  \mbox{ for all } 
m=0\ldots (3N+6)(K-k)-3N-4.
\end{split}
\end{equation}

\smallskip

\noindent In case of $\dbar S_4$, formulas (\ref{fora}) become:
$$P_0^{\eta_3}(\dbar S_4) = 2\langle\partial_1\bar U, \partial_1(b^2
T_{\bar U} e_2)\rangle,  \quad P_1^{\eta_3}(\dbar S_4) = 
2 \Big(\partial_1\langle\partial_1\bar U, \partial_2(b^2 T_{\bar U} e_2)\rangle 
+ \partial_1\langle\partial_2\bar U, \partial_1(b^2 T_{\bar U} e_2)\rangle\Big).$$
We note in passing that $\langle \partial_i\bar U, T_{\bar
  U}e_j\rangle = \langle e_i, e_j\rangle = \delta_{ij}$, which yields that:
\begin{equation*}
\begin{split}
 \big\langle \partial_i \bar U, \partial_j(b^2 T_{\bar U} e_2)\big\rangle
& = \partial_j(b^2) \big\langle \partial_i \bar U, T_{\bar U} e_2\big\rangle
+ b^2 \big\langle \partial_i \bar U, \partial_j(T_{\bar U} e_2)\big\rangle
\\ & = \partial_j(b^2) \delta_{i2} 
- b^2 \big\langle \partial_{ij} \bar U, T_{\bar U} e_2\big\rangle
\quad\mbox{ for all } i,j=1\ldots 2.
\end{split}
\end{equation*}
Consequently, we get from (\ref{Ebounds2b}), \ref{s2Ud}, \ref{s3Ud}:
\begin{equation*}
\begin{split}
&\|\nabla^{(m)}P_0^{\eta_3}(\dbar S_4)\|_0 = 2\big\|\nabla^{(m)}\big(b^2
\big\langle \partial_{11} \bar U, T_{\bar U}
e_2\big\rangle\big)\big\|_0 \\ & \qquad\qquad\qquad \quad \leq 
C\hspace{-2mm} \sum_{p+q+t=m} \vspace{-3mm}
\delta_k\kappa^p\delta_{k}^{1/2}\kappa^{q+1}\sigma^{N/2}\kappa^t\leq 
C \delta_k^{3/2}\kappa^{m+1}\sigma^{N/2},
\\ & \|\nabla^{(m)}P_1^{\eta_3}(\dbar S_4)\|_0\leq
C\big( \|\nabla^{(m+2)} b^2\|_0+
\|\nabla^{(m+1)}\big(b^2\big\langle \partial_{12} \bar U, T_{\bar U}
e_2\big\rangle\big)\big\|_0 \big) \\ &\qquad\qquad\qquad\quad
\leq C\frac{\delta_k}{\sigma}\kappa^{m+2} + C \delta_k^{3/2}\kappa^{m+2}\sigma^{N/2}
\leq C\frac{\delta_k}{\sigma}\kappa^{m+2},
\\ & \mbox{for all } m=0\ldots (3N+6)(K-k)-3N-4,
\end{split}
\end{equation*}
since $\delta_k^{1/2}\sigma^{N/2+1}\leq 1$ by the fourth assumption in \ref{Ass1}.
We are ready to conclude that:
\begin{equation}\label{dbG4}
\begin{split}
& \|\nabla^{(m)}\dbar G_4\|_0 \leq C \sum_{p+q=m}
\mu_{k+1}^{p-1}\delta_k^{3/2}\sigma^{N/2}\kappa^{q+1} + C \sum_{p+q=m}
\mu_{k+1}^{p-2}\frac{\delta_k}{\sigma}\kappa^{q+2} 
 \\ & \leq  C \delta_k^{3/2}\mu_{k+1}^m  + C\frac{\delta_k}{\sigma^{N+1}}\mu_{k+1}^m\leq
C \frac{\delta_k}{\sigma^{N+1}}\mu_{k+1}^m  \\ & \mbox{for all }m=0\ldots (3N+6)(K-k)-3N-4,
\end{split}
\end{equation}
valid provided that $\delta_k^{1/2}\sigma^{N+1}\leq 1$, as usual from  \ref{Ass1}. 

\medskip

\noindent We now deal with the term $\dbar G_3$, where the induction assumptions
\ref{Q2}-\ref{Q4} will be used for the first time in a tight manner. Given
$\dbar S_3$, from (\ref{fora}) and (\ref{S_third}) we read:
$$P_0^{\eta_3}(\dbar S_3) = 2b \langle \partial_{11}\bar U, E^2_{\bar U}\rangle, 
\qquad P_1^{\eta_3}(\dbar S_2) =
4\partial_1\big(b\langle\partial_{12}\bar U, E^2_{\bar U}\rangle\big).$$
Observe, from \ref{Q2} and \ref{Q3}, that:
\begin{equation*}
\begin{split}
& \|\nabla^{(m)}\langle \partial_{11}u_k, E^2_{u_k}\rangle \|_0\leq
C\frac{\delta_{k-1}^{1/2}}{\sigma^N}\mu_k^{m+1} \leq 
C\frac{\delta_{k}^{1/2}}{\sigma^{N/2+2}}\kappa^{m+1}, \\ 
& \|\nabla^{(m)}\langle \partial_{12}u_k, E^2_{u_k}\rangle \|_0\leq
C\frac{\delta_{k-1}^{1/2}}{\sigma^{N/2}}\mu_k^{m+1} \leq 
C\frac{\delta_{k}^{1/2}}{\sigma^{2}}\kappa^{m+1} \\
& \mbox{for all } \quad k=1\ldots K-1, \quad m=0\ldots (3N+6)(K-k).
\end{split}
\end{equation*}
In view of \ref{Q2Ud}, this implies:
\begin{equation*}
\begin{split}
& \|\nabla^{(m)}\langle \partial_{11}\bar U, E^2_{\bar U}\rangle\|_0\leq 
C\frac{\delta_{k}^{1/2}}{\sigma^{N/2+2}}\kappa^{m+1} +
C\delta_k\kappa^{m+1}\sigma^{N/2}\leq C\frac{\delta_{k}^{1/2}}{\sigma^{N/2+2}}\kappa^{m+1},\\ 
& \|\nabla^{(m)}\langle \partial_{12}\bar U, E^2_{\bar U}\rangle\|_0\leq 
C\frac{\delta_{k}^{1/2}}{\sigma^{2}}\kappa^{m+1} +
C\delta_k\kappa^{m+1}\sigma^{N/2}\leq C\frac{\delta_{k}^{1/2}}{\sigma^{2}}\kappa^{m+1}
\\ & \mbox{for all }\quad k=1\ldots K-1, \quad m=0\ldots (3N+6)(K-k)-3N-3,
\end{split}
\end{equation*}
where we have used that $\sigma_k^{1/2}\sigma^{N+2}\leq 1$ by \ref{Ass1}.
Consequently, recalling (\ref{Ebounds2b}), we see that:
\begin{equation*}
\begin{split}
& \|\nabla^{(m)}P_0^{\eta_3}(\dbar S_3)\|_0 \leq C
\|\nabla^{(m)}\big(b \langle \partial_{11}U, E^2_{U}\rangle\big)\|_0\leq
C\frac{\delta_{k}}{\sigma^{N/2+2}}\kappa^{m+1},
\\ & \|\nabla^{(m)}P_1^{\eta_3}(\dbar S_3)\|_0 \leq C
\|\nabla^{(m+1)}\big(b \langle \partial_{12}U, E^2_{U}\rangle\big)\|_0\leq
 C\frac{\delta_{k}}{\sigma^2}\kappa^{m+2} \\ & \mbox{for all }
 m=0\ldots (3N+6)(K-k)-3N-4,\;k=1\ldots K-1.
\end{split}
\end{equation*}
This leads to:
\begin{equation}\label{dbG3uno}
\begin{split}
& \|\nabla^{(m)}\dbar G_3\|_0 \leq
C\sum_{p+q=m}\mu_{k+1}^{p-1}\frac{\delta_{k}}{\sigma^{N/2+2}}\kappa^{q+1}
+ C \sum_{p+q=m}\mu_{k+1}^{p-2}\frac{\delta_{k}}{\sigma^{2}}\kappa^{q+2}
\\ & \leq C \frac{\delta_k}{\sigma^{N+2}}\mu_{k+1}^m  \quad\mbox{ for
  all } m=0\ldots (3N+6)(K-k)-3N-4,\;k=1\ldots K-1.
\end{split}
\end{equation}
Now, at $k=0$ we get by (\ref{m8}) and \ref{Q2Ud}:
\begin{equation*}
\begin{split}
& \|\nabla^{(m)}\langle \partial_{ij}\bar U,
E^2_{\bar U}\rangle\|_0\leq\|\nabla^{(m)} ((\nabla u_0)^T\nabla E^2_{u_0})
\|_0 + C\delta_k\kappa^{m+1}\sigma^{N/2} \\ & \leq  
C\delta_0^{1/2}\mu_0^{m+1} +C\delta_k\kappa^{m+1}\sigma^{N/2} \leq 
C\frac{\delta_{k}^{1/2}}{\sigma^{2}}\kappa^{m+1}   \\ & \mbox{for
  all } m=0\ldots (3N+6)(K-k)-3N-3,\; i,j=1\ldots 2,\; k=0,
\end{split}
\end{equation*}
so we directly note the same bound as in (\ref{dbG3uno}):
\begin{equation}\label{dbG3due}
\begin{split}
& \|\nabla^{(m)}\dbar G_3\|_0 \leq
C\sum_{p+q=m}\mu_{k+1}^{p-1}\frac{\delta_{k}}{\sigma^{2}}\kappa^{q+1}
+ C \sum_{p+q=m}\mu_{k+1}^{p-2}\frac{\delta_{k}}{\sigma^{2}}\kappa^{q+2}
\leq C \frac{\delta_k}{\sigma^{N+2}}\mu_{k+1}^m  \\ & 
\mbox{for all } m=0\ldots (3N+6)(K-k)-3N-4.
\end{split}
\end{equation}
In summary, by (\ref{dbG2}), (\ref{dbG4}), (\ref{dbG3uno}), (\ref{dbG3due}) we get:
\begin{equation}\label{m20dbar}
\|\nabla^{(m)}\dbar G\|_0  \leq  C \frac{\delta_k}{\sigma^{N+1}}\mu_{k+1}^m  
\quad \mbox{ for all } m=0\ldots (3N+6)(K-k)-3N-4,
\end{equation}
which yields, upon recalling  (\ref{m18dbar}) and for $\sigma$ is sufficiently large:
\begin{equation}\label{m18dbar_final}
\begin{split}
& \|\nabla^{(m)}\mathcal{D}(g_0-\delta_{k+1}H_0, u_{k+1})\|_0 
\leq C\frac{\delta_k}{\sigma^{N+1}}\mu_{k+1}^m  \leq
\frac{\delta_k}{\sigma^{N+1/2}}\mu_{k+1}^m \\ & \mbox{for all } 
m=0\ldots (3N+6)(K-k)-3N-5.
\end{split}
\end{equation}
We note that the regularity exponent is consistent with the induction
statement, as $(3N+6)(K-k)-3N-5= 1+(3N+6)(K-(k+1))$.

\smallskip

{\bf 13. (Propagation of bounds in the third corrugation and 
closing the bounds \ref{Q1}, \ref{Q4})}
In the last two steps we will validate all inductive
bounds \ref{P1}-\ref{Q4} at:
$$k+1=1\ldots K.$$
Note that \ref{P5} has already been shown in (\ref{m18dbar_final}).
Using (\ref{Ebounds2}), \ref{s3Ud}, \ref{s4Ud},
(\ref{m14dbar}) yields:
\begin{equation}\label{fafa}
\begin{split} 
& \|\nabla^{(m)} (u_{k+1}-\bar U)\|_0\leq C\sum_{p+q+t=m}\mu_{k+1}^{p-1}
\|\nabla^{(q)}b\|_0\|\nabla^{(t)}E^2_{\bar U}\|_0  \\ & \qquad +  C\sum_{p+q+t=m}\mu_{k+1}^{p-1}
\|\nabla^{(q)}b^2\|_0\|\nabla^{(t)}T_{\bar U}\|_0 + C\sum_{p+q=m}
\|\nabla^{(p)}T_{\bar U}\|_0 \|\nabla^{(q)}\dbar W\|_0 \leq C \delta_k^{1/2}\mu_{k+1}^{m-1}.
\end{split}
\end{equation}
Hence, \ref{P2} at $k+1$ follows directly
from the above and (\ref{mi1}), (\ref{mi2}).
Apply now Lemma \ref{lem_propa} with $u=\bar U$, $v=u_{k+1}$, $\mu=\mu_{k+1}$,
$A=\delta_k^{1/2}$. Indeed, in virtue of the bound displayed above,
there holds $\|\nabla
u_{k+1}-\nabla \bar U\|_0\leq C\delta_k^{1/2}\leq C{\underline\delta}^{1/2}\leq
\rho$ provided that $\underline\delta$ is sufficiently small. Recalling
further (\ref{n2d}), \ref{s2Ud}, \ref{s4Ud}, we obtain existence of
the orthonormal frame $E^1_{u_{k+1}},
E^2_{u_{k+1}}\in\mathcal{C}^{(3N+6)(K-k)-3N-5}(\bar\omega, \R^4)$, namely:
$$(\nabla u_{k+1})^TE^i_{u_{k+1}}=0,\quad |E^1_{u_{k+1}}|=1\quad\mbox{ for }\quad 
i=1,2\quad\mbox{ and } \quad \langle E^1_{u_{k+1}},
E^2_{u_{k+1}}\rangle=0\quad\mbox{in } \bar\omega,$$ 
satisfying the the second bound below, while the first bound is
included in (\ref{fafa}):
\begin{equation}\label{n1dbarf}
\begin{split}
& \|\nabla^{(m+1)}(u_{k+1}-\bar U)\|_0\leq C\delta_k^{1/2}\mu_{k+1}^{m}
\\ & \mbox{and}\quad \|\nabla^{(m)}(E^i_{u_{k+1}}-E^i_{\bar U})\|_0\leq C\delta_k^{1/2}\mu_{k+1}^{m}
\quad\mbox{ for } m=0\ldots (3N+6)(K-k)-3N-5.%\; i=1,2.
\end{split}
\end{equation}
Together with (\ref{n1}) and (\ref{n1dbarf}), we immediately conclude
\ref{P3} and \ref{P4} at the counter $k+1$.
To show \ref{P1} at $k+1$, recall (\ref{m3UUd}) and write:
\begin{equation}\label{m3UUdf}
\|u_{k+1}-u\|_1\leq \|u_{k+1} - \bar U\|_1 + \|U-u\|_0 \leq C\delta_0^{1/2},
\end{equation}
which indeed leads, as in step 1 and the proof of (\ref{m5}), to:
\begin{equation}\label{n2d_fin}
\frac{1}{3\underline\gamma}\Id_2\leq (\nabla u_{k+1})^T\nabla u_{k+1}\leq
3\underline\gamma\Id_2 \quad\mbox{ in } \bar\omega.
\end{equation}
It remains to show the coordinate-specific bounds \ref{Q1}-\ref{Q4}.
Using Lemma \ref{lem_det} and (\ref{n2d_fin}), together with (\ref{n1dbarf}) in
view of \ref{s2Ud}, \ref{s4Ud}, we observe that:
\begin{align*}
& \|\nabla u_{k+1}\|_0\leq (6\underline\gamma)^{1/2}
\tag*{(\theequation)$_1$}\refstepcounter{equation} \label{s1Udf} \\
& \|\nabla^{(m)}\nabla^{(2)} u_{k+1}\|_0\leq 
C\delta_{k}^{1/2}\mu_{k+1}^{m+1} \quad\,\;\mbox{for }m=0\ldots (3N+6)(K-k)-3N-6,
\tag*{(\theequation)$_2$} \label{s2Udf} \\
& \|\nabla^{(m)}E^i_{u_{k+1}}\|_0\leq
C\delta_k^{1/2}\mu_{k+1}^m\quad \mbox{for } m=1\ldots (3N+6)(K-k)-3N-5,\;i=1,2. 
\tag*{(\theequation)$_3$} \label{s4Udf} 
\end{align*}
At this point, the estimate \ref{Q4} at the counter $k+1$ already
follows, directly from the above:
\begin{equation*}
\begin{split}
& \|\nabla^{(m)}\big((\nabla u_{k+1})^T \nabla E^1_{u_{k+1}}\big)\|_0 \leq
C\sum_{p+q=m} \mu_{k+1}^p\delta_{k}^{1/2}\mu_{k+1}^{q+1}\leq C \delta_{k}^{1/2}\mu_{k+1}^{m+1}
\\ & \mbox{for all } m=0\ldots (3N+6)(K-k)-3N-6.
\end{split}
\end{equation*}
To demonstrate the finer bounds \ref{Q1}-\ref{Q3}, we decompose, for any $i,j,s=1\ldots 2$:
\begin{equation}\label{eqeq_fin}
\langle \partial_{ij} u_{k+1}, E^s_{u_{k+1}}\rangle =
\langle \partial_{ij} \bar U, E^s_{\bar U}\rangle
+ \langle\partial_{ij}(u_{k+1}-\bar U), E^s_{\bar U}\rangle +
\langle\partial_{ij}u_{k+1}, (E^s_{u_{k+1}}-E^s_{\bar U})\rangle. 
\end{equation}
For the third term above, we recall (\ref{n1dbarf}) and \ref{s2Udf} and get:
\begin{equation}\label{1_term_fin}
\begin{split}
& \|\nabla^{(m)}\langle\partial_{ij}u_{k+1}, (E^s_{u_{k+1}}-E^s_{\bar U})\rangle\|_0 \leq
C\delta_k\mu_{k+1}^{m+1} \\ & \mbox{for } m=0\ldots (3N+6)(K-k)-3N-6,\; i,j,s=1\ldots 2.
\end{split}
\end{equation}
For the second term in the right hand side of (\ref{eqeq_fin}) and in the
specific case of $s=1$, we write: 
\begin{equation*}
\begin{split}
\langle\partial_{ij}(u_{k+1}-\bar U), E^1_{\bar U} \rangle = & ~\partial_i
\Big(\frac{\Gamma(\mu_{k+1}x_2)}{\mu_{k+1}}
b\Big)\langle \partial_jE^2_{\bar U},
E^1_{\bar U}\rangle + \partial_j\Big(\frac{\Gamma(\mu_{k+1}
  x_2)}{\mu_{k+1}} b\Big)\langle \partial_iE^2_{\bar U},
E^1_{\bar U}\rangle \\ & + \frac{\Gamma(\mu_{k+1}x_2)}{\mu_{k+1}} b
\langle \partial_{ij}E^2_{\bar U},
E^1_{\bar U}\rangle + \Big\langle \partial_{ij}\Big(T_{\bar U}\Big(\frac{\bar\Gamma(\mu_{k+1}
 x_2)}{\mu_{k+1}}b^2e_2+\dbar W\Big)\Big), E^1_{\bar U}\Big\rangle,
\end{split}
\end{equation*}
using that $\langle E^1_{\bar U}, E^2_{\bar U}\rangle=0$. 
Consequently, by (\ref{Ebounds2}), \ref{s3Ud}, \ref{s4Ud}, (\ref{m14dbar}):
\begin{equation}\label{kalb0}
\begin{split}
& \|\nabla^{(m)}\langle\partial_{ij}(u_{k+1} - \bar U), E^1_{\bar U}\rangle\|_0 
\\ & \leq  C\sum_{p+q+t=m}\big\|\nabla^{(p+1)}\Big (\frac{\Gamma(\mu_{k+1}
  x_2)}{\mu_{k+1}} b\Big)\big\|_0 \|\nabla^{(q+1)}E^2_{\bar U}\|_0\|\nabla^{(t)}E^1_{\bar U}\|_0
\\ & \quad +\sum_{p+q+t= m} \big\|\nabla^{(p)}\Big (\frac{\Gamma(\mu_{k+1}
  x_2)}{\mu_{k+1}} b\Big)\big\|_0 \|\nabla^{(q+2)}E^2_{\bar  U}\|_0\|\nabla^{(t)}E^1_{\bar U}\|_0
\\ & \quad + C\sum_{p+q=m}\sum_{~t+s=p+2}\|\nabla^{(t)}T_{\bar U}\|_0
\big\|\nabla^{(s)}\Big (\frac{\bar\Gamma(\mu_{k+1} x_2)}{\mu_{k+1}} b^2e_2 +
\dbar W\Big)\big\|_0\|\nabla^{(q)}E^1_{\bar U}\|_0
\\ & \leq  C\delta_k\mu_{k+1}^{m+1} 
\quad\mbox{for all } m=0\ldots (3N+6)(K-k)-3N-6, \; i,j=1\ldots 2.
\end{split}
\end{equation}
In conclusion, (\ref{eqeq_fin}), (\ref{1_term_fin}) and \ref{Q1Ud} imply:
\begin{equation}\label{21_ter_fin}
\begin{split}
\|\nabla^{(m)}\langle \partial_{ij} u_{k+1}, E^1_{u_{k+1}}\rangle \|_0
&  \leq \|\nabla^{(m)}\langle \partial_{ij} \bar U, E^1_{\bar U}\rangle\|_0+
C\delta_k\mu_{k+1}^{m+1} \\ & \leq  
\|\nabla^{(m)}\langle \partial_{ij} u_k, E^1_{u_k}\rangle\|_0+
C \delta_k^{1/2}\kappa^{m+1}+C\delta_k\mu_{k+1}^{m+1} 
\\ & \leq \|\nabla^{(m)}\langle \partial_{ij} u_k, E^1_{u_k}\rangle\|_0 +
C \frac{\delta_k^{1/2}}{\sigma^{N/2}}\mu_{k+1}^{m+1}.
\end{split}
\end{equation}
From \ref{Q1} we recall that:
\begin{equation*}
\begin{split}
& \|\nabla^{(m)}\langle \partial_{ij} u_k, E^1_{u_k}\rangle\|_0\leq
C\frac{\delta_{k-1}^{1/2}}{\sigma^{N/2}}\mu_k^{m+1}\leq C
\frac{\delta_k^{1/2}}{\sigma^{N/2}}\mu_{k+1}^{m+1}
\\ & \mbox{for all } m=0\ldots (3N+6)(K-k)-3N-6, \; i,j=1\ldots 2,\; k=1\ldots K-1,
\end{split}
\end{equation*}
whereas at $k=0$ we likewise get by (\ref{m8}):
\begin{equation}\label{kalb12}
\begin{split}
& \|\nabla^{(m)}\langle \partial_{ij} u_k, E^s_{u_k}\rangle\|_0\leq
C\delta_0^{1/2}\mu_0^{m+1}\leq C
\frac{\delta_k^{1/2}}{\sigma^{N}}\mu_{k+1}^{m+1}
\\ & \mbox{for all } m=0\ldots (3N+6)K, \; i,j,s=1\ldots 2,\; k=0.
\end{split}
\end{equation}
Combining with (\ref{21_ter_fin}), in either case we obtain the
validity of \ref{Q1} at the $k+1$ counter.

\smallskip

{\bf 14. (Closing the inductive bounds \ref{Q2}, \ref{Q3})}
Finally, for the second term in the right hand side of (\ref{eqeq_fin}) and in the
specific case of $s=2$, we write:
\begin{equation}\label{kalb1}
\begin{split}
\langle\partial_{ij}(u_{k+1}-\bar U), E^2_{\bar U} \rangle = & ~\partial_{ij}
\Big(\frac{\Gamma(\mu_{k+1}x_2)}{\mu_{k+1}}
b\Big) - \frac{\Gamma(\mu_{k+1}x_2)}{\mu_{k+1}} b
\langle \partial_{j}E^2_{\bar U},
\partial_j E^1_{\bar U}\rangle \\ & + \Big\langle \partial_{ij}\Big(T_{\bar U}\Big(\frac{\bar\Gamma(\mu_{k+1}
 x_2)}{\mu_{k+1}}b^2e_2+\dbar W\Big)\Big), E^2_{\bar U}\Big\rangle,
\end{split}
\end{equation}
where we used that $\langle E^2_{\bar U}, E^2_{\bar U} \rangle = 1$.
Derivatives of the last two terms in the right hand side of (\ref{kalb1}) may be
bounded as in (\ref{kalb0}),  by (\ref{Ebounds2b}), \ref{s3Ud},
\ref{s4Ud}, (\ref{m14dbar}):
\begin{equation}\label{kalb2}
\begin{split}
& \big\|\nabla^{(m)}\Big(\frac{\Gamma(\mu_{k+1}x_2)}{\mu_{k+1}} b
\langle \partial_{j}E^2_{\bar U},
\partial_j E^1_{\bar U}\rangle\Big)\big\|_0 \\ & +
\big\|\nabla^{(m)}\Big(\Big\langle \partial_{ij}\Big(T_{\bar
  U}\Big(\frac{\bar\Gamma(\mu_{k+1} 
 x_2)}{\mu_{k+1}}b^2e_2+\dbar W\Big)\Big), E^2_{\bar
 U}\Big\rangle\Big)\big\|_0 \leq C\delta_k\mu_{k+1}^{m+1}.
\end{split}
\end{equation}
We analyze the first term in the right hand side of (\ref{kalb1}):
\begin{equation}\label{kalb3}
\begin{split}
& \big\|\nabla^{(m)}\partial_{11} \Big(\frac{\Gamma(\mu_{k+1}x_2)}{\mu_{k+1}}
b\Big)\big\|_0\leq C\sum_{p+q=m}\mu_{k+1}^{p-1}\|\nabla^{(q+2)}b\|_0
\\ & \qquad\qquad\qquad \qquad\qquad\quad \; \leq C \frac{\delta_k^{1/2}}
{(\mu_{k+1}/\kappa)^2}\mu_{k+1}^{m+1}\leq C \frac{\delta_k^{1/2}}{\sigma^{N}}\mu_{k+1}^{m+1},
\\ & \big\|\nabla^{(m)}\partial_{12} \Big(\frac{\Gamma(\mu_{k+1}x_2)}{\mu_{k+1}}
b\Big)\big\|_0 \leq C\sum_{p+q=m}\mu_{k+1}^{p-1}\|\nabla^{(q+2)}b\|_0
+ C\sum_{p+q=m}\mu_{k+1}^{p}\|\nabla^{(q+1)}b\|_0
\\ & \qquad\qquad\qquad \qquad\qquad\quad \; 
\leq C \frac{\delta_k^{1/2}} {(\mu_{k+1}/\kappa)}\mu_{k+1}^{m+1}\leq C
\frac{\delta_k^{1/2}}{\sigma^{N/2}}\mu_{k+1}^{m+1}.
\end{split}
\end{equation}
In conclusion, (\ref{eqeq_fin}), (\ref{1_term_fin}), (\ref{kalb1}),
(\ref{kalb2}), (\ref{kalb3}) and having $\delta_k^{1/2}\sigma^{N}\leq 1$, imply:
\begin{equation}\label{22_ter_fin}
\begin{split}
& \|\nabla^{(m)}\langle \partial_{11} u_{k+1}, E^2_{u_{k+1}}\rangle \|_0
 \leq \|\nabla^{(m)}\langle \partial_{11} u_k, E^2_{u_k}\rangle\|_0+
C \frac{\delta_k^{1/2}}{\sigma^{N}}\mu_{k+1}^{m+1}, \\
 & \|\nabla^{(m)}\langle \partial_{12} u_{k+1}, E^2_{u_{k+1}}\rangle \|_0
 \leq \|\nabla^{(m)}\langle \partial_{12} u_k, E^2_{u_k}\rangle\|_0 + 
C \frac{\delta_k^{1/2}}{\sigma^{N/2}}\mu_{k+1}^{m+1} 
\\ & \mbox{for all } m=0\ldots (3N+6)(K-k)-3N-6.
\end{split}
\end{equation}
We now claim that both first terms in the right hand sides above are
bounded by the respective second terms. Indeed, for $k\geq 1$
this follows from \ref{Q2}, \ref{Q3}, namely:
\begin{equation*}
\begin{split}
& \|\nabla^{(m)}\langle \partial_{11} u_k, E^2_{u_k}\rangle\|_0\leq 
C \frac{\delta_{k-1}^{1/2}}{\sigma^{N}}\mu_{k}^{m+1}
\leq C \frac{\delta_k^{1/2}}{\sigma^{N}}\mu_{k+1}^{m+1}, \\
 & \|\nabla^{(m)}\langle \partial_{12} u_k, E^2_{u_k}\rangle\|_0 \leq
C \frac{\delta_{k-1}^{1/2}}{\sigma^{N/2}}\mu_{k}^{m+1}
\leq C \frac{\delta_k^{1/2}}{\sigma^{N/2}}\mu_{k+1}^{m+1},
\\ & \mbox{for all } m=0\ldots (3N+6)(K-k),\; k=1\ldots K-1, 
\end{split}
\end{equation*}
whereas at $k=0$ the same bounds follow by (\ref{kalb12}).
Therefore, (\ref{22_ter_fin}) become:
\begin{equation*}
\begin{split}
& \|\nabla^{(m)}\langle \partial_{11} u_{k+1}, E^2_{u_{k+1}}\rangle \|_0
 \leq C \frac{\delta_k^{1/2}}{\sigma^{N}}\mu_{k+1}^{m+1}, \\
 & \|\nabla^{(m)}\langle \partial_{12} u_{k+1}, E^2_{u_{k+1}}\rangle \|_0
 \leq C \frac{\delta_k^{1/2}}{\sigma^{N/2}}\mu_{k+1}^{m+1} 
\quad\mbox{ for all } m=0\ldots (3N+6)(K-k)-3N-6,
\end{split}
\end{equation*}
which is exactly \ref{Q2} and \ref{Q3} at the counter $k+1$. The proof
of Theorem \ref{thm_STA} is complete.
\endproof

\section{Nash-Kuiper's scheme: the proof of Theorem \ref{thm_NK} and
  of Theorem \ref{thm_main}} \label{section_NK}

We now exhibit the details of the Nash-Kuiper iteration
scheme. The proof uses the double exponential ansatz on the progression of
frequencies and defect measures that appeared in \cite{DIS1/5},
and was similarly applied in \cite[proof of Theorem 1.1]{CHI2}.

\medskip

\noindent {\bf Proof of Theorem \ref{thm_NK}}

\smallskip

{\bf 1.} Fix $\alpha$ as in (\ref{ran_NK}) and $\epsilon \in (0,1)$. 
We will obtain $\bar u$ as the limit of 
$\{u_n\in\mathcal{C}^{2}(\bar\omega,\R^4)\}_{n=1}^\infty$ converging in
$\mathcal{C}^{1,\alpha}(\bar\omega,\R^4)$. These immersions will be constructed iteratively,
starting from $u_0$ obtained by Theorem
\ref{lem_stage0}, and then by successive application of (\ref{STA}) with:
$$u=u_{n},\qquad \delta =\delta_n, \qquad\mu=\mu_n, \qquad\sigma=\sigma_{n+1},$$
producing $\tilde u = u_{n+1}$. The progression of the above
parameters is as follows. Fix $\theta$ and $\tau$ in:
\begin{equation}\label{th_t}
\alpha<\theta<\min\Big\{\frac{r+\beta}{2}, \frac{1}{1+
  2J/S}\Big\},\qquad \tau = \underline\tau + 1,
\end{equation}
where $\underline\tau$ is as in Theorem \ref{lem_stage0}. In
particular, there holds:
\begin{equation}\label{ko1}
\frac{J}{S}<\frac{1}{2\theta} - \frac{1}{2} \quad \mbox{ and }\quad \tau>2+\frac{1}{r+\beta}.
\end{equation}
We define for all $n\geq 0$:
\begin{equation}\label{ko2}
\delta_n =\frac{1}{a^{b^n}},\quad \mu_n = a^{\tau+\frac{b^n-1}{2
    \theta}},\quad \sigma_{n+1} =
\Big(\frac{\delta_n}{\delta_{n+1}}\Big)^{1/S} \quad\mbox{so that }\,
\delta_{n+1} = \frac{\delta_n}{\sigma_{n+1}^S},
\end{equation}
for some  $a,b>1$, whose magnitudes are specified by the
requirements in the course of the proof. In general, $b-1>0$ will be 
sufficiently small and $a$ sufficiently large (in that order).

\smallskip

{\bf 2.} From (\ref{ko2}) we read the initial parameters:
$$\delta_0 = \frac{1}{a},\qquad \mu_0 = a^\tau.$$
For $a$ large so that $\delta_0<\underline\delta$, we apply Theorem \ref{lem_stage0} 
obtaining an immersion $u_0\in\mathcal{C}^2(\bar\omega,\R^4)$ in:
\begin{equation}\label{ko5}
\begin{split}
& \|\mathcal{D}(g-\delta_0H_0, u_0)\|_0\leq \frac{r_0}{4}\delta_0, \\
&\|u_0-\underline u\|_0 \leq C\delta_0\leq\frac{\epsilon}{2}, \qquad
\|\nabla u_0-\nabla\underbar u\|_0\leq
\frac{\rho_{2\underline\gamma}}{3}\leq \frac{\rho_{\underline\gamma}}{3},\\ 
& \|u_0\|_2\leq \|\underline u\|_2+ \Big(C\delta_0 + C +
\frac{C}{\delta_0^{\underline\tau}}\Big) \leq
\frac{C}{\delta_0^{\underline\tau}}\leq \delta_0^{1/2}\mu_0,
\end{split}
\end{equation}
where $C$ depends on $\omega, g, \underline u, \underline\gamma$.
The third bound above again follows for $a$ large enough (in function
of $\epsilon$) and the fourth bound follows in view of (\ref{th_t}) because:
$$\delta_0^{1/2+\underline\tau}\mu_0=a^{\tau-1/2- \underline\tau} = a^{1/2}\geq C,$$
likewise for $a$ sufficiently large (in function of the constant $C$
above that depends only on $\omega, \underline u, g$).
As a byproduct, we get: $\delta_0^{1/2}\mu_0\geq 1$.

\smallskip

{\bf 3.} We will show that if $b-1>0$ is
sufficiently small and $a$ is sufficiently large (in that order) then 
one can proceed with the induction on $n$. We first show that for all $n\geq 0$:
\begin{align*}
& \sigma_{n+1}\geq\underline\sigma,\qquad \sigma_{n+1}^p\delta_n^{1/2}\leq 1,
\qquad \delta_{n+1}\leq\underline\delta,\qquad \mu_{n+1}\delta_{n+1}^{1/2}\geq 1,
\tag*{(\theequation)$_1$}\refstepcounter{equation} \label{ko61} \\
& \frac{r_0}{5}\frac{\delta_n}{\sigma_{n+1}^S} +
\frac{\|g\|_{r,\beta}}{\mu_n^{r+\beta}}\leq \frac{r_0}{4}\delta_{n+1},
\tag*{(\theequation)$_2$} \label{ko62}\\
&  C\mu_n\delta_n^{1/2}\sigma_{n+1}^J\leq \delta_{n+1}^{1/2}\mu_{n+1},
\tag*{(\theequation)$_3$} \label{ko63}
\end{align*}
where $C$ is the constant appearing in (\ref{STA}).
In order to show the four bounds in \ref{ko61}, we use the definition
(\ref{ko2}). The first of the claimed bounds follows when $a$ is
sufficiently large (after possibly choosing $b-1$ sufficiently small):
$$\sigma_{n+1}=a^{(b^{n+1}-b^n)/S} = a^{b^n(b-1)/S}\geq a^{(b-1)/S}\geq\underline\sigma.$$
For the second bound, we have:
$$\sigma_{n+1}^p\delta_n^{1/2}= a^{b^n((b-1)p/S - 1/2)}\leq 1,$$
if only $(b-1)p/S\leq 1/2$, which is assured by taking $b-1$ small. The
third bound is clear since $\delta_{n+1}\leq \delta_0\leq
\underline\sigma$, while the fourth bound is due to:
$$ \mu_{n+1}\delta_{n+1}^{1/2} = a^{\tau+b^{n+1}((\frac{1}{2\theta} -
  \frac{1}{2}) - \frac{1}{2\theta}}\geq  a^{\tau+b((\frac{1}{2\theta} -
  \frac{1}{2}) - \frac{1}{2\theta}} = a^{(\tau-\frac{1}{2}) +
  (b-1)(\frac{1}{2\theta}-\frac{1}{2})}\geq 1$$
since $\frac{1}{2\theta}-\frac{1}{2}\geq 0$, so that, for $b-1$
sufficiently small, the positive term $\tau-\frac{1}{2}$ in last
exponent above prevails. This ends the proof of \ref{ko61}.
To show \ref{ko62}, we observe that, by definition:
$$\frac{r_0}{5}\frac{\delta_n}{\sigma_{n+1}^S} =\frac{r_0}{5}\delta_{n+1},$$ 
and thus it suffices to justify the smallness of
$\frac{\|g\|_{r,\beta}}{\mu_n^{r+\beta}}\frac{1}{\delta_{n+1}}$,
equivalent to the largeness of $\delta_{n+1}\mu_n^{r+\beta}$ (in
function of $\omega, \underline u, g, \theta$). Observe that $\frac{r+\beta}{2\theta}-1>0$
by the upper bound on $\theta$ in (\ref{th_t}). Therefore:
\begin{equation*} 
\begin{split}
\delta_{n+1}\mu_n^{r+\beta} & =a^{(\tau
  +\frac{b^n-1}{2\theta})(r+\beta)-b^{n+1} } = a^{b^n(\frac{r+\beta}{2\theta}-b) +
  (\tau-\frac{1}{2\theta}(r+\beta)} \\ & \geq
a^{(\frac{r+\beta}{2\theta}-b) +  (\tau-\frac{1}{2\theta}(r+\beta)}=
a^{\tau(r+\beta)-b}= a^{\tau(r+\beta) -1 -(b-1)},
\end{split}
\end{equation*}
for $b-1$ small. The above right hand side is as large as one wants, again by taking
$b-1$ small and $a$ large (in that order), because $\tau(r+\beta)-1>0$
by the second inequality in (\ref{ko1}). This ends the proof of
\ref{ko62}. Finally, \ref{ko63} is equivalent to the largeness of:
\begin{equation*} 
\frac{\mu_{n+1}}{\mu_n}\big(\frac{\delta_{n+1}}{\delta_n}\big)^{1/2}\frac{1}{\sigma_{n+1}^J}
= a^{b^n\frac{b-1}{2\theta} - b^n(b-1)/2-b^n(b-1)J/S} =
a^{b^n(b-1)(\frac{1}{2\theta}-\frac{1}{2}-\frac{J}{S}) }\geq
a^{(b-1)(\frac{1}{2\theta}-\frac{1}{2}-\frac{J}{S})},
\end{equation*}
where we used the first inequality in (\ref{ko1}). Clearly, for $a$
large (after $b-1>0$ has been priorly fixed), the right hand side
above is as large as desired. This ends the proof of \ref{ko63}. 

\smallskip

{\bf 4.} We thus see that as long as:
\begin{equation}\label{ko3}
\|\nabla u_n - \nabla \underline u\|_0\leq \frac{\rho_{\underline \gamma}}{2},
\end{equation}
one can define $u_{n+1}$ by the indicated application of (\ref{STA}). 
Noting the bound $b^i\geq 1+(\ln b)i$ valid for all $i\geq 0$ and $b\geq 1$, we then get: 
\begin{equation}\label{ko7} 
\begin{split}
\|u_{n+1}-u_0\|_1& \leq C\sum_{i=0}^{n}\delta_i^{1/2}= C\sum_{i=0}^{n}
\frac{1}{a^{b^i/2}}\leq C\sum_{i=0}^{n}
\Big(\frac{1}{a}\Big)^{1/2+(\ln b)i/2} \\ & \leq \frac{C}{a^{1/2}}\sum_{i=0}^\infty
\Big(\frac{1}{a}\Big)^{(\ln b)i/2} = C \frac{a^{(\ln b)/2}}{a^{1/2} (a^{(\ln b)/2}-1)}\leq\frac{\epsilon}{2},
\end{split}
\end{equation}
where the last inequality is valid for $b-1$ small and $a$ sufficiently large. 
In particular, for $\epsilon$ small enough (or, more generally, for a
sufficiently large $a$), the bound (\ref{ko3}) is valid for $u_{n+1}$ in
virtue of the third condition in (\ref{ko5}). Consequently, the infinite sequence
$\{u_n\in\mathcal{C}^2(\bar\omega,\R^4)\}_{n=0}^\infty$ is well
defined and it satisfies, for all $n\geq 0$:
\begin{equation}\label{ko8}
\begin{split}
& \|u_{n+1}-u_n\|_1\leq C\delta_{n}^{1/2}, \qquad \|u_{n}\|_2\leq \delta_{n}^{1/2}\mu_{n},\\
& \|\mathcal{D}(g-\delta_{n}H_0, u_{n})\|_0\leq \frac{r_0}{4}\delta_{n}.
\end{split}
\end{equation}

\smallskip

{\bf 5.} We are now ready to conclude the proof. Firstly, we show that
$\{u_n\}_{n=0}^\infty$ is Cauchy, hence converges in 
$\mathcal{C}^{1,\alpha}(\bar\omega, \R^4)$. To this end, we use the
interpolation inequality in H\"older spaces:
$$\|u_{n+1}-u_n\|_{1,\alpha}\leq C \|u_{n+1}-u_n\|_2^\alpha\|u_{n+1}-u_n\|_1^{1-\alpha}.$$
Since $\{\delta_n^{1/2}\mu_n\}_{n=1}^\infty$ is an increasing sequence,
the above and (\ref{ko8}) imply for all $n\geq 0$:
\begin{equation}\label{ko9}
\begin{split}
\|u_{n+1}-u_n\|_{1,\alpha} & \leq C(\delta_{n+1}^{1/2}
\mu_{n+1})^{\alpha}\delta_n^{(1-\alpha)/2} \\ & \leq C \big(a^{\tau +
  \frac{b^{n+1}-1}{2\theta}- b^{n+1}/2}\big)^\alpha a^{-b^n(1-\alpha)/2}\leq
C a^{(\tau-\frac{1}{2\theta})\alpha} a^{-b^n q}, 
\end{split}
\end{equation}
where $q>0$ is a positive constant, independent of $n$, given in:
$$q= \frac{1-\alpha}{2} - b\alpha\big(\frac{1}{2\theta}-\frac{1}{2}\big)
=\frac{1}{2}\big(1-\frac{\alpha}{\theta}\big)
- (b-1)\alpha\big(\frac{1}{2\theta}-\frac{1}{2}\big)>0,$$
if only $b-1$ is sufficiently small,
since  $1-\alpha/\theta>0$  in virtue of the lower bound in (\ref{th_t}). 
Consequently, (\ref{ko9}) implies the aforementioned Cauchy property in:
\begin{equation*}
\begin{split}
\|u_{n+1}-u_n\|_{1,\alpha} \leq C a^{(\tau-\frac{1}{2\theta})\alpha} a^{-q (1+(\ln b)n)} 
= C a^{(\tau-\frac{1}{2\theta})\alpha-q} \big(a^{\ln b}\big)^{-n}, 
\end{split}
\end{equation*}
where the terms in the right hand side above constitutes a converging
power series. In conclusion, $\bar
u\in\mathcal{C}^{1,\alpha}(\bar\omega,\R^4)$ may be defined as the limit: 
$$u_n\to \bar u \qquad\mbox{as }\; n\to\infty \quad \mbox{in }\;
\mathcal{C}^{1,\alpha}(\bar\omega,\R^4).$$
By (\ref{ko9}) combined with the second statement in (\ref{ko5}) there holds:
$$\|\bar u - \underline u\|_0\leq \|\bar u - u_0\|_0 + \|u_0 - \underline u\|_0
\leq \frac{\epsilon}{2} + \frac{\epsilon}{2} = \epsilon,$$
while by the last statement in (\ref{ko8}) we obtain:
$$\|\mathcal{D}(g, u_{n})\|_0 \leq \|\mathcal{D}(g- \delta_{n} H_0,
u_{n})\|_0 + \delta_n|H_0|\leq \Big(\frac{r_0}{4}+|H_0|\Big)\delta_n
\to 0 \quad\mbox{as }\; n\to\infty.$$
Since the left hand side above clearly converges to $\|\mathcal{D}(g,
\bar u)\|_0$ as $n\to\infty$, there follows:
$$\mathcal{D}(g, \bar u) = 0 \quad\mbox{ in }\; \bar\omega.$$
This ends the proof of Theorem \ref{thm_NK}.
\endproof

\medskip

We are now ready to conclude the main result of this paper:

\medskip

\noindent {\bf Proof of Theorem \ref{thm_main}}

\smallskip

Given $g$ and $u$ satisfying the assumptions, we first apply Theorem
\ref{lem_prestage} to obtain a referential immersability parameter
$\underline\gamma\geq \max\{1,\|g\|_0\}$
and a referential immersion $\underline
u\in\mathcal{C}^\infty(\bar\omega,\R^4)$, with:
\begin{equation*}
\begin{split}
& \frac{1}{\underline\gamma}\Id_2\leq (\nabla \underline u)^T \nabla \underline u\leq
\underline\gamma\,\Id_2 \quad\mbox{ and } \quad \mathcal{D}(g,\underline u)>0
\quad\mbox{ in }\; \bar\omega, \\
&\|\mathcal{D}(g,\underline u)\|_0\leq \frac{\rho_{2\underline\gamma}}{(16)^2N_0}
\quad\mbox{ and } \quad \|u-\underline u\|_0\leq \frac{\epsilon}{2}. 
\end{split}
\end{equation*}
Fix the exponent $\alpha$ in the indicated range, and let $N,K\geq 4$ be two
integers such that:
$$\frac{2K+N}{KN}< \frac{1}{2}\big(\frac{1}{\alpha}-1\big).$$
Invoking Theorems \ref{thm_NK} with $J=2K+1$, $S=NK$, $p= (3N+3)/2$,
and recalling Theorem \ref{thm_STA}, implies existence of an immersion
$\bar u\in\mathcal{C}^{1,\alpha}(\bar\omega, \R^4)$ such that:
\begin{equation*}
\begin{split}
& \mathcal{D}(g,\underline u)=0 \; \mbox{ in }\; \bar\omega
\quad\mbox{ and }\quad
\|\bar u- u\|_0 \leq \|\bar u-\underline u\|_0 +\|u-\underline u\|_0\leq \epsilon. 
\end{split}
\end{equation*}
The proof is complete.
\endproof

\section{Appendix: Poznyak's theorem}\label{sec_appendix}

For completeness, we present Poznyak's theorem and sketch its proof.

\begin{theorem}\label{thm_poz}\cite[Theorem 4]{Poz} %\cite[Theorem 2.3.1]{HH}
Let $\omega\subset\R^2$ be an open, bounded and simply connected set,
and let $g\in \mathcal{C}^\infty (\bar\omega, \R^{2\times 2}_{\sym, >})$. 
% be a given Riemann metric on $\bar\omega$. 
Then, there exists $u\in\mathcal{C}^\infty(\bar\omega, \R^4)$ which
solves (\ref{II_in}) on $\bar\omega$. 
%\begin{equation}\label{II}
%(\nabla u)^T\nabla u = g \qquad\mbox{in }\; \bar\omega.
%\end{equation}
\end{theorem}

\noindent In the proof, the smooth immersion $u$ satisfying (\ref{II_in}) is sought to be of the
following form, for some $\epsilon>0$ and $w\in\mathcal{C}^\infty(\bar\omega, \R)$, 
$v\in\mathcal{C}^\infty(\bar\omega, \R^2)$:
$$u(x) = \epsilon e^{w(x)} \Big(\cos \frac{v^1(x)}{\epsilon}, \sin \frac{v^1(x)}{\epsilon}, 
\cos \frac{v^2(x)}{\epsilon}, \sin \frac{v^2(x)}{\epsilon}\Big).$$
By a direct calculation one easily obtains:
$$(\nabla u)^T\nabla u = e^{2w} (\nabla v)^T\nabla v + 2\epsilon^2
e^{2w}\nabla w\otimes \nabla w.$$
Consequently, (\ref{II_in}) is equivalent to:
$$(\nabla v)^T\nabla v = \tilde g \qquad \mbox{ where }\; 
\tilde g\doteq e^{-2w}g - 2\epsilon^2\nabla
w\otimes \nabla w.$$
It now suffices to check that there exists $\epsilon>0$ and
$w\in\mathcal{C}^\infty(\bar\omega, \R)$, such that the derived
matrix field $\tilde g\in\mathcal{C}^\infty(\bar\omega, \R^{2\times 2}_\sym)$ is
itself a Riemannian metric with zero Gaussian curvature:
\begin{equation}\label{req_poz}
\tilde g>0 \quad \mbox{ and } \quad \kappa (\tilde g) = 0 \qquad\mbox{in }\; \bar\omega.
\end{equation}
For $\epsilon =0$, existence of a smooth $w_0$ such that
$e^{-2w_0}g$ satisfies the second condition (since the first one holds
trivially) in (\ref{req_poz}), is a consequence of the conformal equivalence
of $g$ with the Euclidean metric. In particular, one can take $w_0$ to
be the solution of the Dirichlet problem:
$$\Delta_gw_0 = -\kappa(g) \quad\mbox{ in }\; \omega, \quad \qquad w_0 = 0
\quad\mbox{ on }\; \partial\omega.$$
By applying the implicit function theorem, and in view of the
formula for $\kappa(\tilde g)$ in terms of $\kappa(g), \epsilon, w$,
there follows in fact the existence of a one-parameter family $\epsilon\mapsto w(\epsilon)$
resulting in the validity of (\ref{req_poz}) for all $\epsilon>0$ sufficiently small.
This achieves the result in Theorem \ref{thm_poz}.\endproof

\medskip

\noindent It would be interesting to see whether the main result of
this paper could be recovered by applying convex integration
directly to find $w$ that solves $\kappa(\tilde g)=0$.

\end{document}